\documentclass[11pt, reqno]{amsart}
\usepackage{amsfonts,latexsym,enumerate}
\usepackage{amsmath}
\usepackage{amscd} 
\usepackage{float,amsmath,amssymb,mathrsfs,bm,multirow,graphics}
\usepackage[dvips]{graphicx}
\usepackage[percent]{overpic}
\usepackage{amsaddr}
\usepackage[numbers,sort&compress]{natbib}
\usepackage{xcolor}
\usepackage{todonotes}

\newcommand{\R}{{\Bbb R}}
\newcommand{\N}{{\Bbb N}}
\newcommand{\C}{{\Bbb C}}
\newcommand{\D}{{\Bbb D}}

\newcommand{\re}{\text{\upshape Re\,}}

\allowdisplaybreaks

\newtheorem{theorem}{Theorem}[section]
\newtheorem{proposition}[theorem]{Proposition}
\newtheorem{lemma}[theorem]{Lemma}
\newtheorem{corollary}[theorem]{Corollary}
\newtheorem{definition}[theorem]{Definition}

\newtheorem{remark}[theorem]{Remark}

\newtheorem{RHproblem}[theorem]{RH problem}
\newtheorem{figuretext}{Figure}

\numberwithin{equation}{section}

\usepackage[colorlinks=true]{hyperref}
\hypersetup{urlcolor=blue, citecolor=red, linkcolor=blue}

\usepackage{tikz}
\usepackage{pict2e}
\usetikzlibrary{arrows,calc,chains, positioning, shapes.geometric,shapes.symbols,decorations.markings,arrows.meta}
\usetikzlibrary{patterns,angles,quotes}
\usetikzlibrary{decorations.markings}
\tikzset{middlearrow/.style={
			decoration={markings,
				mark= at position 0.6 with {\arrow{#1}} ,
			},
			postaction={decorate}
		}
	}
\tikzset{->-/.style={decoration={
				markings,
				mark=at position #1 with {\arrow{latex}}},postaction={decorate}}}
	
\tikzset{-<-/.style={decoration={
				markings,
				mark=at position #1 with {\arrowreversed{latex}}},postaction={decorate}}}
				
				\tikzset{
	master/.style={
		execute at end picture={
			\coordinate (lower right) at (current bounding box.south east);
			\coordinate (upper left) at (current bounding box.north west);
		}
	},
	slave/.style={
		execute at end picture={
			\pgfresetboundingbox
			\path (upper left) rectangle (lower right);
		}
	}
}
\tikzset{cross/.style={cross out, draw, 
         minimum size=2*(#1-\pgflinewidth), 
         inner sep=0pt, outer sep=0pt}}

\def\XXint#1#2#3{{\setbox0=\hbox{$#1{#2#3}{\int}$ }
\vcenter{\hbox{$#2#3$ }}\kern-.59\wd0}}

\title[B\MakeLowercase{low-up solutions of the} B\MakeLowercase{oussinesq equation}]
{\Large B\MakeLowercase{low-up solutions of the ``bad"} B\MakeLowercase{oussinesq equation}}

\author{C\MakeLowercase{hristophe} C\MakeLowercase{harlier}}

\address{Centre for Mathematical Sciences, Lund University, Sweden.}
\email{christophe.charlier@math.lu.se}

\begin{document}

\begin{abstract}
We study blow-up solutions of the ``bad" Boussinesq equation, and prove that a wide range of asymptotic scenarios can happen. For example, for each $T>0$, $x_{0}\in \R$ and $\delta \in (0,1)$, we prove that there exist Schwartz class solutions $u(x,t)$ on $\R \times [0,T)$ such that $|u(x,t)| \leq C \frac{1+x^{2}}{(x-x_{0})^{2}}$ and $u(x_{0},t)\asymp (T-t)^{-\delta}$ as $t\to T$.

We also prove that for any $q\in \N$, $T>0$, $x_{0}\in \R$, $\delta \in (0,\frac{1}{2})$, there exist Schwartz class solutions $u(x,t)$ on $\R \times [0,T)$ such that
\begin{itemize}
\item[(i)] $|\partial_{x}^{q_{1}}\partial_{t}^{q_{2}}u(x,t)|\leq C$ for each $q_{1},q_{2}\in \N$ such that $q_{1}+2q_{2}\leq q$,
\item[(ii)] $|\partial_{x}^{q_{1}}\partial_{t}^{q_{2}}u(x,t)| \leq C \frac{1+|x|}{|x-x_{0}|}$  for each $q_{1},q_{2}\in \N$ such that $q_{1}+2q_{2}= q+1$,
\item[(iii)] $|\partial_{x}^{q_{1}}\partial_{t}^{q_{2}}u(x_{0},t)| \asymp (T-t)^{-\delta}$ as $t\to T$ for each $q_{1},q_{2}\in \N$ such that $q_{1}+2q_{2}= q+1$.
\end{itemize}
In particular, when $q=0$, this result establishes the existence of wave-breaking solutions, i.e.\ solutions that remain bounded but whose $x$-derivative blows up in finite time. 
\end{abstract}

\maketitle



\setcounter{tocdepth}{1}

\section{Introduction}

We consider the initial value problem for the Boussinesq equation
\begin{align}
& u_{tt} = u_{xx} + (u^2)_{xx} + u_{xxxx}, \label{boussinesq} \\[0.2cm] 
& u(x,0) = u_{0}(x), \qquad u_{t}(x,0) = u_{1}(x), \qquad x \in \R, \nonumber 
\end{align}
where $u:\R \times [0,T)$ is a smooth real-valued function, 
$T>0$ is the maximal existence time of the solution, subscripts denote partial derivatives, and $u_{0},u_{1}$ are some given initial data. If $T$ is finite, then $u$ is said to be a blow-up solution. This paper is concerned with the existence and the asymptotic behavior as $t\to T$ of blow-up solutions of \eqref{boussinesq}. 

Equation \eqref{boussinesq} was first derived by Boussinesq \cite{B1872} in 1872 as a model for small-amplitude dispersive waves in shallow water. This equation also describes several other physical phenomena, such as the propagation of ion-sound waves in a uniform isotropic plasma \cite{M1978}, nonlinear lattice waves \cite{T1975}, and as a continuum approximation of the Fermi--Pasta--Ulam--Tsingou lattice dynamics \cite{M1978}. Moreover, unlike the KdV equation, \eqref{boussinesq} models waves propagating in both the positive and negative $x$-directions, see e.g. \cite[Section 3.2.5]{J1997}.



The initial value problem for the linear equation obtained by removing the $(u^2)_{xx}$-term in \eqref{boussinesq}, namely $u_{tt} = u_{xx} + u_{xxxx}$, is ill-posed. Because of this feature, \eqref{boussinesq} is often referred to as the ``bad" Boussinesq equation. 
In the 1970's, Hirota \cite{H1973} constructed multi-solitons solutions of \eqref{boussinesq}, and Zakharov \cite{Z1974} obtained a Lax pair for \eqref{boussinesq}. These works led to the understanding that \eqref{boussinesq} is integrable. 
The long-time behavior of global solutions to the initial value problem for \eqref{boussinesq} was recently obtained in \cite{CLmain, CL V, CL IV, CL I, CLsolitonResolution}. However, despite the long history of the Boussinesq equation, only few results on blow-up solutions are available in the literature. The instability of the solitons under small perturbations are discussed in \cite{B1976, BZ2002}. The works \cite{KL1977, LS1985, Y2002, YW2003} consider two initial-boundary value problems for \eqref{boussinesq} on a finite $x$-interval, and provide sufficient conditions on the initial-boundary data for the nonexistence of global solutions. Quoting \cite{Y2002}, ``A part of the reason for the paucity of the results is due to the properties of the linear part of \eqref{boussinesq} that are so `bad' that the traditional mathematical methods cease to be effective". 

The direct and inverse scattering problems for the initial value problem for \eqref{boussinesq} were recently studied in \cite{CLmain}. In particular, it is proved in \cite{CLmain} that for each $T > 0$, there exist Schwartz class solutions on $\R\times [0,T)$ that blow up exactly at time $T$. Here, a Schwartz class solution of \eqref{boussinesq} is defined as follows. 

\begin{definition}\label{def:schwartz class solution}\upshape
Let $T \in (0, \infty]$. We say that $u(x,t)$ is a {\it Schwartz class solution of \eqref{boussinesq} on $\R \times [0,T)$} if
\begin{enumerate}[$(i)$] 
  \item $u$ is a smooth real-valued solution of \eqref{boussinesq} for $(x,t) \in \R \times [0,T)$.

  \item $u$ has rapid decay as $|x| \to \infty$: for each integer $N \geq 1$ and each $\tau \in [0,T)$,
$$\sup_{\substack{x \in \R \\ t \in [0, \tau]}} \sum_{i =0}^N (1+|x|)^N |\partial_x^i u(x,t)|  < \infty.$$
\end{enumerate} 
\end{definition}


Blow-up solutions are better understood for other water wave equations, such as (generalizations of) the Camassa-Holm equation  \cite{CE1998A, CE1998, CE2000, WZ2006, GY2010, MY2023, WY2023, YC2024}, the Degasperis--Procesi equation \cite{Z2004, LY2006}, and the Novikov equation \cite{JN2012, YLZ2012}. The aforementioned works represent only a fraction of the extensive literature on blow-up results. A blow-up solution is said to be a \textit{wave-breaking solution} if $u_{x}$ becomes unbounded as $t\to T$ while $u$ remains bounded. Wave-breaking solutions have attracted a lot of attention over the last 25 years, partly due to the works \cite{CE1998A, CE1998, CE2000}. For example, for the $1$-periodic Camassa-Holm equation, it is proved in \cite{CE2000} that all classical blow-up solutions are wave-breaking solutions, and that the blow-up rate is independent of the initial data and given by
\begin{align*}
\lim_{t\to T} \Big( (T-t) \min_{x\in [0,1)}\{u_{x}(x,t)\} \Big)= -2.
\end{align*}
Moreover, for a large class of odd initial data, it is also proved in \cite{CE2000} that the blow-up set within the interval $[0,1)$ consists of only the two points $\{0,\frac{1}{2}\}$.

 In this work, we prove that a much broader range of blow-up scenarios can happen for the Boussinesq equation. Theorem \ref{thm:main} below proves the existence of unbounded blow-up solutions of \eqref{boussinesq}, while Theorem \ref{thm:wave breaking generalized} shows that for each $q\in \N$, there exist blow-up solutions whose derivatives $\partial_{x}^{q_{1}}\partial_{t}^{q_{2}}u$ are unbounded as $t\to T$ if $q_{1}+2q_{2}=q+1$ but remain bounded if $q_{1}+2q_{2}\leq q$. In particular, for $q=0$, Theorem \ref{thm:wave breaking generalized} establishes the existence of wave-breaking solutions of \eqref{boussinesq}, but this is only one special case among the various possible blow-up scenarios. Theorems \ref{thm:main} and \ref{thm:wave breaking generalized} also show that many different blow-up rates can occur, and that the blow-up set can consist of any singleton.  



\section{Main results}
We first introduce some notation.

For $r\in \N_{>0}:=\{1,2,\ldots\}$, let $\log_{r}$ be the $r$-th iterate of the logarithm, i.e.
\begin{align*}
\log_{1}s = \log s, \qquad \log_{2}s = \log \log s, \qquad \log_{3}s = \log \log \log s, \qquad \mbox{etc}.
\end{align*}
Let $\R_{+}:=[0,+\infty)$. Given parameters $p\in \N$, $\sigma\in \{0,1\}$,
\begin{align*}
& \vec{r} = (r_{1},\ldots,r_{p}) \in \N_{\mathrm{ord}}^{p}, \qquad \mbox{ and } \qquad \vec{a} = (a_{1},\ldots,a_{p}) \in \R_{+}^{p}, 
\end{align*}
where $\N_{\mathrm{ord}}^{p} := \{\vec{r} = (r_{1},\ldots,r_{p})\in (\N_{>0})^{p}: r_{1}<r_{2}<\ldots<r_{p}\}$, we define 
\begin{align}\label{def of LOG}
\mathrm{LOG}(s) = \mathrm{LOG}(s;\vec{r},\vec{a},\sigma) := (-1)^{\sigma}\prod_{j=1}^{p} (\log_{r_{j}} s)^{a_{j}},
\end{align}
with the convention that $\mathrm{LOG}(s) \equiv (-1)^{\sigma}$ if $p=0$.

Let $T>0$. If $f(t)$ and $g(t)$ are two real-valued functions defined for all $t$ such that $T-t>0$ is small enough, then we write $f(t) \asymp  g(t)$ as $t\to T$ if there exists $T_{0} \in (0,T)$ and $c_{1},c_{2} > 0$ such that
\begin{align*}
c_{1}g(t) \leq f(t) \leq c_{2}g(t), \qquad \mbox{for all } t\in [T_{0},T).
\end{align*}


Our first main result proves the existence of unbounded blow-up solutions of \eqref{boussinesq}. 

\begin{theorem}\label{thm:main}
For each $p\in \N$, $\vec{r} \in \N_{\mathrm{ord}}^{p}$, $\vec{a} \in \R_{+}^{p}$, $\sigma\in \{0,1\}$, $T>0$, $x_{0}\in \R$ and $\delta\in (0,1)$, there exist Schwartz class solutions $u:\R \times [0,T)$ of \eqref{boussinesq} such that
\begin{align}
& u(x_{0},t) \asymp \frac{\mathrm{LOG}(\frac{1}{T-t};\vec{r},\vec{a},\sigma)}{(T-t)^{\delta}}, \qquad \mbox{as } t \to T, \label{asymp for u in thm} \\
& |u(x,t)| \leq C \frac{1+x^{2}}{(x-x_{0})^{2}}, \qquad x\in \R\setminus \{x_{0}\}, \; t\in [0,T), \label{bound for u in thm}
\end{align}
for some constant $C>0$. If $\delta \in (0,\frac{1}{2})$, then \eqref{bound for u in thm} can be replaced by
\begin{align}\label{bound for u in thm bis}
|u(x,t)| \leq C \frac{1+|x|}{|x-x_{0}|}, \qquad x\in \R\setminus \{x_{0}\}, \; t\in [0,T).
\end{align}
\end{theorem}
\begin{remark}
\eqref{asymp for u in thm}--\eqref{bound for u in thm bis} imply that as $t$ approaches $T$, the solutions $u(x,t)$ of Theorem \ref{thm:main} blow up near $x_{0}$ while elsewhere they remain bounded. In other words, the blow-up set of $u$ is the singleton $\{(x_{0},T)\}$.
\end{remark}

Let $q \in \N$. Our next result establishes an analog of Theorem \ref{thm:main} for solutions $u$ whose derivatives $\frac{\partial^{q_{1}}}{(\partial x)^{q_{1}}}\frac{\partial^{q_{2}}}{(\partial t)^{q_{2}}} u$ remain bounded if $q_{1}+2q_{2}\leq q$ but become unbounded in finite time if $q_{1}+2q_{2} = q+1$.

\begin{theorem}\label{thm:wave breaking generalized}
For each $q,p\in \N$, $\vec{r} \in \N_{\mathrm{ord}}^{p}$, $\vec{a} \in \R_{+}^{p}$, $\sigma\in \{0,1\}$, $T>0$, $x_{0}\in \R$ and $\delta\in (0,\frac{1}{2})$, there exist Schwartz class solutions $u:\R \times [0,T)$ of \eqref{boussinesq} such that 
\begin{itemize}
\item[(i)] for each $q_{1},q_{2}\in \N$ with $q_{1}+2q_{2} \leq q$, $\sup_{(x,t)\in \R \times [0,T)} \big| \frac{\partial^{q_{1}}}{(\partial x)^{q_{1}}}\frac{\partial^{q_{2}}}{(\partial t)^{q_{2}}} u(x,t) \big|<\infty$ and
\begin{align*}
\frac{\partial^{q_{1}}}{(\partial x)^{q_{1}}}\frac{\partial^{q_{2}}}{(\partial t)^{q_{2}}} u(x,t) \mbox{ can be continuously extended to } \R \times [0,T],
\end{align*}
\item[(ii)] and for each $q_{1},q_{2}\in \N$ with $q_{1}+2q_{2} = q+1$,
\begin{align}
& (-1)^{\lceil \frac{q_{1}}{2} \rceil}\frac{\partial^{q_{1}}}{(\partial x)^{q_{1}}}\frac{\partial^{q_{2}}}{(\partial t)^{q_{2}}} u(x_{0},t) \asymp \frac{\mathrm{LOG}(\frac{1}{T-t};\vec{r},\vec{a},\sigma)}{(T-t)^{\delta}}, \qquad \mbox{as } t \to T, \label{asymp for u in thm gene} \\
& \bigg| \frac{\partial^{q_{1}}}{(\partial x)^{q_{1}}}\frac{\partial^{q_{2}}}{(\partial t)^{q_{2}}} u(x,t) \bigg| \leq C \frac{1+|x|}{|x-x_{0}|}, \qquad x\in \R\setminus \{x_{0}\}, \; t\in [0,T), \label{bound for der der u in tlm gene}
\end{align}
for some constant $C>0$, where $\lceil x \rceil$ denotes the smallest integer greater than or equal to $x$.
\end{itemize}
\end{theorem}
The case $q=0$ of Theorem \ref{thm:wave breaking generalized} proves the existence of wave-breaking solutions and is worth being stated separately.
\begin{corollary}\label{coro:wave breaking}\emph{(Existence of wave-breaking solutions)}

\noindent For each $p\in \N$, $\sigma\in \{0,1\}$, $\vec{r} \in \N_{\mathrm{ord}}^{p}$, $\vec{a} \in \R_{+}^{p}$, $T>0$, $x_{0}\in \R$ and $\delta\in (0,\frac{1}{2})$, there exist Schwartz class solutions $u:\R \times [0,T)$ of \eqref{boussinesq} such that 
\begin{align*}
& \sup_{(x,t)\in \R \times [0,T)} \big|  u(x,t) \big|<\infty \mbox{ and } u(x,t) \mbox{ can be continuously extended to } \R \times [0,T], \\
& \frac{\partial}{\partial x} u(x_{0},t) \asymp \frac{\mathrm{LOG}(\frac{1}{T-t};\vec{r},\vec{a},\sigma)}{(T-t)^{\delta}}, \qquad \mbox{as } t \to T, \\
& \bigg| \frac{\partial}{\partial x} u(x,t) \bigg| \leq C \frac{1+|x|}{|x-x_{0}|}, \hspace{1.9cm} x\in \R\setminus \{x_{0}\}, \; t\in [0,T),
\end{align*}
for some constant $C>0$.
\end{corollary}



\subsection*{Conclusion and open problems.}
 In this paper, we proved some existence and asymptotic results for blow-up solutions of the bad Boussinesq equation. Further properties of blow-up solutions of \eqref{boussinesq} are worth being investigated. For instance:
\begin{enumerate}
\item \medskip By \eqref{asymp for u in thm}--\eqref{bound for u in thm}, the solutions $u$ of Theorem \ref{thm:main} are unbounded only near one point (namely $(x_{0},T)$). This brings the following question:

\medskip \noindent \textit{Given $T>0$, do there exist solutions $u(x,t)$ on $\R \times [0,T)$ such that $u(x,t)$ is unbounded near more than just one point?}

\medskip \noindent Similarly, in Theorem \ref{thm:wave breaking generalized}, by \eqref{asymp for u in thm gene}--\eqref{bound for der der u in tlm gene} the higher order derivatives of $u$ are unbounded only near $(x_{0},T)$. Therefore we ask the following:

\medskip \noindent \textit{Given $T>0$ and $q\in \N$, do there exist solutions $u(x,t)$ on $\R \times [0,T)$ such that
\begin{itemize}
\item[(i)] for each $q_{1},q_{2}\in \N$ with $q_{1}+2q_{2} \leq q$, $\sup_{(x,t)\in \R \times [0,T)} \big| \frac{\partial^{q_{1}}}{(\partial x)^{q_{1}}}\frac{\partial^{q_{2}}}{(\partial t)^{q_{2}}} u(x,t) \big|<\infty$,
\item[(ii)] for at least one pair $(q_{1},q_{2})\in \N^{2}$ with $q_{1}+2q_{2} = q+1$, $\frac{\partial^{q_{1}}}{(\partial x)^{q_{1}}}\frac{\partial^{q_{2}}}{(\partial t)^{q_{2}}} u(x,t)$ is unbounded near more than just one point?
\end{itemize}}
\item \medskip It would be interesting to explore whether different types of blow-up rates than the ones stated in \eqref{asymp for u in thm} and \eqref{asymp for u in thm gene} can occur. In particular, one can ask whether the ranges $\delta \in (0,1)$ in Theorem \ref{thm:main} and $\delta \in (0,\frac{1}{2})$ in Theorem \ref{thm:wave breaking generalized} are optimal:

\medskip \noindent 
\textit{Given $T>0$ and $x_{0}\in \R$, do there exist solutions $u:\R\times [0,T)$ such that $|u(x_{0},t)|\asymp (T-t)^{-\delta}$ as $t\to T$ for some $\delta\geq 1$?}

\medskip \noindent and similarly

\medskip \noindent 
\textit{Given $T>0$ and $q\in \N$, do there exist solutions $u:\R\times [0,T)$ such that 
\begin{itemize}
\item[(i)] for each $q_{1},q_{2}\in \N$ with $q_{1}+2q_{2} \leq q$, $\sup_{(x,t)\in \R \times [0,T)} \big| \frac{\partial^{q_{1}}}{(\partial x)^{q_{1}}}\frac{\partial^{q_{2}}}{(\partial t)^{q_{2}}} u(x,t) \big|<\infty$,
\item[(ii)] for at least one pair $(q_{1},q_{2})\in \N^{2}$ with $q_{1}+2q_{2} = q+1$, $|\frac{\partial^{q_{1}}}{(\partial x)^{q_{1}}}\frac{\partial^{q_{2}}}{(\partial t)^{q_{2}}} u(x,t)| \asymp (T-t)^{-\delta}$ for some $\delta \geq \frac{1}{2}$?
\end{itemize}
}
\medskip \noindent We believe that the answer to the first question is ``yes", and that the answer to the second question is ``no" if $\delta > \frac{1}{2}$, see Subsection \ref{subsection:strategy} for a heuristic discussion.
\item \medskip The solutions $u$ that we construct to prove Theorem \ref{thm:main} correspond to special choices of initial data $\{u_{0},u_{1}\}$. Hence we ask the following:

\medskip \noindent \textit{Let $\{u_{0},u_{1}\}$ be generic real-valued initial data in the Schwartz class (i.e.\ smooth with fast decay at $x=\pm\infty$). The maximal existence time $T>0$ of the solution $u$ to the initial value problem for \eqref{boussinesq} is explicitly given in terms of $\{u_{0},u_{1}\}$ in \cite[Theorem 2.8]{CLmain}. If $T$ is finite, what is the behavior of $u(x,t)$ as $t\to T$?}

%
\end{enumerate}

\subsection*{Outline.} The proofs of Theorems \ref{thm:main} and \ref{thm:wave breaking generalized} are based on a modified version of the inverse scattering transform introduced in \cite{CLmain}.

The result from \cite{CLmain} that is relevant for us can roughly be summarized as follows (for more details, see Section \ref{section:inverse scattering}). Given some scattering data, denoted $r_{1}$ and $r_{2}$, consider the Riemann-Hilbert (RH) problem \ref{RHn}. The solution to this RH problem exists, is unique, is $1\times 3$ vector-valued, and is denoted by
\begin{align*}
n(x,t,k)=\big( n_{1}(x,t,k),n_{2}(x,t,k),n_{3}(x,t,k) \big), \qquad x\in \R, \; t\in [0,T), \; k\in \C\setminus \Gamma,
\end{align*}
where the contour $\Gamma$ is shown in Figure \ref{fig: Dn}. Then $u(x,t)$ defined by
\begin{align*}
u(x,t) = -i\sqrt{3}\frac{\partial}{\partial x}\bigg(  \lim_{k\to \infty} k (n_{3}(x,t,k) -1) \bigg)
\end{align*}
is a Schwartz class solution of (\ref{boussinesq}) on $\R \times [0,T)$, where $T>0$ is the maximal existence time of the solution and is explicitly given in terms of $r_{1}$ via \eqref{Tdef}. 

For generic scattering data, it seems particularly challenging to perform an asymptotic analysis of $n(x,t,k)$ as $t\to T$. (By generic scattering data, we mean functions $r_{1},r_{2}$ satisfying properties $(\ref{Theorem2.3itemi}$)--$(\ref{Theorem2.3itemv})$ below.) 

In Section \ref{section:inverse scattering}, we modify the above inverse scattering result from \cite{CLmain} to handle particular (non-generic) scattering data satisfying \eqref{assumption r1=r2=0} below. The advantage of considering these particular scattering data is that, as we show in Section \ref{section:RH analysis}, quite remarkably $n(x,t,k)$ admits an explicit series representation. Sections \ref{section:proof1} and \ref{section:proof2} are devoted to completing the proofs of Theorems \ref{thm:main} and \ref{thm:wave breaking generalized}, respectively. This is achieved through a detailed analysis of the series representation for $n(x,t,k)$ for several choices of $r_{1}$.

\section{Inverse scattering}\label{section:inverse scattering}

In Subsection \ref{subsection:background}, we briefly recall the inverse scattering transform from \cite{CLmain}. In Subsection \ref{subsection:strategy}, we provide a heuristic discussion as to why the range $\delta \in (0,\frac{1}{2})$ in Theorem \ref{thm:wave breaking generalized} is believed to be optimal while the range $\delta \in (0,1)$ in Theorem \ref{thm:main} is not believed to be optimal. In Subsection \ref{subsection:modified inverse scattering}, we present an analog of the inverse scattering transform from \cite{CLmain}. This ``modified" inverse scattering transform allows to treat scattering data satisfying \eqref{assumption r1=r2=0} and will be used in Sections \ref{section:RH analysis}, \ref{section:proof1} and \ref{section:proof2} to prove Theorems \ref{thm:main} and \ref{thm:wave breaking generalized}. 

\subsection{Background from \cite{CLmain}}\label{subsection:background}

The inverse scattering transform from \cite{CLmain} is formulated in terms of a $1\times 3$ RH problem. The solution to this RH problem is denoted $n$ and is analytic in $\C \setminus \Gamma$, where $\Gamma := \cup_{j=1}^6 \Gamma_j \cup \partial \D \cup \{0\}$ is the contour shown and oriented as in Figure \ref{fig: Dn}, $\partial \D$ is the unit circle, and
\begin{align*}
& \Gamma_{j} = e^{\frac{\pi i(j-1)}{3}}\big( (0,i)\cup (-i,-i\infty) \big), & &  j=1,\ldots,6.
\end{align*}
Here and below, given $a,b\in  \C\cup \{e^{i\theta}\infty:\theta\in [0,2\pi)\}$, we write $(a,b)$ for the open segment oriented from $a$ to $b$. The contour $\Gamma$ can also be rewritten as $\Gamma = \cup_{j=1}^9 \Gamma_j \cup \Gamma_{\star}$, where 
\begin{align*}
& \Gamma_{j} = e^{\frac{\pi i(j-1)}{3}}\{e^{i\theta}:\theta \in (-\tfrac{\pi}{2},-\tfrac{\pi}{6}) \cup (\tfrac{\pi}{2},\tfrac{5\pi}{6})\}, & & j=7,8,9,
\end{align*}
and where $\Gamma_{\star}$ denotes the set of intersection points of $\Gamma$, namely 
\begin{align*}
\Gamma_{\star} := \{i\kappa_j\}_{j=1}^6 \cup \{0\}, \qquad \kappa_{j} := e^{\frac{\pi i(j-1)}{3}}, \qquad j=1,\ldots,6.
\end{align*}
Note that $\partial \mathbb{D}=\overline{\Gamma_{7}}\cup \overline{\Gamma_{8}} \cup \overline{\Gamma_{9}}$.
We now state the RH problem for $n$.
\begin{RHproblem}[RH problem for $n$ from \cite{CLmain}]\label{RHn}
Find a $1 \times 3$-row-vector valued function $n(x,t,k)$ with the following properties:
\begin{enumerate}[$(a)$]
\item\label{RHnitema} $n(x,t,\cdot) : \C \setminus \Gamma \to \mathbb{C}^{1 \times 3}$ is analytic.

\item\label{RHnitemb} The limits of $n(x,t,k)$ as $k$ approaches $\Gamma \setminus \Gamma_\star$ from the left and right exist, are continuous on $\Gamma \setminus \Gamma_\star$, and are denoted by $n_+$ and $n_-$, respectively. Furthermore, 
\begin{align}\label{njump}
  n_+(x,t,k) = n_-(x, t, k) v(x, t, k) \qquad \text{for} \quad k \in \Gamma \setminus \Gamma_\star,
\end{align}
where $v(x,t,\cdot):\cup_{j=1}^9 \Gamma_j\to \C^{3\times 3}$ is defined in \eqref{vdef} below.
\item\label{RHnitemc} $n(x,t,k) = O(1)$ as $k \to k_{\star} \in \Gamma_\star$.

\item\label{RHnitemd} For $k \in \C \setminus \Gamma$, $n$ obeys the symmetries
\begin{align}\label{nsymm}
n(x,t,k) = n(x,t,\omega k)\mathcal{A}^{-1} = n(x,t,k^{-1}) \mathcal{B},
\end{align}
where $\omega := \smash{e^{\frac{2\pi i}{3}}}$,
\begin{align}\label{def of Acal and Bcal}
\mathcal{A} := \begin{pmatrix}
0 & 0 & 1 \\
1 & 0 & 0 \\
0 & 1 & 0
\end{pmatrix} \qquad \mbox{ and } \qquad \mathcal{B} := \begin{pmatrix}
0 & 1 & 0 \\
1 & 0 & 0 \\
0 & 0 & 1
\end{pmatrix}.
\end{align}
\item\label{RHniteme} $n(x,t,k) = (1,1,1) + O(k^{-1})$ as $k \to \infty$.
\end{enumerate}
\end{RHproblem}
We now give the definition of the jump matrix $v$ appearing in condition (b) of RH problem \ref{RHn}.  Define $\{l_j(k), z_j(k)\}_{j=1}^3$ by
\begin{align}\label{lmexpressions intro}
& l_{j}(k) = i \frac{\omega^{j}k + (\omega^{j}k)^{-1}}{2\sqrt{3}}, \qquad z_{j}(k) = i \frac{(\omega^{j}k)^{2} + (\omega^{j}k)^{-2}}{4\sqrt{3}}, \qquad k \in \C\setminus \{0\},
\end{align}
and let $\{\theta_{\ell j}(x,t, k)\}_{1 \leq j<\ell  \leq 3}$ be given by
\begin{align}\label{def of Phi ij}
\theta_{\ell j}(x,t,k) = (l_{\ell}(k)-l_{j}(k))x + (z_{\ell }(k)-z_{j}(k))t.
\end{align}
Define $\hat{\Gamma}_{j} := \Gamma_{j} \cup \partial \D$ for $j=1,\ldots,6$, and let $r_1:\hat{\Gamma}_{1} \to \C$ and $r_2: \hat{\Gamma}_{4}\setminus \{\omega^{2},-\omega^{2}\} \to \C$ be functions satisfying the following properties:
\begin{enumerate}[$(i)$]
 \item \label{Theorem2.3itemi}
$r_1 \in C^\infty(\hat{\Gamma}_{1})$ and $r_2 \in C^\infty(\hat{\Gamma}_{4}\setminus \{\omega^{2}, -\omega^{2}\})$.

\item \label{Theorem2.3itemii}
$r_{1}(k)$ is bounded on $\partial \D$, and $r_{1}(\kappa_{j})\neq 0$ for $j=1,\ldots,6$. $r_{2}(k)$ has simple poles at $k=\omega^2$ and $k = -\omega^2$, and simple zeros at $k=\omega$ and $k=-\omega$. Furthermore,
\begin{align}\label{r1r2at0}
r_{1}(1) = r_{1}(-1) = 1, \qquad r_{2}(1) = r_{2}(-1) = -1.
\end{align}

\item \label{Theorem2.3itemiii}
$r_1(k)$ and $r_2(k)$ are rapidly decreasing as $|k| \to \infty$: for each $N \in \N$,
\begin{align}\label{r1r2rapiddecay}
& \max_{j=0,1,\dots,N}\sup_{k \in \Gamma_{1}} (1+|k|)^N |\partial_k^jr_1(k)| < \infty,  
\qquad
 \max_{j=0,1,\dots,N} \sup_{k \in \Gamma_{4}} (1+|k|)^N|\partial_k^jr_2(k)| < \infty.
\end{align}

\item \label{Theorem2.3itemiv}
For all $k \in \partial \D \setminus \{-\omega,\omega\}$, we have
\begin{align}\label{r1r2 relation on the unit circle}
r_{1}(\tfrac{1}{\omega k}) + r_{2}(\omega k) + r_{1}(\omega^{2} k) r_{2}(\tfrac{1}{k}) = 0.
\end{align}

\item \label{Theorem2.3itemv}
$r_1$ and $r_2$ are related by
\begin{align}\label{r1r2 relation with kbar symmetry}
& r_{2}(k) = \tilde{r}(k) \overline{r_{1}(\bar{k}^{-1})} & & \mbox{for all } k \in \hat{\Gamma}_{4}\setminus \{\omega^{2}, -\omega^{2}\}, 
\end{align}
where
\begin{align}\label{def of tilde r}
\tilde{r}(k):=\frac{\omega^{2}-k^{2}}{1-\omega^{2}k^{2}} & & \mbox{for } k \in \mathbb{C}\setminus \{\omega^{2},-\omega^{2}\}.
\end{align}


\end{enumerate} 
\begin{figure}
\begin{center}
\begin{tikzpicture}[scale=0.7]
\node at (0,0) {};
\draw[black,line width=0.45 mm,->-=0.4,->-=0.85] (0,0)--(30:4);
\draw[black,line width=0.45 mm,->-=0.4,->-=0.85] (0,0)--(90:4);
\draw[black,line width=0.45 mm,->-=0.4,->-=0.85] (0,0)--(150:4);
\draw[black,line width=0.45 mm,->-=0.4,->-=0.85] (0,0)--(-30:4);
\draw[black,line width=0.45 mm,->-=0.4,->-=0.85] (0,0)--(-90:4);
\draw[black,line width=0.45 mm,->-=0.4,->-=0.85] (0,0)--(-150:4);

\draw[black,line width=0.45 mm] ([shift=(-180:2.5cm)]0,0) arc (-180:180:2.5cm);
\draw[black,arrows={-Triangle[length=0.2cm,width=0.18cm]}]
($(3:2.5)$) --  ++(90:0.001);
\draw[black,arrows={-Triangle[length=0.2cm,width=0.18cm]}]
($(57:2.5)$) --  ++(-30:0.001);
\draw[black,arrows={-Triangle[length=0.2cm,width=0.18cm]}]
($(123:2.5)$) --  ++(210:0.001);
\draw[black,arrows={-Triangle[length=0.2cm,width=0.18cm]}]
($(177:2.5)$) --  ++(90:0.001);
\draw[black,arrows={-Triangle[length=0.2cm,width=0.18cm]}]
($(243:2.5)$) --  ++(330:0.001);
\draw[black,arrows={-Triangle[length=0.2cm,width=0.18cm]}]
($(297:2.5)$) --  ++(210:0.001);

\draw[black,line width=0.15 mm] ([shift=(-30:0.55cm)]0,0) arc (-30:30:0.55cm);

\node at (0.8,0) {$\tiny \frac{\pi}{3}$};

\node at (-1:2.9) {\footnotesize $\Gamma_8$};
\node at (60:2.9) {\footnotesize $\Gamma_9$};
\node at (120:2.9) {\footnotesize $\Gamma_7$};
\node at (181:2.9) {\footnotesize $\Gamma_8$};
\node at (240:2.83) {\footnotesize $\Gamma_9$};
\node at (300:2.83) {\footnotesize $\Gamma_7$};

\node at (105:1.45) {\footnotesize $\Gamma_1$};
\node at (138:1.45) {\footnotesize $\Gamma_2$};
\node at (223:1.45) {\footnotesize $\Gamma_3$};
\node at (-104:1.45) {\footnotesize $\Gamma_4$};
\node at (-42:1.45) {\footnotesize $\Gamma_5$};
\node at (43:1.45) {\footnotesize $\Gamma_6$};

\node at (97:3.3) {\footnotesize $\Gamma_4$};
\node at (144:3.3) {\footnotesize $\Gamma_5$};
\node at (217:3.3) {\footnotesize $\Gamma_6$};
\node at (-96:3.3) {\footnotesize $\Gamma_1$};
\node at (-35:3.3) {\footnotesize $\Gamma_2$};
\node at (36:3.3) {\footnotesize $\Gamma_3$};
\end{tikzpicture}
\end{center}
\begin{figuretext}\label{fig: Dn}
The contour $\Gamma = \cup_{j=1}^9 \Gamma_j \cup \Gamma_{\star}$ in the complex $k$-plane.
\end{figuretext}
\end{figure}
For $x\in \R$, $t \geq 0$, and $k \in \Gamma\setminus \Gamma_\star$, the jump matrix $v(x,t,k)$ appearing in \eqref{njump} is defined by
\begin{align}
&  v_1 =   \begin{pmatrix}  
 1 & - r_1(k)e^{-\theta_{21}} & 0 \\
  r_1(\frac{1}{k})e^{\theta_{21}} & 1 - r_1(k)r_1(\frac{1}{k}) & 0 \\
  0 & 0 & 1
  \end{pmatrix}\hspace{-0.1cm},
	\quad  v_2 = \begin{pmatrix}   
 1 & 0 & 0 \\
 0 & 1 - r_2(\omega k)r_2(\frac{1}{\omega k}) & -r_2(\frac{1}{\omega k})e^{-\theta_{32}} \\
 0 & r_2(\omega k)e^{\theta_{32}} & 1 
    \end{pmatrix}\hspace{-0.1cm},
   	\nonumber \\ \nonumber
  &v_3 = \begin{pmatrix} 
 1 - r_1(\omega^2 k)r_1(\frac{1}{\omega^2 k}) & 0 & r_1(\frac{1}{\omega^2 k})e^{-\theta_{31}} \\
 0 & 1 & 0 \\
 -r_1(\omega^2 k)e^{\theta_{31}} & 0 & 1  
  \end{pmatrix}\hspace{-0.1cm},
	\quad   v_4 = \begin{pmatrix}  
  1 - r_2(k)r_{2}(\frac{1}{k}) & -r_2(\frac{1}{k}) e^{-\theta_{21}} & 0 \\
  r_2(k)e^{\theta_{21}} & 1 & 0 \\
  0 & 0 & 1
   \end{pmatrix}\hspace{-0.1cm},
   	\\ \nonumber
&  v_5 = \begin{pmatrix}
  1 & 0 & 0 \\
  0 & 1 & -r_1(\omega k)e^{-\theta_{32}} \\
  0 & r_1(\frac{1}{\omega k})e^{\theta_{32}} & 1 - r_1(\omega k)r_1(\frac{1}{\omega k}) 
  \end{pmatrix}\hspace{-0.1cm},
	\;\;   v_6 = \begin{pmatrix} 
  1 & 0 & r_2(\omega^2 k)e^{-\theta_{31}} \\
  0 & 1 & 0 \\
  -r_2(\frac{1}{\omega^2 k})e^{\theta_{31}} & 0 & 1 - r_2(\omega^2 k)r_2(\frac{1}{\omega^2 k})
   \end{pmatrix}\hspace{-0.1cm}, \nonumber \\
& v_{7} = \begin{pmatrix}
1 & -r_{1}(k)e^{-\theta_{21}} & r_{2}(\omega^{2}k)e^{-\theta_{31}} \\
-r_{2}(k)e^{\theta_{21}} & 1+r_{1}(k)r_{2}(k) & \big(r_{2}(\frac{1}{\omega k})-r_{2}(k)r_{2}(\omega^{2}k)\big)e^{-\theta_{32}} \\
r_{1}(\omega^{2}k)e^{\theta_{31}} & \big(r_{1}(\frac{1}{\omega k})-r_{1}(k)r_{1}(\omega^{2}k)\big)e^{\theta_{32}} & f(\omega^{2}k)
\end{pmatrix}\hspace{-0.1cm}, \nonumber \\
& v_{8} = \begin{pmatrix}
f(k) & r_{1}(k)e^{-\theta_{21}} & \big(r_{1}(\frac{1}{\omega^{2} k})-r_{1}(k)r_{1}(\omega k)\big)e^{-\theta_{31}} \\
r_{2}(k)e^{\theta_{21}} & 1 & -r_{1}(\omega k) e^{-\theta_{32}} \\
\big( r_{2}(\frac{1}{\omega^{2}k})-r_{2}(\omega k)r_{2}(k) \big)e^{\theta_{31}} & -r_{2}(\omega k) e^{\theta_{32}} & 1+r_{1}(\omega k)r_{2}(\omega k)
\end{pmatrix}\hspace{-0.1cm}, \nonumber \\
& v_{9} = \begin{pmatrix}
1+r_{1}(\omega^{2}k)r_{2}(\omega^{2}k) & \big( r_{2}(\frac{1}{k})-r_{2}(\omega k)r_{2}(\omega^{2} k) \big)e^{-\theta_{21}} & -r_{2}(\omega^{2}k)e^{-\theta_{31}} \\
\big(r_{1}(\frac{1}{k})-r_{1}(\omega k) r_{1}(\omega^{2} k)\big)e^{\theta_{21}} & f(\omega k) & r_{1}(\omega k)e^{-\theta_{32}} \\
-r_{1}(\omega^{2}k)e^{\theta_{31}} & r_{2}(\omega k) e^{\theta_{32}} & 1
\end{pmatrix}\hspace{-0.1cm}, \label{vdef}
\end{align}
where the matrices $v_j$ are the restrictions of $v$ to $\Gamma_{j}$, $j=1,\ldots,9$, and 
\begin{align}\label{def of f}
f(k) := 1+r_{1}(k)r_{2}(k) + r_{1}(\tfrac{1}{\omega^{2}k})r_{2}(\tfrac{1}{\omega^{2}k}), \qquad k \in \partial \mathbb{D}.
\end{align}
We now state the following result from \cite{CLmain}.
\begin{theorem}\emph{(\cite[Theorem 2.6]{CLmain})}\label{inverseth}
Let $r_1:\hat{\Gamma}_{1} \to \C$ and $r_2: \hat{\Gamma}_{4}\setminus \{\omega^{2},-\omega^{2}\} \to \C$ be two functions satisfying properties $(\ref{Theorem2.3itemi})$--$(\ref{Theorem2.3itemv})$.
Define $T \in (0, \infty]$ by 
\begin{align}\label{Tdef}
T := \sup \big\{t \geq 0 \, | \, \text{$e^{\frac{|k|^2t}{4}}r_1(1/k)$ and its derivatives are rapidly decreasing as $\Gamma_1 \ni k \to \infty$}\big\}.
\end{align}
Then RH problem \ref{RHn} has a unique solution $n(x,t,k)$ for each $(x,t) \in \R \times [0,T)$ and the function $n_{3}^{(1)}$ defined by
$$n_{3}^{(1)}(x,t) := \lim_{k\to \infty} k (n_{3}(x,t,k) -1)$$ 
is well-defined and smooth for $(x,t) \in \R \times [0,T)$. Moreover, $u(x,t)$ defined by
\begin{align}\label{recoveruvn}
u(x,t) = -i\sqrt{3}\frac{\partial}{\partial x}n_{3}^{(1)}(x,t),
\end{align}
is a Schwartz class solution of (\ref{boussinesq}) on $\R \times [0,T)$.
\end{theorem}
Scattering data $r_1:\hat{\Gamma}_{1} \to \C$ and $r_2: \hat{\Gamma}_{4}\setminus \{\omega^{2},-\omega^{2}\} \to \C$ satisfying properties $(\ref{Theorem2.3itemi})$--$(\ref{Theorem2.3itemv})$ correspond to generic initial data \cite[Theorem 2.3]{CLmain}.
\subsection{Heuristic discussion}\label{subsection:strategy}
In order to analyze the asymptotic behavior of blow-up solutions of \eqref{boussinesq} with given existence time $T\in (0,\infty)$, one could try to perform a Deift--Zhou \cite{DZ1993} steepest descent analysis of the RH problem \ref{RHn} in the general case
\begin{align}\label{r1 in the form w}
r_{1}(\tfrac{1}{k}) = f_{1}(\tfrac{1}{k})e^{\frac{T}{4}k^{2}} = f_{1}(\tfrac{1}{k})e^{-\frac{T}{4}|k|^{2}} , \qquad k \in \hat{\Gamma}_{1},
\end{align}
where $f_{1}\in C^{\infty}(\hat{\Gamma}_{1})$ is such that $|f_{1}(\frac{1}{k})| \leq e^{|k|^{2-\epsilon}}$ for some $\epsilon >0$ and all large $|k|$. The Deift--Zhou method has been particularly successful to analyze the long-time behavior of global solutions of many integrable equations, see e.g. \cite{BJM2018, BKST2009, CF2022, CLmain, CL V, CL IV, CL I, CLsolitonResolution, CLWasymptotics, DZ1993, GLL2024, GM2020, GT2009, HWZ2024, HXF2015, LGWW2019, RS2019, XF2020, YF2023, ZW2024}, but to our knowledge it has not yet been used to analyze the behavior of blow-up solutions. In our case, performing an asymptotic analysis on the RH problem \ref{RHn} seems challenging for general scattering data of the form \eqref{r1 in the form w}.

The fact that the solution $n$ to RH problem \ref{RHn} ceases to exist after time $T$ is due to certain ``bad" entries of $v$ on $\cup_{j=1}^{6}\Gamma_{j}$ which are unbounded if $t>T$, such as the $21$-entry of $v_{1}$ on $\Gamma_{1}\cap \{k:|k|>1\}$ and the $12$-entry of $v_{1}$ on $\Gamma_{1}\cap \{k:|k|<1\}$. Indeed, by substituting \eqref{r1 in the form w} and $\theta_{21}(x,t,k) = \frac{x}{2}\big( k-\frac{1}{k} \big) + \frac{t}{4}\big( \frac{1}{k^{2}}-k^{2} \big)$ in the definition \eqref{vdef} of $v_{1}$, we get
\begin{align*}
v_1(x,t,k) =   \begin{pmatrix}  
 1 & - f_{1}(k)e^{\frac{T-t}{4k^{2}}}e^{\frac{x}{2}( \frac{1}{k}-k ) + \frac{t}{4} k^{2}} & 0 \\
  f_{1}(\tfrac{1}{k})e^{\frac{T-t}{4} k^{2}}e^{\frac{x}{2}( k-\frac{1}{k} ) + \frac{t}{4} \frac{1}{k^{2}}} & 1 - f_1(k)f_1(\frac{1}{k})e^{\frac{T}{4}(k^{2}+\frac{1}{k^{2}})} & 0 \\
  0 & 0 & 1
  \end{pmatrix}, \quad k \in \Gamma_{1}.
\end{align*}
For any fixed $x\in \R$ and $t >T$, $(v_{1}(x,t,k))_{21}$ becomes unbounded as $\Gamma_{1} \ni k\to \infty$, which contradicts the asymptotic behavior of $n$ stated in property ($\ref{RHniteme}$) of RH problem \ref{RHn}.

Therefore, one expects that the leading order behavior of $n$ as $t\to T$ should come from these ``bad" entries. To ease the following discussion, it convenient to consider for a moment a simpler RH problem, whose solution is denoted $\tilde{n}$, and which is obtained from RH problem \ref{RHn} by ignoring the symmetries \eqref{nsymm} and by replacing the jump matrix $v$ by $\tilde{v}$, where $\tilde{v}$ is the identity matrix except on $\Gamma_{3}\cap \{z: |z|> 1\}=(e^{\frac{\pi i}{6}},e^{\frac{\pi i}{6}}\infty)$ where it is given by
\begin{align*}
\tilde{v}(x,t,k) =  \begin{pmatrix} 
 1 & 0 & f_{1}(\tfrac{1}{\omega^{2}k}) e^{-\frac{x}{2}\frac{1}{\omega^{2}k}} e^{\frac{t}{4} \frac{1}{(\omega^{2}k)^{2}}} e^{\frac{x}{2} \omega^{2}k} e^{\frac{T-t}{4}(\omega^{2}k)^{2}}  \\
 0 & 1 & 0 \\
 0 & 0 & 1  
  \end{pmatrix}, \quad k\in (e^{\frac{\pi i}{6}},e^{\frac{\pi i}{6}}\infty).
\end{align*}
In other words, all diagonal entries of $v$ have been replaced by $1$, and all off-diagonal entries of $v$ have been set to $0$, except for the ``bad" $13$-entry of $v$ on $\Gamma_{3}\cap \{z: |z|> 1\}$ which remains unchanged.

More precisely, we consider the following RH problem.
\begin{RHproblem}[RH problem for $\tilde{n}$]\label{RHnt}
Find a $1 \times 3$-row-vector valued function $\tilde{n}(x,t,k)$ with the following properties:
\begin{enumerate}[$(a)$]
\item $\tilde{n}(x,t,\cdot) : \C \setminus (e^{\frac{\pi i}{6}},e^{\frac{\pi i}{6}}\infty) \to \mathbb{C}^{1 \times 3}$ is analytic.

\item The limits of $\tilde{n}(x,t,k)$ as $k$ approaches $(e^{\frac{\pi i}{6}},e^{\frac{\pi i}{6}}\infty)$ from the left and right exist, are continuous on $(e^{\frac{\pi i}{6}},e^{\frac{\pi i}{6}}\infty)$, and are denoted by $\tilde{n}_+$ and $\tilde{n}_-$, respectively. Furthermore, 
\begin{align}\label{jumps of nt}
\tilde{n}_+(x,t,k) = \tilde{n}_-(x, t, k) \tilde{v}(x, t, k) \qquad \text{for} \quad k \in (e^{\frac{\pi i}{6}},e^{\frac{\pi i}{6}}\infty).
\end{align}
\item\label{itemc RHnt} $\tilde{n}(x,t,k) = O(1)$ as $k \to e^{\frac{\pi i}{6}}$.
\item $\tilde{n}(x,t,k) = (1,1,1) + O(k^{-1})$ as $k \to \infty$.
\end{enumerate}
\end{RHproblem}
Let $\tilde{u}(x,t) := -i\sqrt{3}\frac{\partial}{\partial x}\tilde{n}_{3}^{(1)}(x,t)$, where $\tilde{n}_{3}^{(1)}(x,t) := \lim_{k\to \infty} k (\tilde{n}_{3}(x,t,k) -1)$. The solution $\tilde{n}(x,t,k)$ of RH problem \ref{RHnt} is too simple to produce a solution to the bad Boussinesq equation, but since $\tilde{v}$ contains one ``bad" entry, heuristically we expect $\tilde{u}(x,t)$ to be of the same order as $u(x,t)$ (defined in \eqref{recoveruvn} with $r_{1},r_{2}$ given by \eqref{r1 in the form w} and \eqref{r1r2 relation with kbar symmetry}) when $t\to T$.

Let us now compute $\tilde{u}(x,t)$ explicitly. By \eqref{jumps of nt}, the first two entries of $\tilde{n}$ are continuous through $(e^{\frac{\pi i}{6}},e^{\frac{\pi i}{6}}\infty)$, and by $(\ref{itemc RHnt})$ they are bounded near $e^{\frac{\pi i}{6}}$; hence by Morera's Theorem and Riemann's theorem on removable singularities these entries are entire. Since each of these two entries behaves as $1+O(k^{-1})$ as $k\to\infty$, by Liouville's theorem they must be identically equal to $1$. Then the third entry of $\tilde{n}$ is given by the Sokhotski–Plemelj formula. Thus for $x\in \R$, $t\in [0,T)$ and $k\in \C\setminus (e^{\frac{\pi i}{6}},e^{\frac{\pi i}{6}}\infty)$, we have
\begin{align*}
\tilde{n}(x,t,k) = \bigg(1, \; 1, \; 1+\frac{1}{2\pi i}\int_{e^{\frac{\pi i}{6}}}^{e^{\frac{\pi i}{6}}\infty} \frac{f_{1}(\tfrac{1}{\omega^{2}k_{1}}) e^{-\frac{x}{2}\frac{1}{\omega^{2}k_{1}}} e^{\frac{t}{4} \frac{1}{(\omega^{2}k_{1})^{2}}} e^{\frac{x}{2} \omega^{2}k_{1}} e^{\frac{T-t}{4}(\omega^{2}k_{1})^{2}} }{k_{1}-k}dk_{1} \bigg).
\end{align*}
Hence $\tilde{n}_{3}(x,t,k) = 1+\frac{\tilde{n}_{3}^{(1)}(x,t)}{k} + O(k^{-2})$ as $k\to +\infty$, where
\begin{align*}
\tilde{n}_{3}^{(1)}(x,t) & = \frac{-1}{2\pi i}\int_{e^{\frac{\pi i}{6}}}^{e^{\frac{\pi i}{6}}\infty} f_{1}(\tfrac{1}{\omega^{2}k_{1}}) e^{-\frac{x}{2}\frac{1}{\omega^{2}k_{1}}} e^{\frac{t}{4} \frac{1}{(\omega^{2}k_{1})^{2}}} e^{\frac{x}{2} \omega^{2}k_{1}} e^{\frac{T-t}{4}(\omega^{2}k_{1})^{2}}dk_{1} \\
& = \frac{\omega}{2\pi}\int_{1}^{+\infty} f_{1}(\tfrac{i}{y}) e^{-\frac{x}{2}\frac{i}{y}} e^{-\frac{t}{4} \frac{1}{y^{2}}} e^{-\frac{x}{2} iy} e^{-\frac{T-t}{4}y^{2}}dy,
\end{align*}
and $\tilde{u}(x,t) := -i\sqrt{3}\frac{\partial}{\partial x}\tilde{n}_{3}^{(1)}(x,t)$ is given for $x\in \R$ and $t\in [0,T)$ by
\begin{align}\label{expr for ut}
\tilde{u}(x,t) = -\frac{\sqrt{3}}{2\pi}\int_{1}^{+\infty} f_{1}(\tfrac{i}{y}) e^{-\frac{x}{2}\frac{i}{y}} e^{-\frac{t}{4} \frac{1}{y^{2}}} \frac{y+\frac{1}{y}}{2} e^{-\frac{x}{2} iy} e^{-\frac{T-t}{4}y^{2}}dy.
\end{align}

\noindent \textbf{Heuristic discussion about the range $\delta \in (0,1)$ in Theorem \ref{thm:main}.} Note that
\begin{align}\label{lol27}
e^{-\frac{x}{2}\frac{i}{y}} e^{-\frac{t}{4} \frac{1}{y^{2}}} \frac{y+\frac{1}{y}}{2} = \frac{y}{2}(1+o(1)) \qquad \mbox{as } y\to +\infty
\end{align}
uniformly for $(x,t)$ in compact subsets of $\R\times [0,T]$. Hence, if $y \mapsto f_{1}(\frac{i}{y})$ tends to $+\infty$ as $y\to + \infty$, we have
\begin{align}
\tilde{u}(0,t) = -\big( 1+o(1) \big)\frac{\sqrt{3}}{2\pi}\int_{1}^{+\infty} f_{1}(\tfrac{i}{y}) \frac{y}{2} e^{-\frac{T-t}{4}y^{2}}dy, \quad \mbox{as } t \to T.
\end{align}
Let $\epsilon \in (0,2)$, and suppose $f_{1}(\frac{i}{y}) = e^{y^{2-\epsilon}}$ for all $y\geq 10$. Since the maximum of $e^{y^{2-\epsilon}} e^{-\frac{T-t}{4}y^{2}}$ is attained at $y_{\star}:=(1-\epsilon/2)^{\frac{1}{\epsilon}} (\frac{4}{T-t})^{\frac{1}{\epsilon}}$,
\begin{align*}
|\tilde{u}(0,t)| \geq \big( 1+o(1) \big)\frac{\sqrt{3}}{2\pi}\int_{y_{\star}-1}^{y_{\star}+1} f_{1}(\tfrac{i}{y}) e^{-\frac{T-t}{4}y^{2}}dy \asymp e^{\frac{\epsilon (1-\frac{\epsilon}{2})^{2/\epsilon}}{2-\epsilon} (\frac{4}{T-t})^{\frac{2-\epsilon}{\epsilon}}}, \quad \mbox{as } t \to T.
\end{align*}
Since we expect $\tilde{u}(x,t)$ to be of the same order as $u(x,t)$ as $t\to T$, we also expect the range $\delta \in (0,1)$ in Theorem \ref{thm:main} to be far from optimal.

\medskip \noindent \textbf{Heuristic discussion about the range $\delta \in (0,\frac{1}{2})$ in Theorem \ref{thm:wave breaking generalized}.} Suppose for simplicity that $f_{1}(\frac{i}{y})\geq 0$ for $y\geq 1$. Then for $\tilde{u}(x,t)$ to admit a continuous extension on $\R \times [0,T]$,  by \eqref{expr for ut} we must have $\int_{1}^{+\infty} yf_{1}(\frac{i}{y})dy<+\infty$. Since
\begin{align*}
\frac{\partial}{\partial x}\tilde{u}(x,t) = \frac{\sqrt{3}i}{2\pi}\int_{1}^{+\infty} f_{1}(\tfrac{i}{y}) e^{-\frac{x}{2}\frac{i}{y}} e^{-\frac{t}{4} \frac{1}{y^{2}}} \bigg(\frac{y+\frac{1}{y}}{2}\bigg)^{2} e^{-\frac{x}{2} iy} e^{-\frac{T-t}{4}y^{2}}dy, \quad x\in \R, \; t\in [0,T),
\end{align*}
the derivative $\frac{\partial}{\partial x}\tilde{u}(0,t)$ blows up as $t\to T$ if $\int_{1}^{+\infty} y^{2} f_{1}(\frac{i}{y})dy = +\infty$. In this case we have
\begin{align*}
\frac{\partial}{\partial x}\tilde{u}(0,t) = \big( 1+o(1) \big)\frac{\sqrt{3}i}{2\pi}\int_{1}^{+\infty} f_{1}(\tfrac{i}{y}) \frac{y^{2}}{4}  e^{-\frac{T-t}{4}y^{2}}dy.
\end{align*}
Let $\eta \in (-3,-2)$, and suppose $f_{1}(\frac{i}{y})=y^{\eta}$ for $y\geq 10$. Then 
\begin{align*}
\bigg|\frac{\partial}{\partial x}\tilde{u}(0,t)\bigg| \asymp \int_{1}^{+\infty} y^{\eta+2}  e^{-\frac{T-t}{4}y^{2}}dy \asymp \frac{1}{(T-t)^{\frac{\eta+3}{2}}}, \qquad \mbox{as } t \to T.
\end{align*}
The above argument can easily be generalized to analyze higher order derivatives of $\tilde{u}$, and supports our believe that the range $\delta \in (0,\frac{1}{2})$ in Theorem \ref{thm:wave breaking generalized} and Corollary \ref{coro:wave breaking} is essentially optimal.

\medskip The main reason why $\tilde{u}:= -i\sqrt{3}\frac{\partial}{\partial x}\tilde{n}_{3}^{(1)}$ is not a solution to \eqref{boussinesq} is because $\tilde{n}$ does not satisfy the symmetries \eqref{nsymm}. (For example, $\tilde{u}$ is not even real-valued, and this is because $\tilde{n}$ does not obey the $\mathcal{B}$-symmetry in \eqref{nsymm}.) 

\subsection{A modified RH problem}\label{subsection:modified inverse scattering}
To prove Theorems \ref{thm:main} and \ref{thm:wave breaking generalized}, instead of considering RH problem \ref{RHn}, we will consider a simplified RH problem, but one which is complex enough to preserve the symmetries \eqref{nsymm}. This simplified RH problem is obtained from RH problem \ref{RHn} by setting
\begin{align}\label{assumption r1=r2=0}
& r_{1}(k) = 0 = r_{2}(k), \qquad \mbox{for all } k \in \partial \D.
\end{align}
Note that this situation is not covered by Theorem \ref{inverseth}, since $(\ref{Theorem2.3itemii})$ is not satisfied. 
If \eqref{assumption r1=r2=0} holds, then the jump matrices $v_{7}, v_{8}, v_{9}$ in \eqref{vdef} are each identically equal to the identity matrix, and thus, by Morera's theorem, the solution $n$ of RH Problem \ref{RHn} is also analytic in a neighborhood of the unit circle, except possibly near the points $\{i\kappa_j\}_{j=1}^6$.  More precisely, the RH Problem \ref{RHn} becomes as follows.
\begin{RHproblem}[RH problem for $n$ when \eqref{assumption r1=r2=0} holds]\label{RHn modified}
Find a $1 \times 3$-row-vector valued function $n(x,t,k)$ with the following properties:
\begin{enumerate}[$(a)$]
\item\label{RHnitema mod} $n(x,t,\cdot) : \C \setminus (\Gamma_{0}\cup \Gamma_\star) \to \mathbb{C}^{1 \times 3}$ is analytic, where $\Gamma_{0}:= \cup_{j=1}^{6} \Gamma_{j} = \Gamma \setminus \partial \D$.

\item\label{RHnitemb mod} The limits of $n(x,t,k)$ as $k$ approaches $\Gamma_{0}$ from the left and right exist, are continuous on $\Gamma_{0}$, and are denoted by $n_+$ and $n_-$, respectively. Furthermore, 
\begin{align}\label{njump mod}
  n_+(x,t,k) = n_-(x, t, k) v(x, t, k) \qquad \text{for} \quad k \in \Gamma_{0},
\end{align}
and $v$ is defined in \eqref{vdef}.
\item\label{RHnitemc mod} $n(x,t,k) = O(1)$ as $k \to k_{\star} \in \Gamma_\star$.

\item\label{RHnitemd mod} For $k \in \C \setminus \Gamma_{0}$, $n$ obeys the symmetries
\begin{align}\label{nsymm mod}
n(x,t,k) = n(x,t,\omega k)\mathcal{A}^{-1} = n(x,t,k^{-1}) \mathcal{B},
\end{align}
where $\mathcal{A}$ and $\mathcal{B}$ are defined in \eqref{def of Acal and Bcal}.
\item\label{RHniteme mod} $n(x,t,k) = (1,1,1) + O(k^{-1})$ as $k \to \infty$.
\end{enumerate}
\end{RHproblem}
Note that $(\ref{Theorem2.3itemiv})$ is automatically satisfied when \eqref{assumption r1=r2=0} holds. The following result is the analog of Theorem \ref{inverseth} for functions $r_{1},r_{2}$ satisfying \eqref{assumption r1=r2=0}; it shows that the simplified RH problem \ref{RHn modified} still allows to construct solutions of \eqref{boussinesq}. Theorem \ref{inverseth new} can be proved in the same way as in \cite[Section 4]{CLmain}, so we omit the proof.
\begin{theorem}\label{inverseth new}
Let $r_1:\hat{\Gamma}_{1} \to \C$ and $r_2: \hat{\Gamma}_{4}\setminus \{\omega^{2},-\omega^{2}\} \to \C$ be two functions satisfying properties $(\ref{Theorem2.3itemi})$, $(\ref{Theorem2.3itemiii})$, $(\ref{Theorem2.3itemv})$ and \eqref{assumption r1=r2=0}.
Define $T \in (0, \infty]$ by \eqref{Tdef}. Then RH problem \ref{RHn modified} has a unique solution $n(x,t,k)$ for each $(x,t) \in \R \times [0,T)$ and the function $n_{3}^{(1)}$ defined by
\begin{align}\label{def of n3p1p lol}
n_{3}^{(1)}(x,t) := \lim_{k\to \infty} k (n_{3}(x,t,k) -1)
\end{align}
is well-defined and smooth for $(x,t) \in \R \times [0,T)$. Moreover, $u(x,t)$ defined by
\begin{align}\label{recoveruvn 2}
u(x,t) = -i\sqrt{3}\frac{\partial}{\partial x}n_{3}^{(1)}(x,t),
\end{align}
is a Schwartz class solution of (\ref{boussinesq}) on $\R \times [0,T)$.
\end{theorem}


\section{Analysis of RH problem \ref{RHn modified}}\label{section:RH analysis}
The proofs of Theorems \ref{thm:main} and \ref{thm:wave breaking generalized} rely on Theorem \ref{inverseth new} and in particular on an analysis of RH problem \ref{RHn modified}. Suppose $r_{1}:\hat{\Gamma}_{1} \to \C$ satisfies properties $(\ref{Theorem2.3itemi})$, $(\ref{Theorem2.3itemiii})$, $(\ref{Theorem2.3itemv})$, \eqref{assumption r1=r2=0}, and that it is of the form
\begin{align}\label{r1 in the form w 2}
r_{1}(k) = f_{1}(k)  e^{\frac{T}{4k^{2}}} = f_{1}(k) e^{-\frac{T}{4|k|^{2}}} , \qquad k \in \Gamma_{1},
\end{align}
where $f_{1}\in C^{\infty}(\Gamma_{1})$ vanishes to all orders as $k\to k_{\star}\in \{-i,i\}$, $k \in \Gamma_{1}$ and is such that 
\begin{align}\label{f1 general bound}
|f_{1}(\tfrac{1}{k})| \leq e^{|k|^{2-\epsilon}}, \qquad \mbox{for all } k \in (-\tfrac{i}{\epsilon},-i\infty)
\end{align}
for some small $\epsilon >0$. Then Theorem \ref{inverseth new} implies that $u(x,t) = -i\sqrt{3}\frac{\partial}{\partial x}n_{3}^{(1)}(x,t)$ is a Schwartz class solution of (\ref{boussinesq}) on $\R \times [0,T)$, where $\smash{n_{3}^{(1)}}$ is given in terms of the solution $n$ of RH problem \ref{RHn modified} by \eqref{def of n3p1p lol}. 

Even though RH problem \ref{RHn modified} is simpler than RH problem \ref{RHn} due to the absence of jumps on the unit circle, the method we use in this work does not allow to treat general $f_{1}$ satisfying \eqref{f1 general bound}.

 In fact, we will prove Theorems \ref{thm:main} and \ref{thm:wave breaking generalized} by expressing the solution $n$ of RH problem \ref{RHn modified} explicitly as a series for the subclass of functions $r_{1}\in C^{\infty}(\Gamma_{1})$ satisfying
\begin{align}
& r_{1}(k) = 0, & & \hspace{-3cm} k \in [-i,-i\infty), \label{r1=0 on Gamma1p} \\
& r_{1}(k) = f_{1}(k) e^{-\frac{x_{0}}{2}\frac{1}{k}} e^{\frac{T}{4k^{2}}}, & & \hspace{-3cm} k \in (0,i], \label{lol26}
\end{align}
for some $x_{0}\in \R$, and where $f_{1}\in C^{\infty}((0,i))$ is such that 
\begin{align*}
& \bigg\| \frac{f_{1}(\frac{i}{\cdot})}{\cdot} \bigg\|_{L^{1}((1,+\infty))} \leq \frac{1}{2}.
\end{align*}
Throughout this work, if $\gamma$ denotes an oriented contour and $w:\gamma\to \C$, we write $w\in L^{1}(\gamma)$ if $\|w\|_{L^{1}(\gamma)}:=\int_{\gamma} |w(k)| \, |dk| < + \infty$.

The $\overline{k}$ symmetry \eqref{r1r2 relation with kbar symmetry} together with \eqref{r1=0 on Gamma1p} and \eqref{lol26} implies that
\begin{align}
& r_{2}(k)=0, & & \hspace{-3cm} k\in (0,-i], \label{lol40} \\
& r_{2}(k) = f_{2}(k) e^{-\frac{x_{0}}{2}k} e^{\frac{T}{4}k^{2}}, & & \hspace{-3cm} k \in [i,i\infty), \label{lol26 bis}
\end{align} 
where $f_{2}(k) := \tilde{r}(k) \overline{f_{1}(\bar{k}^{-1})}$ (recall that $\tilde{r}$ is defined in \eqref{def of tilde r}).

\begin{remark}
Conditions \eqref{assumption r1=r2=0}, \eqref{r1=0 on Gamma1p} and \eqref{lol40} imply that all ``good" off-diagonal entries of $v$ in \eqref{vdef} have been set to $0$ (here, by ``good" entries, we mean the entries of $v$ that remain bounded for $t>T$), all diagonal entries have been set to $1$, but all ``bad" entries of $v$ remain unchanged. In a sense, we are considering the simplest possible RH problem that allows to produce blow-up solutions of \eqref{boussinesq}. In \cite{CL V, CL IV, CL I, CLsolitonResolution}, the opposite situation is considered: the long-time behavior of global solutions of \eqref{boussinesq} is obtained by analyzing RH problem \ref{RHn} in the case where all ``bad" entries of $v$ are set to $0$.
\end{remark}

\begin{figure}
\begin{center}
\begin{tikzpicture}[scale=0.7]
\node at (0,0) {};
\draw[black,line width=0.45 mm,->-=0.4,->-=0.85] (0,0)--(30:4);
\draw[black,line width=0.45 mm,->-=0.4,->-=0.85] (0,0)--(90:4);
\draw[black,line width=0.45 mm,->-=0.4,->-=0.85] (0,0)--(150:4);
\draw[black,line width=0.45 mm,->-=0.4,->-=0.85] (0,0)--(-30:4);
\draw[black,line width=0.45 mm,->-=0.4,->-=0.85] (0,0)--(-90:4);
\draw[black,line width=0.45 mm,->-=0.4,->-=0.85] (0,0)--(-150:4);

\draw[black,line width=0.45 mm] ([shift=(-180:2.5cm)]0,0) arc (-180:180:2.5cm);
\draw[black,arrows={-Triangle[length=0.2cm,width=0.18cm]}]
($(3:2.5)$) --  ++(90:0.001);
\draw[black,arrows={-Triangle[length=0.2cm,width=0.18cm]}]
($(57:2.5)$) --  ++(-30:0.001);
\draw[black,arrows={-Triangle[length=0.2cm,width=0.18cm]}]
($(123:2.5)$) --  ++(210:0.001);
\draw[black,arrows={-Triangle[length=0.2cm,width=0.18cm]}]
($(177:2.5)$) --  ++(90:0.001);
\draw[black,arrows={-Triangle[length=0.2cm,width=0.18cm]}]
($(243:2.5)$) --  ++(330:0.001);
\draw[black,arrows={-Triangle[length=0.2cm,width=0.18cm]}]
($(297:2.5)$) --  ++(210:0.001);

\draw[black,line width=0.15 mm] ([shift=(-30:0.55cm)]0,0) arc (-30:30:0.55cm);

\node at (0.8,0) {$\tiny \frac{\pi}{3}$};

\node at (-1:2.9) {\footnotesize $\Gamma_8$};
\node at (60:2.9) {\footnotesize $\Gamma_9$};
\node at (120:2.9) {\footnotesize $\Gamma_7$};
\node at (181:2.9) {\footnotesize $\Gamma_8$};
\node at (240:2.83) {\footnotesize $\Gamma_9$};
\node at (300:2.83) {\footnotesize $\Gamma_7$};

\node at (104:1.45) {\footnotesize $\Gamma_1''$};
\node at (138:1.45) {\footnotesize $\Gamma_2''$};
\node at (223:1.45) {\footnotesize $\Gamma_3''$};
\node at (-104:1.45) {\footnotesize $\Gamma_4''$};
\node at (-43:1.45) {\footnotesize $\Gamma_5''$};
\node at (44:1.45) {\footnotesize $\Gamma_6''$};

\node at (96:3.3) {\footnotesize $\Gamma_4'$};
\node at (144:3.3) {\footnotesize $\Gamma_5'$};
\node at (217:3.3) {\footnotesize $\Gamma_6'$};
\node at (-96:3.3) {\footnotesize $\Gamma_1'$};
\node at (-36:3.3) {\footnotesize $\Gamma_2'$};
\node at (36:3.3) {\footnotesize $\Gamma_3'$};
\end{tikzpicture}
\end{center}
\begin{figuretext}\label{fig: Dn prim}
The contour $\Gamma = \cup_{j=1}^6 (\Gamma_j'\cup \Gamma_j'') \cup \cup_{j=7}^9 \Gamma_j \cup \Gamma_{\star}$ in the complex $k$-plane.
\end{figuretext}
\end{figure}

Let $\D := \{k:|k|<1\}$.  For $j=1,\ldots,6$, let us write $\Gamma_{j} = \Gamma_{j'}\cup \Gamma_{j''}$, where $\Gamma_{j'} = \Gamma_{j} \setminus \D$ and $\Gamma_{j''} := \Gamma_{j} \cap \D$, see also Figure \ref{fig: Dn prim}. Thus the jump matrix $v$ defined in \eqref{vdef} becomes
\begin{align*}
&  v_{1'} =   \begin{pmatrix}  
 1 & 0 & 0 \\
  r_1(\frac{1}{k})e^{\theta_{21}} & 1 & 0 \\
  0 & 0 & 1
  \end{pmatrix},
	\quad  v_{2'} = \begin{pmatrix}   
 1 & 0 & 0 \\
 0 & 1 & 0 \\
 0 & r_2(\omega k)e^{\theta_{32}} & 1 
    \end{pmatrix}, \quad v_{3'} = \begin{pmatrix} 
 1 & 0 & r_1(\frac{1}{\omega^2 k})e^{-\theta_{31}} \\
 0 & 1 & 0 \\
 0 & 0 & 1  
  \end{pmatrix},
   	\nonumber \\ 
  & v_{4'} = \begin{pmatrix}  
  1 & 0 & 0 \\
  r_2(k)e^{\theta_{21}} & 1 & 0 \\
  0 & 0 & 1
   \end{pmatrix},
   	\quad v_{5'} = \begin{pmatrix}
  1 & 0 & 0 \\
  0 & 1 & 0 \\
  0 & r_1(\frac{1}{\omega k})e^{\theta_{32}} & 1
  \end{pmatrix},
	\quad   v_{6'} = \begin{pmatrix} 
  1 & 0 & r_2(\omega^2 k)e^{-\theta_{31}} \\
  0 & 1 & 0 \\
  0 & 0 & 1
   \end{pmatrix}, \\ 
&  v_{1''} \hspace{-0.05cm} = \hspace{-0.05cm} \begin{pmatrix}  
 1 & - r_1(k)e^{-\theta_{21}} & 0 \\
  0 & 1 & 0 \\
  0 & 0 & 1
  \end{pmatrix}\hspace{-0.1cm},
	\;  v_{2''} \hspace{-0.05cm} = \hspace{-0.05cm} \begin{pmatrix}   
 1 & 0 & 0 \\
 0 & 1 & -r_2(\frac{1}{\omega k})e^{-\theta_{32}} \\
 0 & 0 & 1 
    \end{pmatrix}\hspace{-0.1cm},
   	\; v_{3''} \hspace{-0.05cm} = \hspace{-0.05cm} \begin{pmatrix} 
 1 & 0 & 0 \\
 0 & 1 & 0 \\
 -r_1(\omega^2 k)e^{\theta_{31}} & 0 & 1  
  \end{pmatrix}\hspace{-0.1cm}, \nonumber
   	\\ 
& v_{4''} \hspace{-0.05cm} = \hspace{-0.05cm} \begin{pmatrix}  
  1 & -r_2(\frac{1}{k}) e^{-\theta_{21}} & 0 \\
  0 & 1 & 0 \\
  0 & 0 & 1
   \end{pmatrix}\hspace{-0.1cm}, \; v_{5''} \hspace{-0.05cm} = \hspace{-0.05cm} \begin{pmatrix}
  1 & 0 & 0 \\
  0 & 1 & -r_1(\omega k)e^{-\theta_{32}} \\
  0 & 0 & 1 
  \end{pmatrix}\hspace{-0.1cm},
	\;   v_{6''} \hspace{-0.05cm} = \hspace{-0.05cm} \begin{pmatrix} 
  1 & 0 & 0 \\
  0 & 1 & 0 \\
  -r_2(\frac{1}{\omega^2 k})e^{\theta_{31}} & 0 & 1
   \end{pmatrix}\hspace{-0.1cm},
\end{align*}
where $v_{j'}, v_{j''}$ are the restrictions of $v$ to $\Gamma_{j'}$ and $\Gamma_{j''}$, respectively. Substituting \eqref{lol26}, \eqref{lol26 bis} and the relations $\theta_{32}(x,t,k)=\theta_{21}(x,t,\omega k)$, $\theta_{31}(x,t,k)=-\theta_{21}(x,t,\omega^{2}k)$, and
\begin{align*}
\theta_{21}(x,t,k) = \frac{x}{2}\bigg( k-\frac{1}{k} \bigg) + \frac{t}{4}\bigg( \frac{1}{k^{2}}-k^{2} \bigg),
\end{align*}
and using the notation
\begin{align}\label{def of h1 h2}
& h_{1}(k) = h_{1}(x,t,k) := f_{1}(k)e^{-\frac{x}{2}k+\frac{t}{4}k^{2}}, & & h_{2}(k) = h_{2}(x,t,k) := f_{2}(k) e^{-\frac{x}{2}\frac{1}{k}+\frac{t}{4}\frac{1}{k^{2}}},
\end{align}
we get
\begin{align*}
&  v_{1'} =   \begin{pmatrix}  
 1 & 0 & 0 \\
  h_{1}(\frac{1}{k}) e^{\frac{x-x_{0}}{2}k} e^{\frac{T-t}{4}k^{2}} & 1 & 0 \\
  0 & 0 & 1
  \end{pmatrix},
	\quad  v_{2'} = \begin{pmatrix}   
 1 & 0 & 0 \\
 0 & 1 & 0 \\
 0 & h_{2}(\omega k)e^{\frac{x-x_{0}}{2}\omega k}e^{\frac{T-t}{4}(\omega k)^{2}} & 1 
    \end{pmatrix},
   	\nonumber \\ \nonumber
  &v_{3'} = \begin{pmatrix} 
 1 & 0 & h_1(\frac{1}{\omega^2 k})e^{\frac{x-x_{0}}{2}\omega^{2}k}e^{\frac{T-t}{4}(\omega^{2}k)^{2}} \\
 0 & 1 & 0 \\
 0 & 0 & 1  
  \end{pmatrix},
	\quad   v_{4'} = \begin{pmatrix}  
  1 & 0 & 0 \\
  h_2(k)e^{\frac{x-x_{0}}{2}k}e^{\frac{T-t}{4}k^{2}} & 1 & 0 \\
  0 & 0 & 1
   \end{pmatrix},
   	\\ 
&  v_{5'} = \begin{pmatrix}
  1 & 0 & 0 \\
  0 & 1 & 0 \\
  0 & h_1(\frac{1}{\omega k})e^{\frac{x-x_{0}}{2}\omega k}e^{\frac{T-t}{4}(\omega k)^{2}} & 1
  \end{pmatrix},
	\;   v_{6'} = \begin{pmatrix} 
  1 & 0 & h_2(\omega^2 k)e^{\frac{x-x_{0}}{2}\omega^{2}k}e^{\frac{T-t}{4}(\omega^{2}k)^{2}} \\
  0 & 1 & 0 \\
  0 & 0 & 1
   \end{pmatrix}, \\ 
&  v_{1''} =   \begin{pmatrix}  
 1 & - h_1(k) e^{\frac{x-x_{0}}{2k}} e^{\frac{T-t}{4}\frac{1}{k^{2}}} & 0 \\
  0 & 1 & 0 \\
  0 & 0 & 1
  \end{pmatrix},
	\quad  v_{2''} = \begin{pmatrix}   
 1 & 0 & 0 \\
 0 & 1 & -h_2(\frac{1}{\omega k}) e^{\frac{x-x_{0}}{2\omega k}}  e^{\frac{T-t}{4}\frac{1}{(\omega k)^{2}}} \\
 0 & 0 & 1 
    \end{pmatrix},
   	\nonumber \\ \nonumber
  &v_{3''} = \begin{pmatrix} 
 1 & 0 & 0 \\
 0 & 1 & 0 \\
 -h_1(\omega^2 k)e^{\frac{x-x_{0}}{2\omega^{2}k}}e^{\frac{T-t}{4}\frac{1}{(\omega^{2}k)^{2}}} & 0 & 1  
  \end{pmatrix},
	\quad   v_{4''} = \begin{pmatrix}  
  1 & -h_2(\frac{1}{k}) e^{\frac{x-x_{0}}{2k}} e^{\frac{T-t}{4}\frac{1}{k^{2}}} & 0 \\
  0 & 1 & 0 \\
  0 & 0 & 1
   \end{pmatrix},
   	\\ 
&  v_{5''} = \begin{pmatrix}
  1 & 0 & 0 \\
  0 & 1 & -h_1(\omega k)e^{\frac{x-x_{0}}{2\omega k}}e^{\frac{T-t}{4}\frac{1}{(\omega k)^{2}}} \\
  0 & 0 & 1 
  \end{pmatrix},
	\;   v_{6''} = \begin{pmatrix} 
  1 & 0 & 0 \\
  0 & 1 & 0 \\
  -h_2(\frac{1}{\omega^2 k}) e^{\frac{x-x_{0}}{2\omega^{2}k}} e^{\frac{T-t}{4}\frac{1}{(\omega^{2}k)^{2}}} & 0 & 1
   \end{pmatrix}. 
\end{align*}
Let $n(k)=n(x,t,k):=(n_{1}(k),n_{2}(k),n_{3}(k))$ be the unique solution of RH problem \ref{RHn modified}. The $\mathcal{A}$-symmetry in \eqref{nsymm mod} implies that $n_{1}(k)=n_{3}(\omega k)$ and $n_{2}(k) = n_{3}(\omega^{2}k)$, i.e.\ that $n$ is of the form $n(k)=(n_{3}(\omega k),n_{3}(\omega^{2}k),n_{3}(k))$. Then the jumps \eqref{njump mod} imply that $n_{3}$ is analytic in $\C\setminus \big(\Gamma_{3'}\cup \Gamma_{6'} \cup \Gamma_{2''} \cup \Gamma_{5''} \cup \{0,e^{-\frac{5\pi i}{6}}, e^{-\frac{\pi i}{6}}, e^{\frac{\pi i}{6}}, e^{\frac{5\pi i}{6}}\}\big)$ and satisfies
\begin{align*}
& n_{3,+}(k) = n_{3,-}(k) + \tilde{g}(\omega^{2}k)n_{3}(\omega k), & & k \in \Gamma_{3'}\cup \Gamma_{6'}, \\
& n_{3,+}(k) = n_{3,-}(k) + \hat{g}(\omega k)n_{3}(\omega^{2} k), & & k \in \Gamma_{2''} \cup \Gamma_{5''},
\end{align*}
where $\tilde{g}(k)=\tilde{g}(x,t,k)$ and $\hat{g}(k)=\hat{g}(x,t,k)$ are given by
\begin{align}
& \tilde{g}(k) := \begin{cases}
r_1(\frac{1}{k})e^{\theta_{21}} = h_{1}(x,t,\frac{1}{k}) e^{\frac{x-x_{0}}{2}k} e^{\frac{T-t}{4}k^{2}}, & \mbox{for } k\in \Gamma_{1'}, \\
r_2(k)e^{\theta_{21}} = h_{2}(x,t,k) e^{\frac{x-x_{0}}{2}k} e^{\frac{T-t}{4}k^{2}}, & \mbox{for } k\in \Gamma_{4'},
\end{cases} \label{gtilde def} \\
& \hat{g}(k) := \begin{cases}
- r_1(k)e^{-\theta_{21}}=-h_{1}(x,t,k) e^{\frac{x-x_{0}}{2k}} e^{\frac{T-t}{4}\frac{1}{k^{2}}}, & \mbox{for } k\in \Gamma_{1''}, \\
- r_2(\frac{1}{k})e^{-\theta_{21}}=-h_{2}(x,t,\tfrac{1}{k}) e^{\frac{x-x_{0}}{2k}} e^{\frac{T-t}{4}\frac{1}{k^{2}}}, & \mbox{for } k\in \Gamma_{4''}.
\end{cases} \label{ghat def}
\end{align}
Note that $\hat{g}(k) = -\tilde{g}(k^{-1})$. Since $n_{3}(k) = 1+O(k^{-1})$ as $k\to\infty$, by the Sokhotski–Plemelj formula we have
\begin{align}\label{lol14}
& n_{3}(k) = 1+\int_{\Gamma_{3'}\cup \Gamma_{6'}} \frac{\tilde{g}(\omega^{2}k_{1})}{k_{1}-k}n_{3}(\omega k_{1})\frac{dk_{1}}{2\pi i} + \int_{\Gamma_{2''}\cup \Gamma_{5''}} \frac{\hat{g}(\omega k_{1})}{k_{1}-k}n_{3}(\omega^{2}k_{1})\frac{dk_{1}}{2\pi i}.
\end{align}
The $\mathcal{B}$-symmetry in \eqref{nsymm mod} implies that $n_{3}(k)=n_{3}(k^{-1})$. Hence
\begin{align}
n_{3}(k) & = 1+\int_{\Gamma_{5'}\cup \Gamma_{2'}} \frac{\tilde{g}(\omega k_{1})}{\omega^{2}k_{1}-k} n_{3}(k_{1})\omega^{2}\frac{dk_{1}}{2\pi i} + \int_{\Gamma_{6''}\cup \Gamma_{3''}} \frac{\hat{g}(\omega^{2} k_{1})}{\omega k_{1}-k} n_{3}(k_{1}) \omega \frac{dk_{1}}{2\pi i} \nonumber \\
& = 1+\int_{\Gamma_{5'}\cup \Gamma_{2'}} \frac{\tilde{g}(\omega k_{1})}{\omega^{2}k_{1}-k} n_{3}(k_{1})\omega^{2}\frac{dk_{1}}{2\pi i} + \int_{\Gamma_{6''}\cup \Gamma_{3''}} \frac{\hat{g}(\omega^{2} k_{1})}{\omega k_{1}-k} n_{3}(k_{1}^{-1}) \omega \frac{dk_{1}}{2\pi i} \nonumber \\
& = 1+\int_{\Gamma_{5'}\cup \Gamma_{2'}} \frac{\tilde{g}(\omega k_{1})}{\omega^{2}k_{1}-k} n_{3}(k_{1})\omega^{2}\frac{dk_{1}}{2\pi i} + \int_{\Gamma_{2'}\cup \Gamma_{5'}} \frac{\hat{g}(\frac{1}{\omega k_{1}})}{\frac{1}{\omega^{2} k_{1}}-k} n_{3}(k_{1}) \omega \frac{dk_{1}}{2\pi i k_{1}^{2}} \nonumber \\
& = 1+ \int_{\Gamma_{2'}\cup \Gamma_{5'}} \bigg( \frac{\omega^{2}}{\omega^{2}k_{1}-k} - \frac{\frac{\omega}{k_{1}^{2}}}{\frac{1}{\omega^{2} k_{1}}-k} \bigg) \tilde{g}(\omega k_{1}) n_{3}(k_{1})\frac{dk_{1}}{2\pi i}. \label{lol7}
\end{align}
We can rewrite this as
\begin{align}\label{volterra}
n_{3}(k) = 1+\int_{\tilde{\Gamma}} F(k,k_{1}) n_{3}(k_{1})dk_{1},
\end{align}
with $\tilde{\Gamma} := \Gamma_{2'}\cup \Gamma_{5'}$ and 
\begin{align}
F(k,k_{1}) & = F(x,t,k,k_{1}) := \bigg( \frac{\omega^{2}}{\omega^{2}k_{1}-k} - \frac{\frac{\omega}{k_{1}^{2}}}{\frac{1}{\omega^{2} k_{1}}-k} \bigg) \frac{\tilde{g}(x,t,\omega k_{1})}{2\pi i}. \label{def of F 0} 
\end{align}
It will also be convenient to introduce the notation
\begin{align}
& \tilde{h}_{0}(x,t,k) := \begin{cases}
h_{1}(x,t,\frac{1}{k}), & \mbox{for } k\in \Gamma_{1'}, \\
h_{2}(x,t,k), & \mbox{for } k\in \Gamma_{4'},
\end{cases} & & \tilde{f}_{0}(k) := \begin{cases}
f_{1}(\frac{1}{k}), & \mbox{for } k\in \Gamma_{1'}, \\
f_{2}(k), & \mbox{for } k\in \Gamma_{4'}.
\end{cases} \label{lol 4 5 bis 2}
\end{align} 
The function $\tilde{f}_{0}$ is independent of $x$ and $t$, and by \eqref{def of h1 h2} we have
\begin{align}\label{relation betweens h0 and f0}
\tilde{h}_{0}(x,t,k) = \tilde{f}_{0}(k)e^{-\frac{x}{2}\frac{1}{k}+\frac{t}{4}\frac{1}{k^{2}}}.
\end{align}
Using \eqref{gtilde def} and \eqref{lol 4 5 bis 2}, $F$ can be rewritten as
\begin{align}
F(x,t,k,k_{1}) & = e^{\frac{x-x_{0}}{2}\omega k_{1}} e^{\frac{T-t}{4}(\omega k_{1})^{2}}  \bigg( \frac{\omega^{2}}{\omega^{2}k_{1}-k} - \frac{\frac{\omega}{k_{1}^{2}}}{\frac{1}{\omega^{2} k_{1}}-k} \bigg) \frac{\tilde{h}_{0}(x,t,\omega k_{1})}{2\pi i}. \label{def of F}
\end{align}
Note also that
\begin{align*}
\frac{\omega^{2}}{\omega^{2}k_{1}-k} - \frac{\frac{\omega}{k_{1}^{2}}}{\frac{1}{\omega^{2} k_{1}}-k} = \frac{\omega k_{1}-\frac{1}{k_{1}}}{ (\omega^{2}k_{1}-k)(\omega^{2}k_{1}-\frac{1}{k})},
\end{align*}
from which it readily follows that
\begin{multline}\label{F symm new}
F(x,t,k,k_{1}) = F(x,t,k^{-1},k_{1}), \\ k\in \C\setminus (\overline{\Gamma_{3'}\cup \Gamma_{6'} \cup \Gamma_{2''}\cup \Gamma_{5''}}), \; k_{1} \in \Gamma_{2'}\cup \Gamma_{5'}, \; x\in \R, \; t\geq 0.
\end{multline}


\begin{proposition}\label{prop:n3 as a series via volterra}
Let $r_1:\hat{\Gamma}_{1} \to \C$ and $r_2: \hat{\Gamma}_{4}\setminus \{\omega^{2},-\omega^{2}\} \to \C$ be two functions satisfying properties $(\ref{Theorem2.3itemi})$, $(\ref{Theorem2.3itemiii})$, $(\ref{Theorem2.3itemv})$, \eqref{assumption r1=r2=0} and such that
\begin{align}
& r_{1}(k) = 0, & & \mbox{for } k \in [-i,-i\infty), \nonumber \\
& r_{1}(k) = f_{1}(k) e^{-\frac{x_{0}}{2}\frac{1}{k}} e^{\frac{T}{4k^{2}}}, & & \mbox{for } k \in (0,i], \label{lol15}
\end{align}
for some $x_{0}\in \R$ and $T>0$ and where $f_{1}$ is such that
\begin{align}
& f_{1} \in C^{\infty}((0,i)), \label{f1 is Cinf} \\
& f_{1}(k) = 0 \hspace{2cm} \mbox{for } k\in [\tfrac{i}{2},i], \label{f1 is 0 on i 2 to i} \\
& \bigg\| \frac{f_{1}(\frac{i}{\cdot})}{\cdot} \bigg\|_{L^{1}((1,+\infty))} \leq \frac{1}{M}, \label{L1 norm is less than 1 2}
\end{align}
for some $M\geq 2$. 
For $x\in \R$, $t\in [0,T]$ and $k \in \C\setminus (\overline{\Gamma_{3'}\cup \Gamma_{6'} \cup \Gamma_{2''}\cup \Gamma_{5''}})$, define $m(k)=m(x,t,k)$ by
\begin{align}\label{def of n3}
m(x,t,k)=1+\sum_{j=1}^\infty m_{j}(x,t,k), 
\end{align}
where 
\begin{align}\label{def of mpjp}
m_{j}(x,t,k) & := \int_{\tilde{\Gamma}}\dots \int_{\tilde{\Gamma}} F(k,k_{1}) F(k_{1},k_{2})
\cdots F(k_{j-1},k_{j}) dk_1 \cdots dk_j, \qquad j \geq 1.
\end{align}
The function $m$ satisfies the following properties:
\begin{itemize}
\item[(a)] For each fixed $x\in \R$ and $t\in [0,T]$, the function $k\mapsto m(x,t,k)$ is analytic on $\C\setminus (\overline{\Gamma_{3'}\cup \Gamma_{6'} \cup \Gamma_{2''}\cup \Gamma_{5''}})$ and satisfies \eqref{volterra} (with $n_{3}$ replaced by $m$).
\item[(b)] For each fixed $x\in \R$ and $t\in [0,T]$, the limits of $m(x,t,k)$ as $k$ approaches $\Gamma_{3'}\cup \Gamma_{6'} \cup \Gamma_{2''}\cup \Gamma_{5''}$ from the left and right exist and are continuous on $\Gamma_{3'}\cup \Gamma_{6'} \cup \Gamma_{2''}\cup \Gamma_{5''}$. 
\item[(c)] For each fixed $x\in \R$ and $t\in [0,T]$, 
\begin{align*}
m(x,t,k)=O(1), \qquad \mbox{as } k \to k_{\star}\in \{e^{-\frac{5\pi i}{6}}, e^{-\frac{\pi i}{6}}, e^{\frac{\pi i}{6}}, e^{\frac{5 \pi i}{6}}\}.
\end{align*}
\item[(d)] For each $x\in \R$ and $t\in [0,T]$, $m$ obeys the symmetry
\begin{align}\label{n3 sym}
m(x,t,k) = m(x,t,k^{-1}),  \qquad k \in \C\setminus (\overline{\Gamma_{3'}\cup \Gamma_{6'} \cup \Gamma_{2''}\cup \Gamma_{5''}}).
\end{align}
\item[(e)] For each fixed $x\in \R$ and $t\in [0,T]$,
\begin{align}\label{m near inf}
m(x,t,k) = 1+O(k^{-1}), \qquad \mbox{as } k \to \infty,
\end{align}
uniformly for $\arg k \in [0,2\pi]$.
\end{itemize}
Moreover, for all $x\in \R$ and $t\in [0,T]$,
\begin{align}\label{lol12}
\| \sup_{s\in \tilde{\Gamma}_{\mathrm{m}}} |F(x,t,s,\cdot)| \|_{L^{1}(\tilde{\Gamma})} \leq \frac{1}{M},
\end{align}
where $\tilde{\Gamma}_{\mathrm{m}}:=(\Gamma_{2'}\cup \Gamma_{5'})\cap \{k\in \C:|k| > 2\}$.
\end{proposition}
\begin{proof}
We first prove that $m$ in \eqref{def of n3} is well-defined for $x\in \R$, $t\in [0,T]$ and $k \in \C\setminus (\overline{\Gamma_{3'}\cup \Gamma_{6'} \cup \Gamma_{2''}\cup \Gamma_{5''}})$. Let $m_{0} = 1$ and define $m_{j}(k)=m_{j}(x,t,k)$ for $j \in \N_{>0}$, $x \in \R$, $t\in [0,T]$, $k \in \C\setminus (\Gamma_{3'}\cup \Gamma_{6'} \cup \Gamma_{2''}\cup \Gamma_{5''})$ inductively by 
\begin{align}\label{def of mpjp inductively}
& m_{j+1}(x,t,k) = \int_{\tilde{\Gamma}} F(x,t,k,k_{1}) m_{j}(x,t,k_{1}) dk_{1}.
\end{align}
The above formula can be rewritten as in \eqref{def of mpjp}, and thus
\begin{align}
|m_{j}(x,t,k)| & \leq \int_{\tilde{\Gamma}}\dots \int_{\tilde{\Gamma}} |F(x,t,k,k_{1})| \prod_{p=2}^{j} \sup_{s\in \tilde{\Gamma}} |F(x,t,s,k_{p})|\; |dk_{p}| \nonumber \\
& = \int_{\tilde{\Gamma}} |F(x,t,k,k_{1})| \, |dk_{1}| \; \times \; \bigg( \int_{\tilde{\Gamma}} \sup_{s\in \tilde{\Gamma}} |F(x,t,s,k_{1})|\; |dk_{1}| \bigg)^{j-1}. \label{bound for mpjp}
\end{align}
Using \eqref{f1 is 0 on i 2 to i} and recalling that $f_{2}(k) := \tilde{r}(k) \overline{f_{1}(\bar{k}^{-1})}$, we infer that $f_{2}(k) = 0$ for $k\in [i,2i]$. Therefore, by \eqref{lol 4 5 bis 2}--\eqref{def of F}, every occurrence of $\tilde{\Gamma}$ in \eqref{def of mpjp inductively} and \eqref{bound for mpjp} can be replaced by $\tilde{\Gamma}_{\mathrm{m}}$. For $k_{1}\in \tilde{\Gamma}_{\mathrm{m}}$, simple geometric considerations show that
\begin{align}\label{some inf}
\inf_{s\in \tilde{\Gamma}_{\mathrm{m}} }|\omega^{2}k_{1}-s| \geq \frac{\sqrt{3}}{2}|k_{1}|, \qquad \inf_{s\in \tilde{\Gamma}_{\mathrm{m}}}|\tfrac{1}{\omega^{2} k_{1}}-s| \geq |s|-\frac{1}{|k_{1}|} \geq \frac{3}{2},
\end{align}
which yield
\begin{align*}
& \sup_{s\in \tilde{\Gamma}_{\mathrm{m}}} \bigg| \frac{e^{\frac{x-x_{0}}{2}\omega k_{1}} e^{\frac{T-t}{4}(\omega k_{1})^{2}}}{2\pi i}  \bigg( \frac{\omega^{2}}{\omega^{2}k_{1}-s} - \frac{\frac{\omega}{k_{1}^{2}}}{\frac{1}{\omega^{2} k_{1}}-s} \bigg) \bigg| \\
& \leq \frac{e^{\frac{T-t}{4}(\omega k_{1})^{2}}}{2\pi} \bigg( \frac{1}{\frac{\sqrt{3}}{2}|k_{1}|} + \frac{\frac{1}{2|k_{1}|}}{\frac{3}{2}} \bigg) \leq \frac{e^{\frac{T-t}{4}(\omega k_{1})^{2}}}{4|k_{1}|}.
\end{align*}
Thus, by \eqref{lol 4 5 bis 2}--\eqref{def of F}, for $x\in \R$, $t\in [0,T]$ and $k_{1}\in \tilde{\Gamma}$ we get
\begin{align}
\sup_{s\in \tilde{\Gamma}_{\mathrm{m}}} |F(x,t,s,k_{1})| & \leq \frac{e^{\frac{T-t}{4}(\omega k_{1})^{2}}}{4|k_{1}|} \begin{cases}
|h_{2}(x,t,\omega k_{1})|, & \mbox{if } k_{1}\in \Gamma_{2'} \\
|h_{1}(x,t,\frac{1}{\omega k_{1}})|, & \mbox{if } k_{1}\in \Gamma_{5'}
\end{cases} \nonumber \\
& \leq \frac{e^{\frac{T-t}{4}(\omega k_{1})^{2}}}{4|k_{1}|} \begin{cases}
|f_{2}(\omega k_{1})|, & \mbox{if } k_{1}\in \Gamma_{2'}, \\
|f_{1}(\frac{1}{\omega k_{1}})|, & \mbox{if } k_{1}\in \Gamma_{5'}.
\end{cases} \label{lol10}
\end{align}
Using also $f_{2}(k) := \tilde{r}(k) \overline{f_{1}(\bar{k}^{-1})}$, \eqref{L1 norm is less than 1 2} and the fact that $|\tilde{r}(k_{1})|=1$ for $k_{1}\in i \R$, for all $x\in \R$ and $t\in [0,T]$ we obtain
\begin{align*}
\| \sup_{s\in \tilde{\Gamma}_{\mathrm{m}}} |F(x,t,s,\cdot)| \|_{L^{1}(\tilde{\Gamma})} & \leq \int_{1}^{+\infty} \frac{e^{-\frac{T-t}{4}y^{2}}}{4y}\big( 1+|\tilde{r}(iy)|  \big) |f_{1}(\tfrac{i}{y})| dy \\
& \leq \frac{1}{2}\int_{1}^{+\infty} \frac{|f_{1}(\tfrac{i}{y})|}{y} dy \leq \frac{1}{2M}, 
\end{align*}
which implies \eqref{lol12}. By \eqref{bound for mpjp} and \eqref{lol12}, for $j\in \N_{>0}$, $x \in \R$, $t \in [0,T]$ and $k \in \C\setminus (\overline{\Gamma_{3'}\cup \Gamma_{6'} \cup \Gamma_{2''}\cup \Gamma_{5''}})$, we have
\begin{align} \label{phi0jestimate}
|m_{j}(x,t,k)| \leq  M^{-(j-1)}\int_{\tilde{\Gamma}}|F(x,t,k,k_{1})| \, |dk_{1}|.
\end{align}
Let $K$ be a compact subset of $\C\setminus (\overline{\Gamma_{3'}\cup \Gamma_{6'} \cup \Gamma_{2''}\cup \Gamma_{5''}})$. Since $K$ is compact, there exists $C_{K}>0$ such that
\begin{align*}
\sup_{k\in K} \bigg| \frac{1}{2\pi i}\bigg( \frac{\omega^{2}}{\omega^{2}k_{1}-k} - \frac{\frac{\omega}{k_{1}^{2}}}{\frac{1}{\omega^{2} k_{1}}-k} \bigg) \bigg| \leq \frac{C_{K}}{2|k_{1}|} \qquad \mbox{for all } k_{1} \in \tilde{\Gamma}.
\end{align*}
This implies, in a similar way as in \eqref{lol10}, that for $x\in \R$, $t\in [0,T]$ and $k_{1}\in \tilde{\Gamma}$ we have
\begin{align}
\sup_{k\in K} |F(x,t,k,k_{1})| & \leq  C_{K}\frac{e^{\frac{T-t}{4}(\omega k_{1})^{2}}}{2|k_{1}|} \begin{cases}
|f_{2}(\omega k_{1})|, & \mbox{if } k_{1}\in \Gamma_{2'}, \\
|f_{1}(\frac{1}{\omega k_{1}})|, & \mbox{if } k_{1}\in \Gamma_{5'}.
\end{cases} \label{lol11}
\end{align}
Using again $f_{2}(k) := \tilde{r}(k) \overline{f_{1}(\bar{k}^{-1})}$, \eqref{L1 norm is less than 1 2} and $|\tilde{r}(k_{1})|=1$ for $k_{1}\in i \R$, for all $x\in \R$ and $t\in [0,T]$ we get
\begin{align*}
\int_{\tilde{\Gamma}}\sup_{k\in K}|F(x,t,k,k_{1})| |dk_{1}| \leq  C_{K}\int_{1}^{+\infty} \frac{|f_{1}(\tfrac{i}{y})|}{y}dy \leq \frac{C_{K}}{M}.
\end{align*}
By \eqref{phi0jestimate}, this yields the estimate
\begin{align*}
\sup_{k\in K} |m_{j}(x,t,k)| \leq  C_{K}M^{-j}, \qquad j\in \N_{>0}, \; x\in \R, \; t\in [0,T].
\end{align*}
Therefore, the series on the right-hand side of \eqref{def of n3} converges absolutely and uniformly for $k \in K$, $x\in \R$ and $t\in [0,T]$. Using (\ref{def of mpjp}), we infer that for each $j\in \N_{>0}$, $x\in \R$ and $t\in [0,T]$, the function $k \mapsto m_{j}(x,t,k)$ is analytic on $\mathrm{int}\, K$; the uniform convergence then implies that $k \mapsto m(x,t,k)$ is analytic on $\mathrm{int}\, K$. Since $K$ is arbitrary, $k \mapsto m(x,t,k)$ is analytic on $\C\setminus (\overline{\Gamma_{3'}\cup \Gamma_{6'} \cup \Gamma_{2''}\cup \Gamma_{5''}})$ for each $x\in \R$ and $t\in [0,T]$. The fact that $m$ satisfies \eqref{volterra} directly follows by substituting the definition \eqref{def of n3}. This finishes the proof of (a).
 
The symmetry \eqref{F symm new} together with \eqref{def of mpjp} implies that $m_{j}(x,t,k)=m_{j}(x,t,k^{-1})$ for all $x\in \R$, $t\in [0,T]$, $k \in \C\setminus (\overline{\Gamma_{3'}\cup \Gamma_{6'} \cup \Gamma_{2''}\cup \Gamma_{5''}})$ and $j\in \N_{>0}$. Now \eqref{n3 sym} directly follows from \eqref{def of n3}. This finishes the proof of (d).
 
Let us now turn to the proof of (b). Since $m$ satisfies \eqref{volterra} and \eqref{n3 sym}, it also satisfies \eqref{lol14}, i.e.
\begin{align}
& m(x,t,k) = 1+\int_{\Gamma_{3'}\cup \Gamma_{6'}} \frac{\tilde{g}(\omega^{2} k_{1})}{k_{1}-k} m(x,t,\omega k_{1}) \frac{dk_{1}}{2\pi i} + \int_{\Gamma_{5''}\cup \Gamma_{2''}} \frac{\hat{g}(\omega k_{1})}{k_{1}-k} m(x,t,\omega^{2} k_{1}) \frac{dk_{1}}{2\pi i}. \label{lol13}
\end{align}
By (a), $m$ is analytic for $k\in \C\setminus (\overline{\Gamma_{3'}\cup \Gamma_{6'} \cup \Gamma_{2''}\cup \Gamma_{5''}})$, and by \eqref{def of h1 h2}, \eqref{gtilde def}, \eqref{ghat def},  $f_{2}(k) := \tilde{r}(k) \overline{f_{1}(\bar{k}^{-1})}$, \eqref{def of tilde r} and \eqref{f1 is Cinf}, we have $\tilde{g}(\omega^{2} \cdot)\in C^{\infty}(\Gamma_{3'}\cup \Gamma_{6'})$ and $\hat{g}(\omega \cdot)\in C^{\infty}(\Gamma_{5''}\cup \Gamma_{2''})$. In particular, $m$ is the Cauchy transform of a smooth function on $\Gamma_{3'}\cup \Gamma_{6'} \cup \Gamma_{2''}\cup \Gamma_{5''}$. Thus, by the Plemelj--Privalov theorem, the limits of $m(x,t,k)$ as $k$ approaches $\Gamma_{3'}\cup \Gamma_{6'} \cup \Gamma_{2''}\cup \Gamma_{5''}$ from the left and right exist and are (H\"{o}lder) continuous on $\Gamma_{3'}\cup \Gamma_{6'} \cup \Gamma_{2''}\cup \Gamma_{5''}$. This proves (b).

By \eqref{f1 is 0 on i 2 to i} and \eqref{f1 is Cinf}, $f_{1}\in C^{\infty}((0,i))$ and $f_{1}$ vanishes to all orders at $i$. Recall also from (a) that $m$ is analytic for $k\in \C\setminus (\overline{\Gamma_{3'}\cup \Gamma_{6'} \cup \Gamma_{2''}\cup \Gamma_{5''}})$. Hence it follows from \eqref{lol13} and \cite[Lemma 4.7]{CLmain} that $m(x,t,k)=O(1)$ as $k \to k_{\star}\in \{e^{-\frac{5\pi i}{6}}, e^{-\frac{\pi i}{6}}, e^{\frac{\pi i}{6}}, e^{\frac{5 \pi i}{6}}\}$. This proves (c).

By (a), $m(x,t,\omega \cdot)$ is analytic in a neighborhood of $\Gamma_{3'}\cup \Gamma_{6'}$ and $m(x,t,\omega^{2} \cdot)$ is analytic in a neighborhood of $\Gamma_{5''}\cup \Gamma_{2''}$ (recall that $\Gamma_{5''}\cup \Gamma_{2''}$ does not include $0$). Inequalities \eqref{lol12} and \eqref{phi0jestimate} imply 
\begin{align}\label{bound for mj on Gammat}
\sup_{k\in \tilde{\Gamma}_{\mathrm{m}}} |m_{j}(x,t,k)| \leq M^{-j}, \qquad j\in \N_{>0}, \; x\in \R, \; t\in [0,T],
\end{align}
and similarly we can show that
\begin{align*}
\sup_{k\in \omega\tilde{\Gamma}_{\mathrm{m}}} |m_{j}(x,t,k)| \leq M^{-j}, \qquad j\in \N_{>0}, \; x\in \R, \; t\in [0,T].
\end{align*}
By \eqref{def of n3}, we thus have 
\begin{align*}
\sup_{k\in \tilde{\Gamma}_{\mathrm{m}}\cup \omega\tilde{\Gamma}_{\mathrm{m}}} |m(x,t,k)| \leq 2, \qquad x\in \R, \; t\in [0,T].
\end{align*}
Hence, by \eqref{gtilde def}--\eqref{ghat def} and \eqref{lol15}, for $t\in [0,T)$ the function $\tilde{g}(\omega^{2} \cdot) m(x,t,\omega \cdot) \in C^{\infty}(\Gamma_{3'}\cup \Gamma_{6'})$ has fast decay at $\infty$, and the function $\hat{g}(\omega \cdot) m(x,t,\omega^{2} \cdot) \in C^{\infty}(\Gamma_{5''}\cup \Gamma_{2''})$ has fast decay at $0$. It follows from \cite[Lemma 4.7]{CLmain} and \eqref{lol13} that, for each fixed $x\in \R$ and $t\in [0,T]$,
\begin{align}\label{m near 0}
m(x,t,k) = C_{0} + O(k), \qquad \mbox{as } k \to 0,
\end{align}
uniformly for $\arg k \in [0,2\pi]$, where
\begin{align}
C_{0} & := 1+\int_{\Gamma_{3'}\cup \Gamma_{6'}} \frac{\tilde{g}(\omega^{2} k_{1})}{k_{1}} m(x,t,\omega k_{1}) \frac{dk_{1}}{2\pi i} + \int_{\Gamma_{5''}\cup \Gamma_{2''}} \frac{\hat{g}(\omega k_{1})}{k_{1}} m(x,t,\omega^{2} k_{1}) \frac{dk_{1}}{2\pi i} \nonumber \\
& = 1+\int_{\Gamma_{3'}} \frac{r_{1}(\frac{1}{\omega^{2} k_{1}})e^{\theta_{21}(\omega^{2} k_{1})}}{k_{1}} m(\omega k_{1}) \frac{dk_{1}}{2\pi i}+\int_{\Gamma_{6'}} \frac{r_{2}(\omega^{2} k_{1})e^{\theta_{21}(\omega^{2} k_{1})}}{k_{1}} m(\omega k_{1}) \frac{dk_{1}}{2\pi i} \nonumber \\
& - \int_{\Gamma_{5''}} \frac{r_{1}(\omega k_{1})e^{-\theta_{21}(\omega k_{1})}}{k_{1}} m(\omega^{2} k_{1}) \frac{dk_{1}}{2\pi i}- \int_{\Gamma_{2''}} \frac{r_{2}(\tfrac{1}{\omega k_{1}})e^{-\theta_{21}(\omega k_{1})}}{k_{1}} m(\omega^{2} k_{1}) \frac{dk_{1}}{2\pi i}. \label{def of C0}
\end{align}
For the last equality in \eqref{def of C0} we have used \eqref{ghat def}. Changing variables $k\to \frac{1}{k}$ and using \eqref{n3 sym} and $\theta_{21}(k^{-1})=-\theta_{21}(k)$, we obtain
\begin{align*}
& \int_{\Gamma_{5''}} r_{1}(\omega k_{1})e^{-\theta_{21}(\omega k_{1}) } m(\omega^{2}k_{1}) \frac{dk_{1}}{k_{1}} = \int_{\Gamma_{3'}} r_{1}(\tfrac{1}{\omega^{2} k_{1}})e^{\theta_{21}(\omega^{2} k_{1})} m(\omega k_{1}) \frac{dk_{1}}{k_{1}}, \\
& \int_{\Gamma_{2''}} r_{2}(\tfrac{1}{\omega k_{1}})e^{-\theta_{21}(\omega k_{1})} m(\omega^{2}k_{1}) \frac{dk_{1}}{k_{1}} = \int_{\Gamma_{6'}} r_{2}(\omega^{2} k_{1})e^{\theta_{21}(\omega^{2} k_{1})} m(\omega k_{1}) \frac{dk_{1}}{k_{1}},
\end{align*}
and substituting the above in \eqref{def of C0} yields $C_{0}=1$. Thus \eqref{m near 0} becomes $m(x,t,k) = 1 + O(k)$ as $k\to 0$ uniformly for $\arg k \in [0,2\pi]$, and \eqref{m near inf} now directly follows from \eqref{n3 sym}. This proves (e).
\end{proof}
%

\begin{lemma}\label{lemma:n in terms of m}
Let $r_1:\hat{\Gamma}_{1} \to \C$ and $r_2: \hat{\Gamma}_{4}\setminus \{\omega^{2},-\omega^{2}\} \to \C$ be as in Proposition \ref{prop:n3 as a series via volterra}. Then for each $x\in \R$ and $t\in [0,T)$, the unique solution of RH problem \ref{RHn modified} is given by
\begin{align}\label{def of n in terms of m}
n(x,t,k) = (m(x,t,\omega k),m(x,t,\omega^{2}k),m(x,t,k)), \quad k \in \C \setminus (\Gamma_{0}\cup \Gamma_{\star}),
\end{align}
where $m$ is defined in \eqref{def of n3}.
\end{lemma}
\begin{proof}
Let $\tilde{n}$ denote the right-hand side of \eqref{def of n in terms of m}, and let $n$ be the unique solution of RH problem \ref{RHn modified}. We must prove $\tilde{n}=n$. Since the solution to RH problem \ref{RHn modified} is unique, it suffices to check that $\tilde{n}$ satisfies properties $(\ref{RHnitema mod})$--$(\ref{RHniteme mod})$ of RH problem \ref{RHn modified}. 

Items (a), (c) and (e) of Proposition \ref{prop:n3 as a series via volterra} directly imply that $\tilde{n}$ satisfies properties $(\ref{RHnitema mod})$, $(\ref{RHnitemc mod})$ and $(\ref{RHniteme mod})$ of RH problem \ref{RHn modified}, respectively.

Since $\tilde{n}_{3}(x,t,k)=\tilde{n}_{1}(x,t,\omega^{2}k)=\tilde{n}_{2}(x,t,\omega k)$, $\tilde{n}$ obeys $\tilde{n}(x,t,k) = \tilde{n}(x,t,\omega k)\mathcal{A}^{-1}$, and since $m$ satisfies \eqref{n3 sym}, we also have $\tilde{n}(x,t,k) = \tilde{n}(x,t,k^{-1}) \mathcal{B}$. This shows that $\tilde{n}$ satisfies property $(\ref{RHnitemd mod})$ of RH problem \ref{RHn modified}.

Proposition \ref{prop:n3 as a series via volterra} (b) implies that the limits of $\tilde{n}(x,t,k)$ as $k$ approaches $\Gamma_{0}$ from the left and right exist and are continuous on $\Gamma_{0}$, and Proposition \ref{prop:n3 as a series via volterra} (a) implies that $m$ satisfies 
\begin{multline*}
m(x,t,k) = 1+\int_{\tilde{\Gamma}} F(x,t,k,k_{1}) m(x,t,k_{1})dk_{1}, \\ x\in \R, \; t\in [0,T], \; k\in \C\setminus (\overline{\Gamma_{3'}\cup \Gamma_{6'} \cup \Gamma_{2''}\cup \Gamma_{5''}}),
\end{multline*}
from which
\begin{align*}
\tilde{n}_{+}(x,t,k) = \tilde{n}_{-}(x,t,k)v(x,t,k), \qquad  k \in \Gamma_{0},
\end{align*}
follows. This shows that $\tilde{n}$ satisfies property (b) of RH problem \ref{RHn modified}.

Hence $\tilde{n}$ is a solution of RH problem \ref{RHn modified}. By uniqueness of the solution, $\tilde{n}=n$, which finishes the proof.
\end{proof}

\section{Proof of Theorem \ref{thm:main}}\label{section:proof1}
In this section, $r_1:\hat{\Gamma}_{1} \to \C$ and $r_2: \hat{\Gamma}_{4}\setminus \{\omega^{2},-\omega^{2}\} \to \C$ are as in Proposition \ref{prop:n3 as a series via volterra}, and $m$ is defined by \eqref{def of n3}.

Lemma \ref{lemma:n in terms of m} together with Theorem \ref{inverseth new} implies that
\begin{align}\label{def of mp1p}
m^{(1)}(x,t) := \lim_{k\to \infty} k (m(x,t,k) -1)
\end{align}
is well-defined and smooth for $(x,t) \in \R \times [0,T)$. Moreover, $u(x,t)$ defined by
\begin{align}\label{recoveruvn 3}
u(x,t) = -i\sqrt{3}\frac{\partial}{\partial x}m^{(1)}(x,t),
\end{align}
is a Schwartz class solution of (\ref{boussinesq}) on $\R \times [0,T)$.
\begin{lemma}
For $x\in \R$ and $t\in [0,T)$, $m^{(1)}$ can be written as
\begin{align}\label{def of mp1p as a series}
& m^{(1)}(x,t) = \sum_{j=1}^{\infty} m_{j}^{(1)}(x,t),
\end{align}
where
\begin{align}\label{def of mjp1p}
& m_{j}^{(1)}(x,t) := \int_{\tilde{\Gamma}}\dots \int_{\tilde{\Gamma}} F^{(1)}(k_{1}) F(k_{1},k_{2}) \cdots F(k_{j-1},k_{j}) dk_1 \cdots dk_j,
\end{align}
and $F^{(1)}(k_{1})=F^{(1)}(x,t,k_{1})$ is given by
\begin{align}
F^{(1)}(k_{1}) & := - \bigg( \omega^{2}-\frac{\omega}{k_{1}^{2}} \bigg)  \frac{\tilde{g}(\omega k_{1})}{2\pi i} = -e^{\frac{x-x_{0}}{2}\omega k_{1}} e^{\frac{T-t}{4}(\omega k_{1})^{2}} \bigg(  \omega^{2} - \frac{\omega}{k_{1}^{2}}   \bigg)\frac{\tilde{h}_{0}(x,t,\omega k_{1})}{2\pi i}. \label{def of Fp1p}
\end{align}
\end{lemma}
\begin{proof}
Let $\tilde{m}^{(1)}(x,t)$ be the right-hand side of \eqref{def of mp1p as a series}. We need to prove that $m^{(1)} = \tilde{m}^{(1)}$. Using \eqref{def of F}, we write
\begin{align*}
F(x,t,k,k_{1}) = \frac{F^{(1)}(x,t,k_{1})}{k} + \tilde{F}(x,t,k,k_{1}),
\end{align*}
where
\begin{align*}
\tilde{F}(x,t,k,k_{1}) & := e^{\frac{x-x_{0}}{2}\omega k_{1}} e^{\frac{T-t}{4}(\omega k_{1})^{2}} \bigg( \frac{\omega^{2}k_{1}}{(\omega^{2}k_{1}-k)k}  \omega^{2} - \frac{\frac{1}{\omega^{2}k_{1}}}{(\frac{1}{\omega^{2}k_{1}}-k)k} \frac{\omega}{k_{1}^{2}}   \bigg)\frac{\tilde{h}_{0}(x,t,\omega k_{1})}{2\pi i},
\end{align*}
so that \eqref{def of mpjp inductively} becomes
\begin{align*}
& m_{j+1}(x,t,k) = \int_{\tilde{\Gamma}} \bigg( \frac{F^{(1)}(x,t,k_{1})}{k} + \tilde{F}(x,t,k,k_{1}) \bigg) m_{j}(x,t,k_{1}) dk_{1}.
\end{align*}
To analyze the integral involving $\tilde{F}$, we split it into two parts:
\begin{align}
& \int_{\tilde{\Gamma}} \tilde{F}(x,t,k,k_{1}) m_{j}(x,t,k_{1}) dk_{1} =: \mathcal{I}_{1,j}(x,t,k) + \mathcal{I}_{2,j}(x,t,k), \nonumber \\
& \mathcal{I}_{1,j}(x,t,k) := \int_{\tilde{\Gamma}}  e^{\frac{x-x_{0}}{2}\omega k_{1}} e^{\frac{T-t}{4}(\omega k_{1})^{2}}  \frac{\omega^{2}k_{1}}{(\omega^{2}k_{1}-k)k} \tilde{h}_{0}(x,t,\omega k_{1}) \frac{\omega^{2}}{2\pi i} m_{j}(x,t,k_{1}) dk_{1}, \label{I1j} \\
& \mathcal{I}_{2,j}(x,t,k) := \int_{\tilde{\Gamma}}  e^{\frac{x-x_{0}}{2}\omega k_{1}} e^{\frac{T-t}{4}(\omega k_{1})^{2}}  \frac{-\frac{1}{\omega^{2}k_{1}}}{(\frac{1}{\omega^{2}k_{1}}-k)k} \tilde{h}_{0}(x,t,\omega k_{1})  \frac{\omega}{2\pi i k_{1}^{2}} m_{j}(x,t,k_{1}) dk_{1}. \label{I2j}
\end{align}
Assumption \eqref{f1 is 0 on i 2 to i} implies that $\tilde{\Gamma}$ can be replaced by $\tilde{\Gamma}_{\mathrm{m}}:=(\Gamma_{2'}\cup \Gamma_{5'})\cap \{k\in \C:|k| > 2\}$ in \eqref{I1j} and \eqref{I2j}.
For $k\in (1,+\infty)$ and $k_{1}\in \tilde{\Gamma}_{\mathrm{m}}$, we have $|\frac{1}{\omega^{2}k_{1}}-k| \geq k-\frac{1}{2} \geq \frac{k}{2}$ and $|\omega^{2}k_{1}-k| \geq \frac{k}{2}$. Hence, using also \eqref{relation betweens h0 and f0} and \eqref{bound for mj on Gammat}, we get
\begin{align*}
& |\mathcal{I}_{1,j}(x,t,k)| \leq \frac{M^{-j}}{k^{2}} \int_{\tilde{\Gamma}}  e^{\frac{T-t}{4}(\omega k_{1})^{2}}   |\tilde{f}_{0}(\omega k_{1})| \frac{|k_{1}|}{\pi} |dk_{1}|,  \\
& |\mathcal{I}_{2,j}(x,t,k)| \leq \frac{M^{-j}}{k^{2}} \int_{\tilde{\Gamma}}  e^{\frac{T-t}{4}(\omega k_{1})^{2}}    \frac{|\tilde{f}_{0}(\omega k_{1})|}{\pi |k_{1}|^{3}} |dk_{1}|.
\end{align*}
The above analysis together with \eqref{def of n3} and \eqref{def of mpjp} implies, for $k\in (1,+\infty)$, $x\in \R$ and $t\in [0,T)$, that
\begin{align*}
m(x,t,k) = 1 + \frac{\tilde{m}^{(1)}(x,t)}{k} + \mathcal{R}(x,t,k),
\end{align*}
where $\mathcal{R}$ satisfies (using $\sum_{j=1}^{+\infty}M^{-j} = \frac{1}{M-1}\leq 1$ since $M\geq 2$)
\begin{align*}
&  |\mathcal{R}(x,t,k)| \leq \sum_{j=1}^{+\infty} \big( |\mathcal{I}_{1,j}(x,t,k)|+|\mathcal{I}_{2,j}(x,t,k)| \big) \\
& \leq \frac{1}{k^{2}} \int_{\tilde{\Gamma}}  e^{\frac{T-t}{4}(\omega k_{1})^{2}}   |\tilde{f}_{0}(\omega k_{1})| \frac{|k_{1}|+\frac{1}{|k_{1}|^{3}}}{\pi} |dk_{1}|  = \frac{2}{k^{2}} \int_{1}^{+\infty}  e^{-\frac{T-t}{4}y^{2}}   |f_{1}(\tfrac{i}{y})| \frac{|y|+\frac{1}{|y|^{3}}}{\pi} dy.
\end{align*}
For the above last equality, we have used \eqref{lol 4 5 bis 2}, $f_{2}(k) := \tilde{r}(k) \overline{f_{1}(\bar{k}^{-1})}$, and the fact that $|\tilde{r}(k_{1})|=1$ for $k_{1}\in i \R$. In particular, for each fixed $x\in \R$ and $t\in [0,T)$, $\mathcal{R}(x,t,k) = O(k^{-2})$ as $k\to +\infty$, $k\in (1,+\infty)$. Thus, for each $x\in \R$ and $t\in [0,T)$,
\begin{align*}
\lim_{\substack{k\to +\infty\\ k \in (1,+\infty)}} k (m(x,t,k) -1) = \tilde{m}^{(1)}(x,t).
\end{align*}
Since the limit in \eqref{def of mp1p} can be computed along any path going to $\infty$, we have $m^{(1)} = \tilde{m}^{(1)}$, which finishes the proof.
\end{proof}
Using \eqref{relation betweens h0 and f0}, we can rewrite $F^{(1)}$ in \eqref{def of Fp1p} as
\begin{align}
F^{(1)}(x,t,k_{1}) & = -e^{\frac{x-x_{0}}{2}\omega k_{1}} e^{\frac{T-t}{4}(\omega k_{1})^{2}} e^{-\frac{x}{2}\frac{1}{\omega k_{1}}+\frac{t}{4}\frac{1}{(\omega k_{1})^{2}}} \frac{\tilde{f}_{0}(\omega k_{1})}{2\pi i} \bigg(  \omega^{2} - \frac{\omega}{k_{1}^{2}} \bigg), \label{Fp1p bis}
\end{align}
from which it follows that
\begin{align}\label{lol43}
\frac{\partial}{\partial x}F^{(1)}(x,t,k_{1}) = \frac{\omega k_{1}-\frac{1}{\omega k_{1}}}{2}F^{(1)}(x,t,k_{1}).
\end{align}
Similarly, using \eqref{relation betweens h0 and f0}--\eqref{def of F}, we get
\begin{align}\label{lol44}
\frac{\partial}{\partial x}F(x,t,k,k_{1}) = \frac{\omega k_{1}-\frac{1}{\omega k_{1}}}{2}F(x,t,k,k_{1}).
\end{align}
Hence, by differentiating \eqref{def of mjp1p} with respect to $x$, we obtain
\begin{multline}\label{mjp1p der}
\frac{\partial}{\partial x} m_{j}^{(1)}(x,t) = \int_{\tilde{\Gamma}}\dots \int_{\tilde{\Gamma}} F^{(1)}(k_{1}) F(k_{1},k_{2})
\cdots F(k_{j-1},k_{j}) \bigg(\sum_{p=1}^{j} \frac{\omega k_{p}-\frac{1}{\omega k_{p}}}{2} \bigg) dk_1 \cdots dk_j, \\
j\in \N_{>0}, \; x\in \R, \; t\in [0,T).
\end{multline}
\begin{lemma}\label{lemma: bound on derivative of mjp1p}
For $j\in \N_{>0}$, $x\in \R$ and $t\in [0,T)$,
\begin{align}
\bigg|\frac{\partial}{\partial x} m_{j}^{(1)}(x,t)\bigg| & \leq \frac{1}{M^{j-1}} \int_{1}^{+\infty} e^{-\frac{T-t}{4}y^{2}}|y f_{1}(\tfrac{i}{y})| \, dy \nonumber \\
& + \frac{j-1}{ M^{j-2}} \bigg( \int_{1}^{+\infty} e^{-\frac{T-t}{4}y^{2}}|f_{1}(\tfrac{i}{y})| \, dy \bigg)^{2}. \label{lol19}
\end{align}
\end{lemma}
\begin{proof}
Using \eqref{mjp1p der} and
\begin{align*}
\bigg|\frac{\omega k_{p}-\frac{1}{\omega k_{p}}}{2}\bigg| \leq |k_{p}|, \qquad k_{p}\in \tilde{\Gamma},
\end{align*}
we obtain the following estimate, valid for $j\in \N_{>0}$, $x\in \R$ and $t\in [0,T)$:
\begin{align}
& \bigg|\frac{\partial}{\partial x} m_{j}^{(1)}(x,t)\bigg| \leq \int_{\tilde{\Gamma}} |k_{1}F^{(1)}(x,t,k_{1})| \, |dk_{1}| \; \times \; \bigg( \int_{\tilde{\Gamma}} \sup_{s\in \tilde{\Gamma}} |F(x,t,s,k_{1})|\; |dk_{1}| \bigg)^{j-1} \nonumber \\
& + (j-1) \int_{\tilde{\Gamma}} |F^{(1)}(x,t,k_{1})| \, |dk_{1}| \; \times \; \bigg( \int_{\tilde{\Gamma}} \sup_{s\in \tilde{\Gamma}} |F(x,t,s,k_{1})|\; |dk_{1}| \bigg)^{j-2} \nonumber \\
&  \; \times \; \int_{\tilde{\Gamma}} \sup_{s\in \tilde{\Gamma}} |k_{1}F(x,t,s,k_{1})|\; |dk_{1}| \nonumber \\
& \leq \frac{1}{M^{j-1}} \int_{\tilde{\Gamma}} |k_{1}F^{(1)}(x,t,k_{1})| \, |dk_{1}|  + \frac{j-1}{ M^{j-2}} \int_{\tilde{\Gamma}} |F^{(1)}(x,t,k_{1})| \, |dk_{1}| \; \nonumber \\
&  \times \; \int_{\tilde{\Gamma}} \sup_{s\in \tilde{\Gamma}} |k_{1}F(x,t,s,k_{1})|\; |dk_{1}|, \label{lol16}
\end{align}
where for the last inequality we have used \eqref{lol12}. Also, by \eqref{def of Fp1p} and \eqref{relation betweens h0 and f0}, for $k_{1}\in \tilde{\Gamma}$ we have
\begin{align}
|F^{(1)}(x,t,k_{1})| & \leq  e^{\frac{T-t}{4}(\omega k_{1})^{2}}\bigg( 1 + \frac{1}{|k_{1}|^{2}}  \bigg) \frac{|\tilde{h}_{0}(x,t,\omega k_{1})|}{2\pi } \nonumber \\
& \leq e^{\frac{T-t}{4}(\omega k_{1})^{2}}\bigg( 1 + \frac{1}{|k_{1}|^{2}}  \bigg) \frac{|\tilde{f}_{0}(\omega k_{1})|}{2\pi}. \label{lol18}
\end{align}
Using \eqref{lol10}, \eqref{lol18}, \eqref{lol 4 5 bis 2}, $f_{2}(k) := \tilde{r}(k) \overline{f_{1}(\bar{k}^{-1})}$ and the fact that $|\tilde{r}(k_{1})|=1$ for $k_{1}\in i \R$, for any $q\in \N$ we get
\begin{align}
& \int_{\tilde{\Gamma}} |k_{1}^{q}F^{(1)}(x,t,k_{1})| \, |dk_{1}| \leq \frac{2}{\pi} \int_{\Gamma_{1'}} e^{\frac{T-t}{4}k_{1}^{2}}|k_{1}^{q}f_{1}(\tfrac{1}{k_{1}})| \, |dk_{1}| \leq  \int_{1}^{+\infty} e^{-\frac{T-t}{4}y^{2}}|y^{q}f_{1}(\tfrac{i}{y})| \, dy, \label{lol21} \\
& \int_{\tilde{\Gamma}} \sup_{s\in \tilde{\Gamma}} |k_{1}^{q}F(x,t,s,k_{1})|\; |dk_{1}| \leq  \int_{1}^{+\infty} e^{-\frac{T-t}{4}y^{2}}|y^{q-1}f_{1}(\tfrac{i}{y})| \, dy, \label{lol22}
\end{align}
and substituting the above in \eqref{lol16} yields the claim.
\end{proof}
Lemma \ref{lemma: bound on derivative of mjp1p} implies that the series \eqref{def of mp1p as a series} can be differentiated with respect to $x$ termwise, i.e.
\begin{align}\label{der of mp1p as a series}
& \frac{\partial}{\partial x} m^{(1)}(x,t) =  \sum_{j=1}^{\infty} \frac{\partial}{\partial x}m_{j}^{(1)}(x,t), \qquad x\in \R, \; t \in [0,T),
\end{align}
and provides the following estimate, valid for all $x\in \R$ and $t \in [0,T)$:
\begin{align}
& \sum_{j=2}^{\infty} \bigg|\frac{\partial}{\partial x}m_{j}^{(1)}(x,t)\bigg| \nonumber \\
& \leq \frac{1}{M-1} \int_{1}^{+\infty} e^{-\frac{T-t}{4}y^{2}}|yf_{1}(\tfrac{i}{y})| \, dy + \frac{M^{2}}{(M-1)^{2}}\bigg( \int_{1}^{+\infty} e^{-\frac{T-t}{4}y^{2}}|f_{1}(\tfrac{i}{y})| \, dy \bigg)^{2}. \label{estimate for the remainder}
\end{align}
Also, by combining \eqref{mjp1p der} with $j=1$, \eqref{Fp1p bis} and \eqref{lol 4 5 bis 2}, for $x\in \R$ and $t \in [0,T)$ we get
\begin{align*}
& \frac{\partial}{\partial x} m_{1}^{(1)}(x,t) = \int_{\tilde{\Gamma}} F^{(1)}(k_{1}) \frac{\omega k_{1}-\frac{1}{\omega k_{1}}}{2}  dk_1 \\
& = - \int_{\tilde{\Gamma}} e^{\frac{x-x_{0}}{2}\omega k_{1}} e^{\frac{T-t}{4}(\omega k_{1})^{2}} e^{-\frac{x}{2}\frac{1}{\omega k_{1}}+\frac{t}{4}\frac{1}{(\omega k_{1})^{2}}} \frac{\tilde{f}_{0}(\omega k_{1})}{2\pi i} \bigg(  \omega^{2} - \frac{\omega}{k_{1}^{2}} \bigg) \frac{\omega k_{1}-\frac{1}{\omega k_{1}}}{2}  dk_1 \\
& = - \int_{\Gamma_{2'}} e^{\frac{x-x_{0}}{2}\omega k_{1}} e^{\frac{T-t}{4}(\omega k_{1})^{2}} e^{-\frac{x}{2}\frac{1}{\omega k_{1}}+\frac{t}{4}\frac{1}{(\omega k_{1})^{2}}} \frac{f_{2}(\omega k_{1})}{2\pi i}\bigg(  \omega^{2} - \frac{\omega}{k_{1}^{2}} \bigg) \frac{\omega k_{1}-\frac{1}{\omega k_{1}}}{2}  dk_1 \\
& - \int_{\Gamma_{5'}} e^{\frac{x-x_{0}}{2}\omega k_{1}} e^{\frac{T-t}{4}(\omega k_{1})^{2}} e^{-\frac{x}{2}\frac{1}{\omega k_{1}}+\frac{t}{4}\frac{1}{(\omega k_{1})^{2}}} \frac{f_{1}(\tfrac{1}{\omega k_{1}})}{2\pi i} \bigg(  \omega^{2} -  \frac{\omega}{k_{1}^{2}} \bigg) \frac{\omega k_{1}-\frac{1}{\omega k_{1}}}{2}  dk_1,
\end{align*}
and then a change of variables yields
\begin{align*}
\frac{\partial}{\partial x} m_{1}^{(1)}(x,t) & = - \int_{1}^{+\infty} e^{\frac{x-x_{0}}{2}i y} e^{-\frac{T-t}{4}y^{2}} e^{\frac{x}{2}\frac{i}{y}-\frac{t}{4}\frac{1}{y^{2}}} \frac{f_{2}(iy)}{2\pi i}\bigg(  \omega^{2} + \frac{1}{y^{2}} \bigg) \frac{iy-\frac{1}{iy}}{2} \frac{i}{\omega} dy \\
& - \int_{1}^{+\infty} e^{-\frac{x-x_{0}}{2}iy} e^{-\frac{T-t}{4} y^{2}} e^{-\frac{x}{2}\frac{i}{y}-\frac{t}{4}\frac{1}{y^{2}}} \frac{f_{1}(\tfrac{i}{y})}{2\pi i} \bigg( \omega^{2} + \frac{1}{y^{2}} \bigg) \frac{-iy+\frac{1}{iy}}{2} \frac{-i}{\omega} dy.
\end{align*}
Using $f_{2}(iy) = \tilde{r}(iy) \overline{f_{1}(\frac{i}{y})}$ and \eqref{def of tilde r}, the above integrals can be combined as
\begin{align*}
\frac{\partial}{\partial x} m_{1}^{(1)}(x,t)
= \frac{1}{\pi i} \int_{1}^{+\infty} e^{-\frac{T-t}{4} y^{2}} e^{-\frac{t}{4}\frac{1}{y^{2}}} \frac{y+\frac{1}{y}}{2} \, \re \bigg[ e^{-\frac{x-x_{0}}{2}iy} e^{-\frac{x}{2}\frac{i}{y}} f_{1}(\tfrac{i}{y}) \bigg( \omega + \frac{\omega^{2}}{y^{2}} \bigg) \bigg] dy.
\end{align*}
Since $u(x,t) := -i\sqrt{3}\frac{\partial}{\partial x}m^{(1)}(x,t)$, we have just proved the following.
\begin{theorem}\label{thm:inside the proof}
$u$ can be written as
\begin{align}
u(x,t) & = \frac{-\sqrt{3}}{\pi} \int_{1}^{+\infty} e^{-\frac{T-t}{4} y^{2}} e^{-\frac{t}{4}\frac{1}{y^{2}}} \frac{y+\frac{1}{y}}{2} \, \re \bigg[ e^{-\frac{x-x_{0}}{2}iy} e^{-\frac{x}{2}\frac{i}{y}} f_{1}(\tfrac{i}{y}) \bigg( \omega + \frac{\omega^{2}}{y^{2}} \bigg) \bigg] dy \nonumber \\
& - i \sqrt{3} \sum_{j=2}^{\infty} \frac{\partial}{\partial x}m_{j}^{(1)}(x,t), \qquad x\in \R, \; t\in [0,T), \label{lol39}
\end{align}
and the estimate \eqref{estimate for the remainder} holds for all $(x,t)\in \R \times [0,T)$.
\end{theorem}

For convenience, we introduce the notation 
\begin{align*}
\ell:= \frac{2}{\sqrt{T-t}}.
\end{align*}
\begin{lemma}\label{lemma:precise estimates for u}
Let $p\in \N$, $\vec{r} \in \N_{\mathrm{ord}}^{p}$, $\vec{a} \in \R_{+}^{p}$, $\sigma\in \{0,1\}$, and define $\mathrm{LOG}(y) = \mathrm{LOG}(y;\vec{r},\vec{a},\sigma)$ as in \eqref{def of LOG}. Suppose that $f_{1}$ is as in Proposition \ref{prop:n3 as a series via volterra}. Assume furthermore that $f_{1}(\frac{i}{y}) = y^{\eta} \mathrm{LOG}(y)$ for all $y\geq y_{0}$ for some $y_{0}>1$, and where $\eta\in (-2,0)$. Then, as $t\to T$, and uniformly for $x\in \R$,
\begin{multline}
\frac{-\sqrt{3}}{\pi} \int_{1}^{+\infty} e^{-\frac{y^{2}}{\ell^{2}} } e^{-\frac{t}{4}\frac{1}{y^{2}}} \frac{y+\frac{1}{y}}{2} \, \re \bigg[ e^{-\frac{x_{0}}{2}\frac{i}{y}} f_{1}(\tfrac{i}{y}) \bigg( \omega + \frac{\omega^{2}}{y^{2}} \bigg) \bigg] dy \\
= (1+o(1)) \frac{\sqrt{3}}{4\pi}  \ell^{\eta+2} \mathrm{LOG}(\ell) \int_{0}^{+\infty} e^{-y^{2}} y^{\eta+1}dy \label{lol20}
\end{multline}
and 
\begin{align}\label{lol34}
\sqrt{3} \sum_{j=2}^{\infty} \bigg|\frac{\partial}{\partial x}m_{j}^{(1)}(x,t)\bigg| \leq \frac{\sqrt{3}}{M-1} (1+o(1)) \ell^{\eta+2} \mathrm{LOG}(\ell) \int_{0}^{+\infty} e^{-y^{2}} y^{\eta+1}dy.
\end{align}
Moreover, there exists $C>0$ such that, for any $x\in \R\setminus \{x_{0}\}$ and $t\in [0, T)$,
\begin{align}
& \sqrt{3} \sum_{j=1}^{\infty} \bigg|\frac{\partial}{\partial x}m_{j}^{(1)}(x,t)\bigg| \leq C \frac{1+x^{2}}{(x-x_{0})^{2}}. \label{lol36}
\end{align}
If $\eta \in (-2,-1)$, then \eqref{lol36} can be replaced by
\begin{align}\label{lol36 bis}
\sqrt{3} \sum_{j=2}^{\infty} \bigg|\frac{\partial}{\partial x}m_{j}^{(1)}(x,t)\bigg| \leq C \frac{1+|x|}{|x-x_{0}|}.
\end{align}
\end{lemma}
\begin{proof}
By \eqref{L1 norm is less than 1 2}, $\int_{1}^{+\infty} \frac{|f_{1}(\frac{i}{y})|}{y}dy<+\infty$. Thus, as $t \to T$,
\begin{align}
& \frac{-\sqrt{3}}{\pi} \int_{1}^{+\infty} e^{-\frac{y^{2}}{\ell^{2}} } e^{-\frac{t}{4}\frac{1}{y^{2}}} \frac{y+\frac{1}{y}}{2} \, \re \bigg[ e^{-\frac{x_{0}}{2}\frac{i}{y}} f_{1}(\tfrac{i}{y}) \bigg( \omega + \frac{\omega^{2}}{y^{2}} \bigg) \bigg] dy \nonumber \\
& = \frac{-\sqrt{3}}{\pi} \int_{1}^{+\infty} e^{-\frac{y^{2}}{\ell^{2}} } e^{-\frac{t}{4}\frac{1}{y^{2}}} \frac{y}{2} \, \re \bigg[ e^{-\frac{x_{0}}{2}\frac{i}{y}} f_{1}(\tfrac{i}{y})  \omega \bigg] dy + O(1). \label{lol28} 
\end{align}
Let us split the integral on the right-hand side of \eqref{lol28} according to
\begin{align*}
\int_{1}^{+\infty} = \int_{1}^{L} + \int_{L}^{+\infty},
\end{align*}
where $L=L(t):=\frac{\ell}{\log \ell}$. Since $\vec{a} \in \R_{+}^{p}$, $y\mapsto \mathrm{LOG}(y)$ is non-decreasing for all sufficiently large $y>0$. Using also that $f_{1}(\frac{i}{y}) = y^{\eta} \mathrm{LOG}(y)$ for all $y\geq y_{0}$ and for some $\eta\in (-2,0)$, the integral over $(1,L)$ can be upper bounded as follows
\begin{multline}\label{lol41}
\bigg| \frac{-\sqrt{3}}{\pi} \int_{1}^{L} e^{-\frac{y^{2}}{\ell^{2}} } e^{-\frac{t}{4}\frac{1}{y^{2}}} \frac{y}{2} \, \re \bigg[ e^{-\frac{x_{0}}{2}\frac{i}{y}} f_{1}(\tfrac{i}{y})  \omega \bigg] dy \bigg| \leq \frac{\sqrt{3}}{2\pi} \int_{1}^{L} y |f_{1}(\tfrac{i}{y})| dy \\
\leq \frac{\sqrt{3}}{2\pi} \frac{L^{\eta + 2}}{\eta+2} \mathrm{LOG}(L) \big( 1+o(1) \big), \qquad \mbox{as } t \to T,
\end{multline}
and the integral over $(L,+\infty)$ admits the following expansion as $t\to T$:
\begin{align}
& \frac{-\sqrt{3}}{\pi} \int_{L}^{+\infty} e^{-\frac{y^{2}}{\ell^{2}} } e^{-\frac{t}{4}\frac{1}{y^{2}}} \frac{y}{2} \, \re \bigg[ e^{-\frac{x_{0}}{2}\frac{i}{y}} f_{1}(\tfrac{i}{y})  \omega \bigg] dy \nonumber \\
& = \frac{-\sqrt{3}}{\pi} \int_{L}^{+\infty} e^{-\frac{y^{2}}{\ell^{2}} } \frac{y}{2} \, \re \bigg[ f_{1}(\tfrac{i}{y})  \omega \big(1+O(y^{-1})\big) \bigg] dy \nonumber \\
& = (1+o(1)) \frac{\sqrt{3}}{4\pi}  \int_{L}^{+\infty} e^{-\frac{y^{2}}{\ell^{2}} } y^{\eta+1}\mathrm{LOG}(y)dy = (1+o(1))\frac{\sqrt{3}}{4\pi}  \ell^{\eta+2} \int_{L/\ell}^{+\infty} e^{-y^{2}} y^{\eta+1}\mathrm{LOG}(\ell y)dy \nonumber \\
& = (1+o(1))\frac{\sqrt{3}}{4\pi}  \ell^{\eta+2} \bigg( \mathrm{LOG}(\ell) \int_{L/\ell}^{+\infty} e^{-y^{2}} y^{\eta+1}dy + \mathcal{J} \bigg), \nonumber \\
& = (1+o(1))\frac{\sqrt{3}}{4\pi}  \ell^{\eta+2} \bigg( \mathrm{LOG}(\ell) \int_{0}^{+\infty} e^{-y^{2}} y^{\eta+1}dy + \mathcal{J} \bigg), \label{lol29}
\end{align}
where $\mathcal{J}:=\int_{L/\ell}^{+\infty} e^{-y^{2}} y^{\eta+1}\big( \mathrm{LOG}(\ell y)-\mathrm{LOG}(\ell) \big)dy$. If $p=0$ or if $\vec{a}=(0,\ldots,0)$, then $\mathrm{LOG}(s) \equiv (-1)^{\sigma}$ and $\mathcal{J}=0$, and \eqref{lol20} directly follows by combining \eqref{lol29} with \eqref{lol28} and \eqref{lol41}. 

Suppose now that $p\in \N_{>0}$ and $a_{j}>0$ for some $j\in \{1,\ldots,p\}$. We write $\mathcal{J} = \mathcal{J}_{1}+\mathcal{J}_{2}$, where 
\begin{align*}
& \mathcal{J}_{1} := \int_{L/\ell}^{\log \ell} e^{-y^{2}} y^{\eta+1}\big( \mathrm{LOG}(\ell y)-\mathrm{LOG}(\ell) \big)dy,\\
& \mathcal{J}_{2} := \int_{\log \ell}^{+\infty} e^{-y^{2}} y^{\eta+1}\big( \mathrm{LOG}(\ell y)-\mathrm{LOG}(\ell) \big)dy.
\end{align*}
For each $r\in \N_{>0}$, define $\exp_{r}:\R\to \R$ by
\begin{align*}
\exp_{r}(y):=\exp(\exp_{r-1}(y)), \qquad \exp_{0}(y):=y,
\end{align*}
and for each $r\in \N_{>0}$ and $\ell > \exp_{r}(0)$, define also $z_{r}:(-\log_{r}\ell,+\infty)\to \R$ by
\begin{align*}
z_{r}(y):=\log(1+\tfrac{y}{\log_{r}\ell}), 
\end{align*}
which is strictly increasing on its domain of definition. Let $\circ$ denote function composition. Iterating the identity $\log(ab) = \log a + \log b$, we get
\begin{multline}\label{lol30}
\log_{r}(\ell y) = \log_{r} \ell + z_{r-1} \circ z_{r-2} \circ \dots \circ z_{1}(\log y), \\
\mbox{for all } r\geq2, \; \ell > \exp_{r-1}(0), \; y>\exp_{r-1}(0)/\ell.
\end{multline}
For $\ell>\exp_{r+1}(0)$, we have $\log_{r} \ell >1$ and thus $z_{r}(y) \leq \log (1+y) \leq y$ for all $y \geq 0$. Hence
\begin{multline*}
\log_{r}(\ell y) \leq (\log_{r}\ell) \; \bigg( 1+\frac{\log y}{\log_{r} \ell} \bigg) \leq (\log_{r}\ell) \; \big( 1+\log y \big), \\
\mbox{for all } r\in \N_{>0}, \; \ell>\exp_{r+1}(0), \; y \geq e.
\end{multline*} 
Since $a_{1},\ldots,a_{p}\in \R_{+}$, by \eqref{def of LOG} we also have
\begin{align*}
(-1)^{\sigma}\mathrm{LOG}(\ell y) \leq (-1)^{\sigma}\mathrm{LOG}(\ell) \prod_{j=1}^{p} \Big( 1+\log y \Big)^{a_{j}}, \qquad \mbox{for all } \ell > \exp_{1+r_{p}}(0), \; y \geq e,
\end{align*}
which implies
\begin{align}\label{J2 bound}
|\mathcal{J}_{2}| \leq |\mathrm{LOG}(\ell)| \int_{\log \ell}^{+\infty} e^{-y^{2}} y^{\eta+1}\bigg( \prod_{j=1}^{p} \Big( 1+\log y \Big)^{a_{j}} + 1 \bigg)dy = O \bigg( \frac{|\mathrm{LOG}(\ell)|}{e^{\frac{1}{2}(\log \ell)^{2}}} \bigg)
\end{align}
as $\ell \to + \infty$. 
Equality \eqref{lol30} also implies that for all $y=y(\ell)>0$ such as $\frac{\log y}{\log \ell} \to 0$ as $\ell \to +\infty$, we have
\begin{align}\label{iterated log asymp}
\log_{r}(\ell y) = \log_{r} (\ell) \; \bigg(1 + O\bigg(\frac{\log y}{\prod_{n=1}^{r}\log_{n} \ell} \bigg) \bigg),
\end{align}
and thus
\begin{align}\label{lol32}
\mathrm{LOG}(\ell y) = \mathrm{LOG}(\ell) \bigg(1 + O\bigg(\frac{\log y}{\log \ell} \bigg) \bigg), \qquad \mbox{as } \ell \to + \infty
\end{align}
holds uniformly for $y \in (L/\ell,\log \ell)$ (recall $L/\ell = 1/\log \ell$). The asymptotic formula \eqref{lol32} yields the estimate
\begin{align}\label{J1 bound}
|\mathcal{J}_{1}| \leq  \frac{C \; \mathrm{LOG}(\ell)}{\log \ell} \int_{\frac{1}{\log \ell}}^{\log \ell} e^{-y^{2}} y^{\eta+1}(\log y) \; dy = O\bigg( \frac{\mathrm{LOG}(\ell)}{\log \ell} \bigg).
\end{align}
By \eqref{J2 bound} and \eqref{J1 bound}, we have 
\begin{align*}
|\mathcal{J}| = O\bigg( \frac{\mathrm{LOG}(\ell)}{\log \ell} \bigg),
\end{align*}
and substituting this in \eqref{lol29} yields
\begin{multline}
\frac{-\sqrt{3}}{\pi} \int_{L}^{+\infty} e^{-\frac{y^{2}}{\ell^{2}} } e^{-\frac{t}{4}\frac{1}{y^{2}}} \frac{y}{2} \, \re \bigg[ e^{-\frac{x}{2}\frac{i}{y}} f_{1}(\tfrac{i}{y})  \omega \bigg] dy \\ =  (1+o(1)) \frac{\sqrt{3}}{4\pi}  \ell^{\eta+2} \mathrm{LOG}(\ell) \int_{0}^{+\infty} e^{-y^{2}} y^{\eta+1}dy. \label{lol31}
\end{multline}
Now \eqref{lol20} directly follows by combining \eqref{lol31} with \eqref{lol28} and \eqref{lol41}.

Since $f_{1}(\frac{i}{y}) = y^{\eta} \mathrm{LOG}(y) \in \R$ for all $y\geq y_{0}$, we obtain (in a similar way as for \eqref{lol20}) 
\begin{align}
\int_{1}^{+\infty} e^{-\frac{y^{2}}{\ell^{2}}}|y f_{1}(\tfrac{i}{y})| \, dy & = (1+o(1)) \int_{L}^{+\infty} e^{-\frac{y^{2}}{\ell^{2}}}y^{\eta+1} \mathrm{LOG}(y) \, dy \nonumber \\
& = (1+o(1)) \ell^{\eta+2} \mathrm{LOG}(\ell) \int_{0}^{+\infty} e^{-y^{2}} y^{\eta+1}dy, \label{lol17} \\
\int_{1}^{+\infty} e^{-\frac{y^{2}}{\ell^{2}}}| f_{1}(\tfrac{i}{y})| \, dy & = O\big(1+\ell^{\eta+1}\mathrm{LOG}(\ell)\big). \label{lol33}
\end{align}
Substituting \eqref{lol17} and \eqref{lol33} in \eqref{estimate for the remainder}, we find \eqref{lol34}.

We now turn to the proof of \eqref{lol36}. By \eqref{mjp1p der}, for $j\in \N_{>0}$, $x\in \R$ and $t\in [0,T)$,
\begin{align*}
& \frac{\partial}{\partial x} m_{j}^{(1)}(x,t) = \sum_{p=1}^{j} (\mathcal{J}_{j,p}^{(1)}-\mathcal{J}_{j,p}^{(2)})  \\
& \mathcal{J}_{j,p}^{(1)} := \int_{\tilde{\Gamma}}\dots \int_{\tilde{\Gamma}} F^{(1)}(k_{1}) F(k_{1},k_{2})
\cdots F(k_{j-1},k_{j}) \frac{\omega k_{p}}{2}  dk_1 \cdots dk_j,  \\
& \mathcal{J}_{j,p}^{(2)} := \int_{\tilde{\Gamma}}\dots \int_{\tilde{\Gamma}} F^{(1)}(k_{1}) F(k_{1},k_{2})
\cdots F(k_{j-1},k_{j}) \frac{1}{2\omega k_{p}}  dk_1 \cdots dk_j.
\end{align*}
Let us define $G^{(1)}(k_{1})=G^{(1)}(x,t,k_{1})$ by
\begin{align}
G^{(1)}(k_{1}) & = - e^{\frac{(\omega k_{1})^{2}}{\ell^{2}}} e^{-\frac{x}{2}\frac{1}{\omega k_{1}}+\frac{t}{4}\frac{1}{(\omega k_{1})^{2}}} \frac{\tilde{f}_{0}(\omega k_{1})}{2\pi i} \bigg(  \omega^{2} - \frac{\omega}{k_{1}^{2}} \bigg), \label{Gp1p bis}
\end{align}
so that, by \eqref{Fp1p bis},
\begin{align*}
F^{(1)}(k_{1}) = e^{\frac{x-x_{0}}{2}\omega k_{1}} G^{(1)}(k_{1}).
\end{align*}
Define also
\begin{align}\label{lol42}
H_{k_{2}}^{(1)}(k_{1}) := \ell k_{1} G^{(1)}(\ell k_{1})F(\ell k_{1},k_{2}).
\end{align}
Using \eqref{lol42} and the change of variables $k_{1}\to \ell k_{1}$, $\mathcal{J}_{j,1}^{(1)}$ can be rewritten as
\begin{align*}
\mathcal{J}_{j,1}^{(1)} & = \frac{\omega}{2} \int_{\tilde{\Gamma}}\dots \int_{\tilde{\Gamma}} e^{\frac{x-x_{0}}{2}\omega k_{1}} H_{k_{2}}^{(1)}(\tfrac{k_{1}}{\ell}) F(k_{2},k_{3})
\cdots F(k_{j-1},k_{j})   dk_1 \cdots dk_j \\
& = \frac{\omega}{2} \int_{\tilde{\Gamma}/\ell} \int_{\tilde{\Gamma}}\dots \int_{\tilde{\Gamma}} e^{\frac{x-x_{0}}{2}\ell \omega k_{1}} H_{k_{2}}^{(1)}(k_{1}) F(k_{2},k_{3})
\cdots F(k_{j-1},k_{j})   \ell dk_1 \cdots dk_j.
\end{align*}
Suppose $x\neq x_{0}$. Since $H_{k_{2}}^{(1)}(k_{1})$ has fast decay as 
\begin{align*}
\tilde{\Gamma}/\ell \ni k_{1} \to k_{\star} \in \partial (\tilde{\Gamma}/\ell):=\{e^{-\frac{\pi i}{6}}/\ell,e^{-\frac{\pi i}{6}}\infty,e^{\frac{5\pi i}{6}}/\ell,e^{\frac{5\pi i}{6}}\infty \},
\end{align*}
by using twice integration by parts, we get
\begin{align*}
\mathcal{J}_{j,1}^{(1)} & =  \frac{\frac{\omega}{2}}{(\frac{x-x_{0}}{2}\ell \omega)^{2}} \int_{\tilde{\Gamma}/\ell} \int_{\tilde{\Gamma}}\dots \int_{\tilde{\Gamma}} e^{\frac{x-x_{0}}{2}\ell \omega k_{1}} (H_{k_{2}}^{(1)})''(k_{1}) F(k_{2},k_{3})
\cdots F(k_{j-1},k_{j})   \ell dk_1 \cdots dk_j \\
& =  \frac{\frac{\omega}{2}}{(\frac{x-x_{0}}{2}\ell \omega)^{2}} \int_{\tilde{\Gamma}}\dots \int_{\tilde{\Gamma}} e^{\frac{x-x_{0}}{2} \omega k_{1}} (H_{k_{2}}^{(1)})''(\tfrac{k_{1}}{\ell}) F(k_{2},k_{3}) \cdots F(k_{j-1},k_{j})  dk_1 \cdots dk_j.
\end{align*}
An explicit computation of $(H_{k_{2}}^{(1)})''(\tfrac{k_{1}}{\ell})$ then yields $\mathcal{J}_{j,1}^{(1)}=\mathcal{H}_{j,1}+\mathcal{H}_{j,2}+\mathcal{H}_{j,3}$, where
\begin{align*}
& \mathcal{H}_{j,1} := \frac{\frac{\omega}{2}}{(\frac{x-x_{0}}{2}\ell \omega)^{2}} \int_{\tilde{\Gamma}}\dots \int_{\tilde{\Gamma}} e^{\frac{x-x_{0}}{2} \omega k_{1}} \frac{d^{2}}{dy^{2}}\big(\ell y G^{(1)}(\ell y)\big)|_{y= \frac{k_{1}}{\ell}}F(k_{1},k_{2}) \cdots F(k_{j-1},k_{j})  dk_1 \cdots dk_j, \\
& \mathcal{H}_{j,2} := \frac{\frac{\omega}{2}}{(\frac{x-x_{0}}{2}\ell \omega)^{2}} \int_{\tilde{\Gamma}}\dots \int_{\tilde{\Gamma}} e^{\frac{x-x_{0}}{2} \omega k_{1}} 2 \frac{d}{dy}\big(\ell y G^{(1)}(\ell y)\big)|_{y= \frac{k_{1}}{\ell}} \frac{d}{dy}\big(F(\ell y,k_{2})\big)|_{y= \frac{k_{1}}{\ell}} \\
& \hspace{4.5cm} \times F(k_{2},k_{3}) \cdots F(k_{j-1},k_{j})  dk_1 \cdots dk_j, \\
& \mathcal{H}_{j,3} := \frac{\frac{\omega}{2}}{(\frac{x-x_{0}}{2}\ell \omega)^{2}} \int_{\tilde{\Gamma}}\dots \int_{\tilde{\Gamma}} e^{\frac{x-x_{0}}{2} \omega k_{1}} k_{1} G^{(1)}(k_{1}) \frac{d^{2}}{dy^{2}}\big(F(\ell y,k_{2})\big)|_{y= \frac{k_{1}}{\ell}} \\
& \hspace{4.5cm} \times F(k_{2},k_{3}) \cdots F(k_{j-1},k_{j})  dk_1 \cdots dk_j.
\end{align*}
Let us now estimate the integrands of $\mathcal{H}_{j,1}$, $\mathcal{H}_{j,2}$ and $\mathcal{H}_{j,3}$. Using \eqref{def of F}, we get
\begin{align*}
& \frac{d}{dy}\big(F(\ell y,k_{2})\big)|_{y= \frac{k_{1}}{\ell}} = F(k_{1},k_{2}) \frac{\ell (1-k_{1})(1+k_{1})}{k_{1}(\omega^{2}k_{2}-k_{1})(\frac{1}{\omega^{2}k_{2}}-k_{1})},
\end{align*}
and thus, for $x\in \R$, $t\in [0,T]$ and $k_{2}\in \tilde{\Gamma}_{\mathrm{m}}$,
\begin{align}
\sup_{k_{1}\in \tilde{\Gamma}_{\mathrm{m}}} \bigg|k_{1}\frac{d}{dy}\big(F(\ell y,k_{2})\big)|_{y= \frac{k_{1}}{\ell}} \bigg| & \leq \sup_{k_{1}\in \tilde{\Gamma}_{\mathrm{m}}} |F(k_{1},k_{2})| \sup_{k_{1}\in \tilde{\Gamma}_{\mathrm{m}}} \bigg| \frac{\ell (1-k_{1})(1+k_{1})}{(\omega^{2}k_{2}-k_{1})(\frac{1}{\omega^{2}k_{2}}-k_{1})} \bigg| \nonumber \\
& \leq C \ell \sup_{k_{1}\in \tilde{\Gamma}_{\mathrm{m}}} |F(k_{1},k_{2})| \label{lol38}
\end{align}
for some constant $C>0$ independent of $x$ and $t$. Similarly, increasing $C>0$ if necessary, for $x\in \R$, $t\in [0,T]$ and $k_{2}\in \tilde{\Gamma}_{\mathrm{m}}$, we obtain
\begin{align*}
\sup_{k_{1}\in \tilde{\Gamma}_{\mathrm{m}}} \bigg|k_{1}^{2}\frac{d^{2}}{dy^{2}}\big(F(\ell y,k_{2})\big)|_{y= \frac{k_{1}}{\ell}} \bigg| \leq C\ell^{2} \sup_{k_{1}\in \tilde{\Gamma}_{\mathrm{m}}} |F(k_{1},k_{2})|.
\end{align*}
Also, a direct computation shows that, for all large $\ell$ and all $k_{1}\in \tilde{\Gamma}$, $x\in \R$ and $t\in [0,T]$, 
\begin{align}
& \bigg|\frac{d}{dy}\big(\ell y G^{(1)}(\ell y)\big)|_{y= \frac{k_{1}}{\ell}}\bigg| \nonumber \\
& \hspace{2cm} \leq C \ell \bigg|\frac{d}{dk_{1}}\big(k_{1}\tilde{f}_{0}(\omega k_{1})\big)\bigg| e^{-\frac{|k_{1}|^{2}}{\ell^{2}}} + C \bigg( \frac{|k_{1}|}{\ell} + \ell \frac{(1+|x|)}{|k_{1}|^{2}} \bigg) |k_{1} \tilde{f}_{0}(\omega k_{1})| e^{-\frac{|k_{1}|^{2}}{\ell^{2}}}, \nonumber \\
& \bigg|\frac{d^{2}}{dy^{2}}\big(\ell y G^{(1)}(\ell y)\big)|_{y= \frac{k_{1}}{\ell}}\bigg| \nonumber \\
& \hspace{1.7cm} \leq C \ell^{2} \bigg|\frac{d^{2}}{dk_{1}^{2}}\big(k_{1}\tilde{f}_{0}(\omega k_{1})\big)\bigg| e^{-\frac{|k_{1}|^{2}}{\ell^{2}}} + C \bigg( |k_{1}| + \frac{\ell^{2}(1+|x|)}{|k_{1}|^{2}} \bigg) \bigg|\frac{d}{dk_{1}}\big(k_{1}\tilde{f}_{0}(\omega k_{1})\big)\bigg| e^{-\frac{|k_{1}|^{2}}{\ell^{2}}}  \nonumber \\
& \hspace{1.7cm} + C \bigg( 1 + \frac{|k_{1}|^{2}}{\ell^{2}} + \frac{1+|x|}{|k_{1}|} + \ell^{2} \frac{(1+|x k_{1}|+x^{2})}{|k_{1}|^{4}} \bigg) |k_{1} \tilde{f}_{0}(\omega k_{1})| e^{-\frac{|k_{1}|^{2}}{\ell^{2}}}, \label{lol37}
\end{align}
provided that $C$ is chosen large enough. Using the above inequalities together with \eqref{lol12}, \eqref{lol18}, and using also that $f_{1}(\frac{i}{y}) = y^{\eta} \mathrm{LOG}(y)$ for all large $y>0$, for all $x\in \R\setminus\{x_{0}\}$ and $t\in [0,T]$ we obtain
\begin{align*}
|\mathcal{H}_{j,1}| & \leq \frac{2}{(x-x_{0})^{2}\ell^{2}} \int_{\tilde{\Gamma}} \bigg|\frac{d^{2}}{dy^{2}}\big(\ell y G^{(1)}(\ell y)\big)|_{y= \frac{k_{1}}{\ell}}\bigg| \, |dk_{1}| \;  \bigg( \int_{\tilde{\Gamma}} \sup_{s\in \tilde{\Gamma}} |F(s,k_{1})|\; |dk_{1}| \bigg)^{j-1} \\
& \leq \frac{C_{2}}{3} \frac{1+x^{2}}{(x-x_{0})^{2}}M^{-(j-1)}, \\
|\mathcal{H}_{j,2}| & \leq \frac{4}{(x-x_{0})^{2}\ell^{2}} \int_{\tilde{\Gamma}} \bigg|\frac{1}{k_{1}}\frac{d}{dy}\big(\ell y G^{(1)}(\ell y)\big)|_{y= \frac{k_{1}}{\ell}}\bigg| \, |dk_{1}| \int_{\tilde{\Gamma}}\sup_{k_{1}\in \tilde{\Gamma}_{\mathrm{m}}} \bigg| k_{1}\frac{d}{dy}\big(F(\ell y,k_{2})\big)|_{y= \frac{k_{1}}{\ell}} \bigg|\, |dk_{2}| \\
& \times \bigg( \int_{\tilde{\Gamma}} \sup_{s\in \tilde{\Gamma}} |F(s,k_{1})|\; |dk_{1}| \bigg)^{j-2} \leq \frac{C_{2}}{3} \frac{1+x^{2}}{(x-x_{0})^{2}}M^{-(j-1)}, \\
|\mathcal{H}_{j,3}| & \leq \frac{2}{(x-x_{0})^{2}\ell ^{2}} \int_{\tilde{\Gamma}} \bigg|\frac{G^{(1)}(k_{1})}{k_{1}}\bigg| \, |dk_{1}| \int_{\tilde{\Gamma}} \sup_{k_{1}\in \tilde{\Gamma}_{\mathrm{m}}} \bigg| k_{1}^{2}\frac{d^{2}}{dy^{2}}\big(F(\ell y,k_{2})\big)|_{y= \frac{k_{1}}{\ell}} \bigg| \, |dk_{2}| \\
& \times \bigg( \int_{\tilde{\Gamma}} \sup_{s\in \tilde{\Gamma}} |F(s,k_{1})|\; |dk_{1}| \bigg)^{j-2} \leq \frac{C_{2}}{3} \frac{1+x^{2}}{(x-x_{0})^{2}}M^{-(j-1)},
\end{align*}
for some constant $C_{2}>0$ independent of $x$ and $t$. Hence 
\begin{align*}
|\mathcal{J}_{j,1}^{(1)}| \leq |\mathcal{H}_{j,1}|+|\mathcal{H}_{j,2}|+|\mathcal{H}_{j,3}| \leq C_{2} \frac{1+x^{2}}{(x-x_{0})^{2}}M^{-(j-1)},
\end{align*}
and thus
\begin{align*}
\sum_{j=1}^{+\infty} |\mathcal{J}_{j,1}^{(1)}| \leq C_{2} \frac{1+x^{2}}{(x-x_{0})^{2}}\frac{M}{M-1}, \qquad x\in \R\setminus\{x_{0}\}, \; t\in [0,T).
\end{align*}
A long but similar analysis (which we omit) yields 
\begin{multline*}
|\mathcal{J}_{j,p}^{(n)}| \leq C_{2} \frac{1+x^{2}}{(x-x_{0})^{2}}M^{-(j-1)}, \\ \mbox{for all } n\in\{1,2\}, \; j\in \N_{>0}, \; p \in \{1,\ldots,j\}, \; x\in \R\setminus\{x_{0}\}, \; t\in [0,T),
\end{multline*}
(increasing $C_{2}$ if necessary) and after summing over $n$, $p$ and $j$ we find \eqref{lol36}. 

If $\eta\in (-2,-1)$, then only one integration by parts is needed. For example, for $\mathcal{J}_{j,1}^{(1)}$, we use
\begin{align*}
\mathcal{J}_{j,1}^{(1)} & =  -\frac{\frac{\omega}{2}}{\frac{x-x_{0}}{2}\ell \omega} \int_{\tilde{\Gamma}/\ell} \int_{\tilde{\Gamma}}\dots \int_{\tilde{\Gamma}} e^{\frac{x-x_{0}}{2}\ell \omega k_{1}} (H_{k_{2}}^{(1)})'(k_{1}) F(k_{2},k_{3})
\cdots F(k_{j-1},k_{j})   \ell dk_1 \cdots dk_j \\
& =  -\frac{\frac{\omega}{2}}{\frac{x-x_{0}}{2}\ell \omega} \int_{\tilde{\Gamma}}\dots \int_{\tilde{\Gamma}} e^{\frac{x-x_{0}}{2} \omega k_{1}} (H_{k_{2}}^{(1)})'(\tfrac{k_{1}}{\ell}) F(k_{2},k_{3}) \cdots F(k_{j-1},k_{j})  dk_1 \cdots dk_j.
\end{align*}
We rewrite this as $\mathcal{J}_{j,1}^{(1)}=\tilde{\mathcal{H}}_{j,1}+\tilde{\mathcal{H}}_{j,2}$, where 
\begin{align*}
& \tilde{\mathcal{H}}_{j,1} := -\frac{\frac{\omega}{2}}{\frac{x-x_{0}}{2}\ell \omega} \int_{\tilde{\Gamma}}\dots \int_{\tilde{\Gamma}} e^{\frac{x-x_{0}}{2} \omega k_{1}} \frac{d}{dy}\big(\ell y G^{(1)}(\ell y)\big)|_{y= \frac{k_{1}}{\ell}} F(k_{1},k_{2}) \cdots F(k_{j-1},k_{j})  dk_1 \cdots dk_j, \\
& \tilde{\mathcal{H}}_{j,2} := -\frac{\frac{\omega}{2}}{\frac{x-x_{0}}{2}\ell \omega} \int_{\tilde{\Gamma}}\dots \int_{\tilde{\Gamma}} e^{\frac{x-x_{0}}{2} \omega k_{1}} k_{1} G^{(1)}(k_{1}) \frac{d}{dy}\big(F(\ell y,k_{2})\big)|_{y= \frac{k_{1}}{\ell}} \\
& \hspace{4.5cm} \times F(k_{2},k_{3}) \cdots F(k_{j-1},k_{j})  dk_1 \cdots dk_j.
\end{align*}
Combining \eqref{lol12}, \eqref{lol18}, \eqref{lol38} and \eqref{lol37} yields 
\begin{align*}
|\tilde{\mathcal{H}}_{j,1}| \leq \frac{C_{3}}{2}\frac{1+|x|}{|x-x_{0}|}M^{-(j-1)}, \qquad |\tilde{\mathcal{H}}_{j,2}| \leq \frac{C_{3}}{2}\frac{1+|x|}{|x-x_{0}|}M^{-(j-1)},
\end{align*}
for all $j\in \N_{>0}$, $x\in \R\setminus\{x_{0}\}$ and $t\in [0,T)$, and for some constant $C_{3}>0$ independent of $x$ and $t$. Hence 
\begin{align*}
\sum_{j=1}^{+\infty} |\mathcal{J}_{j,1}^{(1)}| \leq C_{3} \frac{1+|x|}{|x-x_{0}|}\frac{M}{M-1},
\end{align*}
for all $x\in \R\setminus\{x_{0}\}$ and $t\in [0,T]$. A similar analysis yields 
\begin{multline*}
|\mathcal{J}_{j,p}^{(n)}| \leq C_{3} \frac{1+|x|}{|x-x_{0}|}M^{-(j-1)} \\ \mbox{for all } n\in\{1,2\}, \; j\in \N_{>0}, \; p \in \{1,\ldots,j\}, \; x\in \R\setminus\{x_{0}\}, \; t\in [0,T),
\end{multline*}
(increasing $C_{3}$ if necessary) and after summing over $n$, $p$ and $j$ we find \eqref{lol36 bis}.
\end{proof}

We now prove our first main result.

\begin{proof}[Proof of Theorem \ref{thm:main}]
Let $p\in \N$, $\vec{r} \in \N_{\mathrm{ord}}^{p}$, $\vec{a} \in \R_{+}^{p}$, $\sigma\in \{0,1\}$, $T>0$, $x_{0}\in \R$ and $\delta\in (0,1)$. Define $\mathrm{LOG}(s) = \mathrm{LOG}(s;\vec{r},\vec{a},\sigma)$ as in \eqref{def of LOG}, and let $\eta := -2+2\delta \in (-2,0)$. Fix $M\geq 2$. Let $f_{1}:(0,i] \to \C$ be any function such that $f_{1} \in C^{\infty}((0,i))$, $f_{1}(k) = 0$ for $k\in [\tfrac{i}{2},i]$, and satisfying
\begin{align}
& \bigg\| \frac{f_{1}(\frac{i}{\cdot})}{\cdot} \bigg\|_{L^{1}((1,+\infty))} = \int_{1}^{+\infty} \frac{|f_{1}(\frac{i}{y})|}{y}dy \leq \frac{1}{M}, \label{L1 norm is less than 1 2 bis} \\
& f_{1}(\tfrac{i}{y}) = y^{\eta} \mathrm{LOG}(y) \mbox{ for all } y>y_{0}, \nonumber
\end{align}
for some $y_{0}>1$. Define $r_1:\hat{\Gamma}_{1} \to \C$ by
\begin{align}
& r_{1}(k) = 0, & & \mbox{for } k \in [-i,-i\infty) \cup \partial \D, \nonumber \\
& r_{1}(k) = f_{1}(k) e^{-\frac{x_{0}}{2}\frac{1}{k}} e^{\frac{T}{4k^{2}}}, & & \mbox{for } k \in (0,i], \nonumber
\end{align}
and define $r_{2}:\hat{\Gamma}_{4}\setminus \{\omega^{2}, -\omega^{2}\}\to \C$ by
\begin{align*}
r_{2}(k) = \tilde{r}(k) \overline{r_{1}(\bar{k}^{-1})} & & \mbox{for all } k \in \hat{\Gamma}_{4}\setminus \{\omega^{2}, -\omega^{2}\}.
\end{align*}
Define $m$, $m^{(1)}$ and $u$ as in \eqref{def of n3}, \eqref{def of mp1p}, and \eqref{recoveruvn 3}, respectively. Lemma \ref{lemma:n in terms of m} combined with Theorem \ref{inverseth new} implies that $u(x,t)$ is a Schwartz class solution of (\ref{boussinesq}) on $\R \times [0,T)$. Since $u(x,t) := -i\sqrt{3}\frac{\partial}{\partial x}m^{(1)}(x,t)$, by \eqref{der of mp1p as a series} we have
\begin{align*}
|u(x,t)| \leq \sqrt{3} \sum_{j=1}^{\infty} \bigg|\frac{\partial}{\partial x}m_{j}^{(1)}(x,t)\bigg|, \qquad x\in \R, \; t \in [0,T),
\end{align*}
and therefore \eqref{lol36} and \eqref{lol36 bis} imply that $u$ satisfies \eqref{bound for u in thm} and \eqref{bound for u in thm bis}. Moreover, using Theorem \ref{thm:inside the proof}, \eqref{lol20} and \eqref{lol34}, we infer that
\begin{align*}
u(x_{0},t) = (1+o(1)) \frac{\sqrt{3}}{4\pi}  \ell^{\eta+2} \mathrm{LOG}(\ell) \int_{0}^{+\infty} e^{-y^{2}} y^{\eta+1}dy + \mathcal{R}, \qquad \mbox{as } t \to T,
\end{align*}
where $\mathcal{R}$ satisfies
\begin{align*}
|\mathcal{R}| \leq \frac{\sqrt{3}}{M-1} (1+o(1)) \ell^{\eta+2} \mathrm{LOG}(\ell) \int_{0}^{+\infty} e^{-y^{2}} y^{\eta+1}dy, \qquad \mbox{as } t \to T.
\end{align*}
Choosing $M$ such that $\frac{\sqrt{3}}{M-1} < \frac{\sqrt{3}}{4\pi}$ ensures that $u(x_{0},t) \asymp \ell^{\eta+2} \mathrm{LOG}(\ell)$ as $t\to T$. Since $\mathrm{LOG}(\ell)\asymp \mathrm{LOG}(\ell^{2}/4)$ as $t\to T$, this proves \eqref{asymp for u in thm} and finishes the proof of Theorem \ref{thm:main}.
\end{proof}

\section{Proof of Theorem \ref{thm:wave breaking generalized}}\label{section:proof2}
In this section, $r_1:\hat{\Gamma}_{1} \to \C$ and $r_2: \hat{\Gamma}_{4}\setminus \{\omega^{2},-\omega^{2}\} \to \C$ are as in Proposition \ref{prop:n3 as a series via volterra}, and $m$, $m^{(1)}$ and $u$ are defined by \eqref{def of n3}, \eqref{def of mp1p} and \eqref{recoveruvn 3}, respectively.

Using \eqref{def of F} and \eqref{Fp1p bis}, we obtain
\begin{align*}
& \frac{\partial}{\partial t}F^{(1)}(x,t,k_{1}) = \frac{\frac{1}{(\omega k_{1})^{2}}-(\omega k_{1})^{2}}{4}F^{(1)}(x,t,k_{1}), \\
& \frac{\partial}{\partial t}F(x,t,k,k_{1}) = \frac{\frac{1}{(\omega k_{1})^{2}}-(\omega k_{1})^{2}}{4}F(x,t,k,k_{1}).
\end{align*}
Hence, by differentiating \eqref{def of mjp1p} $q_{1}$ times with respect to $x$ and $q_{2}$ times with respect to $t$ for some $q_{1},q_{2}\in \N$, using also \eqref{lol43} and \eqref{lol44}, we get
\begin{align}
& \frac{\partial^{q_{1}}}{(\partial x)^{q_{1}}}\frac{\partial^{q_{2}}}{(\partial t)^{q_{2}}} m_{j}^{(1)}(x,t) = \int_{\tilde{\Gamma}}\dots \int_{\tilde{\Gamma}} F^{(1)}(k_{1}) F(k_{1},k_{2})
\cdots F(k_{j-1},k_{j}) \nonumber \\
& \times  \bigg(\sum_{p=1}^{j} \frac{\omega k_{p}-\frac{1}{\omega k_{p}}}{2} \bigg)^{q_{1}} \bigg(\sum_{p=1}^{j} \frac{\frac{1}{(\omega k_{p})^{2}}-(\omega k_{p})^{2}}{4} \bigg)^{q_{2}} dk_1 \cdots dk_j, \quad j\in \N_{>0}, \; x\in \R, \; t\in [0,T). \label{mjp1p der dert}
\end{align}
The following lemma is an analog of Lemma \ref{lemma: bound on derivative of mjp1p} to estimate higher order derivatives of $u$.

\begin{lemma}\label{lemma: bound on derivative of mjp1p derxx and dert}
Let $q\in \N$. For $j\in \N_{>0}$, $x\in \R$, $t\in [0,T)$ and $q_{1},q_{2}\in \N$ such that $q_{1}+2q_{2}=q$,
\begin{align}
\bigg|\frac{\partial^{q_{1}}}{(\partial x)^{q_{1}}}\frac{\partial^{q_{2}}}{(\partial t)^{q_{2}}} m_{j}^{(1)}(x,t)\bigg| & \leq  \frac{1}{M^{j-1}} \int_{1}^{+\infty} e^{-\frac{T-t}{4}y^{2}}|y^{q}f_{1}(\tfrac{i}{y})| \, dy \nonumber \\
&  + \frac{j^{q}-1}{M^{j-q}}   \bigg(\int_{1}^{+\infty} e^{-\frac{T-t}{4}y^{2}}|y^{q-1}f_{1}(\tfrac{i}{y})| \, dy \bigg)^{q}. \label{lol19 der t and der xx}
\end{align}
Moreover, if $(\cdot)^{q} f_{1}(\frac{i}{\cdot}) \in L^{1}((1,+\infty))$, then for each $j\in \N_{>0}$, $\frac{\partial^{q_{1}}}{(\partial x)^{q_{1}}}\frac{\partial^{q_{2}}}{(\partial t)^{q_{2}}} m_{j}^{(1)}(x,t)$ is continuous for $(x,t)\in \R\times [0,T]$ and 
\begin{align}
\bigg|\frac{\partial^{q_{1}}}{(\partial x)^{q_{1}}}\frac{\partial^{q_{2}}}{(\partial t)^{q_{2}}} m_{j}^{(1)}(x,t)\bigg| & \leq  \frac{1}{M^{j-1}} \int_{1}^{+\infty} |y^{q}f_{1}(\tfrac{i}{y})| \, dy  + \frac{j^{q}-1}{M^{j-q}}   \bigg(\int_{1}^{+\infty} |y^{q-1}f_{1}(\tfrac{i}{y})| \, dy \bigg)^{q} \label{lol19 der t and der xx bis}
\end{align}
for all $x \in \R$, $t\in [0,T]$.
\end{lemma}
\begin{proof}
For $j\in \N_{>0}$ and $k_{1},\ldots,k_{j}\in \{k:|k|\geq 1\}$,
\begin{align*}
\bigg|\sum_{p=1}^{j} \frac{\omega k_{p}-\frac{1}{\omega k_{p}}}{2} \bigg| \leq \sum_{p=1}^{j} |k_{p}|, \qquad \bigg|\sum_{p=1}^{j} \frac{\frac{1}{(\omega k_{p})^{2}}-(\omega k_{p})^{2}}{4} \bigg|  \leq \bigg( \sum_{p=1}^{j} |k_{p}| \bigg)^{2}.
\end{align*}
Thus, by \eqref{mjp1p der dert},
\begin{multline}
\bigg|\frac{\partial^{p_{1}}}{(\partial x)^{p_{1}}}\frac{\partial^{p_{2}}}{(\partial t)^{p_{2}}} m_{j}^{(1)}(x,t)\bigg| \leq  \int_{\tilde{\Gamma}}\dots \int_{\tilde{\Gamma}} |F^{(1)}(k_{1})| \sup_{s\in \tilde{\Gamma}} |F(s,k_{2})| \cdots \sup_{s\in \tilde{\Gamma}} |F(s,k_{j})| \\
 \times  \bigg( \sum_{p=1}^{j} |k_{p}| \bigg)^{q} |dk_1| \cdots |dk_j|, \qquad x\in \R, \; t\in [0,T). \label{lol23}
\end{multline}
Expending the sum in \eqref{lol23} and using \eqref{lol21}, \eqref{lol22} and the simple inequality
\begin{align*}
\int_{1}^{+\infty} e^{-\frac{T-t}{4}y^{2}}|y^{a}f_{1}(\tfrac{i}{y})| \, dy \leq \int_{1}^{+\infty} e^{-\frac{T-t}{4}y^{2}}|y^{b}f_{1}(\tfrac{i}{y})| \, dy, \qquad 0\leq a\leq b,
\end{align*}
we find \eqref{lol19 der t and der xx}. If $(\cdot)^{q} f_{1}(\frac{i}{\cdot}) \in L^{1}((1,+\infty))$, then \eqref{mjp1p der dert} implies that $\frac{\partial^{q_{1}}}{(\partial x)^{q_{1}}}\frac{\partial^{q_{2}}}{(\partial t)^{q_{2}}} m_{j}^{(1)}(x,t)$ is continuous for $(x,t)\in \R\times [0,T]$, and the same analysis as above shows that \eqref{lol23} (and therefore also \eqref{lol19 der t and der xx}) holds for all $x\in \R$ and $t\in [0,T]$. Since the right-hand side of \eqref{lol19 der t and der xx} is increasing for $t\in [0,T]$, \eqref{lol19 der t and der xx bis} follows by setting $t=T$ in the right-hand side of \eqref{lol19 der t and der xx}.
\end{proof}
Lemma \ref{lemma: bound on derivative of mjp1p derxx and dert} implies that the series \eqref{def of mp1p as a series} can be differentiated termwise any number of times with respect to $x$ and $t$, i.e.
\begin{align}\label{der of mp1p as a series der p1p2}
& \frac{\partial^{1+q_{1}}}{(\partial x)^{1+q_{1}}}\frac{\partial^{q_{2}}}{(\partial t)^{q_{2}}} m^{(1)}(x,t) =  \sum_{j=1}^{\infty} \frac{\partial^{1+q_{1}}}{(\partial x)^{1+q_{1}}}\frac{\partial^{q_{2}}}{(\partial t)^{q_{2}}}m_{j}^{(1)}(x,t), \qquad x\in \R, \; t\in [0,T),
\end{align}
for any $q_{1},q_{2}\in \N$. Using also \eqref{lol39} this yields
\begin{align}
& \frac{\partial^{q_{1}}}{(\partial x)^{q_{1}}}\frac{\partial^{q_{2}}}{(\partial t)^{q_{2}}} u(x,t)  = - i \sqrt{3} \sum_{j=2}^{\infty} \frac{\partial^{q_{1}+1}}{(\partial x)^{q_{1}+1}}\frac{\partial^{q_{2}}}{(\partial t)^{q_{2}}} m_{j}^{(1)}(x,t) -\frac{\sqrt{3}}{\pi} \int_{1}^{+\infty} e^{-\frac{T-t}{4} y^{2}} e^{-\frac{t}{4}\frac{1}{y^{2}}} \nonumber \\
& \times \bigg(\frac{y+\frac{1}{y}}{2}\bigg)^{q_{1}+1}\bigg(\frac{y^{2}-\frac{1}{y^{2}}}{4}\bigg)^{q_{2}}  \, \re \bigg[ (-i)^{q_{1}} e^{-\frac{x-x_{0}}{2}iy} e^{-\frac{x}{2}\frac{i}{y}} f_{1}(\tfrac{i}{y}) \bigg( \omega + \frac{\omega^{2}}{y^{2}} \bigg) \bigg] dy \label{lol24}
\end{align}
for $x\in \R$, $t \in [0,T)$. Let $q\in \N$ be fixed. If $\big(y\mapsto y^{q+1}f_{1}(\frac{i}{y}) \big)\in L^{1}((1,+\infty))$, then Lemma \ref{lemma: bound on derivative of mjp1p derxx and dert} also implies for $(x,t)\in \R\times[0,T]$ and $q_{1}+2q_{2}\leq q$ that
\begin{align}
\sum_{j=2}^{\infty} \bigg|  \frac{\partial^{1+q_{1}}}{(\partial x)^{1+q_{1}}}\frac{\partial^{q_{2}}}{(\partial t)^{q_{2}}}m_{j}^{(1)}(x,t) \bigg| & \leq \frac{1}{M-1} \int_{1}^{+\infty} |y^{q+1}f_{1}(\tfrac{i}{y})| \, dy \nonumber \\
&  + \sum_{j=2}^{\infty} \frac{j^{q+1}}{M^{j-q-1}}   \bigg(\int_{1}^{+\infty} |y^{q}f_{1}(\tfrac{i}{y})| \, dy \bigg)^{q+1}. \label{estimate for the remainder bis}
\end{align}
Thus the right-hand sides of \eqref{der of mp1p as a series der p1p2} and \eqref{lol24} are continuous for $\R\times [0,T]$ if $q_{1}+2q_{2}\leq q$, since they are uniformly convergent series of continuous functions on $\R\times [0,T]$. We have just proved the following.
\begin{theorem}\label{thm:inside the proof bis}
Suppose $(\cdot)^{q+1}f_{1}(\frac{i}{\cdot}) \in L^{1}((1,+\infty))$ for some $q\in \N$. Then $u$ is a Schwartz class solution of (\ref{boussinesq}) on $\R \times [0,T)$ such that 
\begin{align*}
\frac{\partial^{q_{1}}}{(\partial x)^{q_{1}}}\frac{\partial^{q_{2}}}{(\partial t)^{q_{2}}} u(x,t) \mbox{ is continuous on } \R \times [0,T] \mbox{ for each } q_{1},q_{2}\in \N  \mbox{ with } q_{1}+2q_{2}\leq q,
\end{align*}
and the estimate \eqref{estimate for the remainder bis} holds for all $(x,t)\in \R \times [0,T]$ and $q_{1},q_{2}\in \N$  such that $q_{1}+2q_{2}\leq q$.
\end{theorem}

The proof of the following lemma is similar to the proof of Lemma \ref{lemma:precise estimates for u} and is therefore omitted. Recall that $\ell:=\frac{2}{\sqrt{T-t}}$.
\begin{lemma}\label{lemma:precise estimates for u der}
Let $q,p\in \N$, $\vec{r} \in \N_{\mathrm{ord}}^{p}$, $\vec{a} \in \R_{+}^{p}$, $\sigma\in \{0,1\}$, and define $\mathrm{LOG}(y) = \mathrm{LOG}(y;\vec{r},\vec{a},\sigma)$ as in \eqref{def of LOG}. Suppose that $f_{1}$ is as in Proposition \ref{prop:n3 as a series via volterra}. Assume furthermore that $f_{1}(\frac{i}{y}) = y^{\eta} \mathrm{LOG}(y)$ for all $y\geq y_{0}$ for some $y_{0}>1$, and where $\eta\in (-q-3,-q-2)$. Let $q_{1},q_{2}\in \N$ be such that $q_{1}+2q_{2}=q+1$. As $t\to T$ and uniformly for $x\in \R$,
\begin{multline}
-\frac{\sqrt{3}}{\pi} \int_{1}^{+\infty} e^{-\frac{y^{2}}{\ell^{2}}} e^{-\frac{t}{4}\frac{1}{y^{2}}} \bigg(\frac{y+\frac{1}{y}}{2}\bigg)^{q_{1}+1}\bigg(\frac{y^{2}-\frac{1}{y^{2}}}{4}\bigg)^{q_{2}}   \re \bigg[ (-i)^{q_{1}} e^{-\frac{x}{2}\frac{i}{y}} f_{1}(\tfrac{i}{y}) \bigg( \omega + \frac{\omega^{2}}{y^{2}} \bigg) \bigg] dy \\
= -(1+o(1)) \frac{\sqrt{3}\re[(-i)^{q_{1}}\omega]}{2^{q+2}\pi}  \ell^{\eta+q+3} \mathrm{LOG}(\ell) \int_{0}^{+\infty} e^{-y^{2}} y^{\eta+q+2}dy \label{lol20 bis bis}
\end{multline}
and 
\begin{align}\label{lol34 der}
\sqrt{3} \sum_{j=2}^{\infty} \bigg|\frac{\partial^{1+q_{1}}}{(\partial x)^{1+q_{1}}}\frac{\partial^{q_{2}}}{(\partial t)^{q_{2}}} m_{j}^{(1)}(x,t)\bigg| \leq \frac{(1+o(1))\sqrt{3}}{M-1} \ell^{\eta+q+3} \mathrm{LOG}(\ell) \int_{0}^{+\infty} e^{-y^{2}} y^{\eta+q+2}dy.
\end{align}
Moreover, there exists $C>0$ such that, for $x\in \R\setminus \{x_{0}\}$ and $t\in [0, T)$,
\begin{align}
& \sqrt{3} \sum_{j=1}^{\infty} \bigg|\frac{\partial^{1+q_{1}}}{(\partial x)^{1+q_{1}}}\frac{\partial^{q_{2}}}{(\partial t)^{q_{2}}} m_{j}^{(1)}(x,t)\bigg| \leq C \frac{1+|x|}{|x-x_{0}|}. \label{lol36 der}
\end{align}
\end{lemma}
 
\begin{proof}[Proof of Theorem \ref{thm:wave breaking generalized}]
Let $q, p\in \N$, $\vec{r} \in \N_{\mathrm{ord}}^{p}$, $\vec{a} \in \R_{+}^{p}$, $\sigma\in \{0,1\}$, $T>0$, $x_{0}\in \R$ and $\delta\in (0,\frac{1}{2})$. Define $\mathrm{LOG}(s) = \mathrm{LOG}(s;\vec{r},\vec{a},\sigma)$ as in \eqref{def of LOG}, and let $\eta := -q-3+2\delta \in (-q-3,-q-2)$. Fix $M\geq 2$. Let $f_{1}:(0,i] \to \C$ be any function such that $f_{1} \in C^{\infty}((0,i))$, $f_{1}(k) = 0$ for $k\in [\tfrac{i}{2},i]$, and satisfying
\begin{align}
& \bigg\| \frac{f_{1}(\frac{i}{\cdot})}{\cdot} \bigg\|_{L^{1}((1,+\infty))} = \int_{1}^{+\infty} \frac{|f_{1}(\frac{i}{y})|}{y}dy \leq \frac{1}{M}, \label{L1 norm is less than 1 2 bis bis} \\
& f_{1}(\tfrac{i}{y}) = y^{\eta} \mathrm{LOG}(y) \mbox{ for all } y>y_{0}, \nonumber
\end{align}
for some $y_{0}>1$. Define $r_1:\hat{\Gamma}_{1} \to \C$ by
\begin{align}
& r_{1}(k) = 0, & & \mbox{for } k \in [-i,-i\infty) \cup \partial \D, \nonumber \\
& r_{1}(k) = f_{1}(k) e^{-\frac{x_{0}}{2}\frac{1}{k}} e^{\frac{T}{4k^{2}}}, & & \mbox{for } k \in (0,i], \nonumber
\end{align}
and define $r_{2}:\hat{\Gamma}_{4}\setminus \{\omega^{2}, -\omega^{2}\}\to \C$ by
\begin{align*}
r_{2}(k) = \tilde{r}(k) \overline{r_{1}(\bar{k}^{-1})} & & \mbox{for all } k \in \hat{\Gamma}_{4}\setminus \{\omega^{2}, -\omega^{2}\}.
\end{align*}
Define $m$, $m^{(1)}$ and $u$ as in \eqref{def of n3}, \eqref{def of mp1p}, and \eqref{recoveruvn 3}, respectively. Lemma \ref{lemma:n in terms of m} combined with Theorem \ref{inverseth new} implies that $u(x,t)$ is a Schwartz class solution of (\ref{boussinesq}) on $\R \times [0,T)$. Since $u(x,t) := -i\sqrt{3}\frac{\partial}{\partial x}m^{(1)}(x,t)$, Lemma \ref{lemma: bound on derivative of mjp1p derxx and dert} implies that
\begin{align*}
& \sup_{(x,t)\in \R\times [0,T)}\bigg| \frac{\partial^{q_{1}}}{(\partial x)^{q_{1}}}\frac{\partial^{q_{2}}}{(\partial t)^{q_{2}}} u(x,t) \bigg|  \leq \sqrt{3} \sum_{j=1}^{\infty} \sup_{(x,t)\in \R\times [0,T)} \bigg|\frac{\partial^{1+q_{1}}}{(\partial x)^{1+q_{1}}}\frac{\partial^{q_{2}}}{(\partial t)^{q_{2}}} m_{j}^{(1)}(x,t)\bigg| \\
& \leq \frac{\sqrt{3} M}{M-1} \int_{1}^{+\infty} |y^{q+1}f_{1}(\tfrac{i}{y})| \, dy + \sum_{j=2}^{\infty} \frac{j^{q+1}}{M^{j-q-1}}   \bigg(\int_{1}^{+\infty} |y^{q}f_{1}(\tfrac{i}{y})| \, dy \bigg)^{q+1} < +\infty
\end{align*}
holds for all $q_{1},q_{2}\in \N$ such that $q_{1}+2q_{2}\leq q$. Also, by Theorem \ref{thm:inside the proof bis},
\begin{align*}
\frac{\partial^{q_{1}}}{(\partial x)^{q_{1}}}\frac{\partial^{q_{2}}}{(\partial t)^{q_{2}}} u(x,t) \mbox{ is continuous on } \R \times [0,T] \mbox{ for each } q_{1},q_{2}\in \N  \mbox{ with } q_{1}+2q_{2}\leq q.
\end{align*}
This finishes the proof of Theorem \ref{thm:wave breaking generalized} (i). 

By \eqref{lol36 der}, there exists $C>0$ such that, for $x\in \R\setminus \{x_{0}\}$, $t\in [0, T)$ and any $q_{1},q_{2}\in \N_{>0}$ satisfying $q_{1}+2q_{2}=q+1$,
\begin{align*}
\bigg| \frac{\partial^{q_{1}}}{(\partial x)^{q_{1}}}\frac{\partial^{q_{2}}}{(\partial t)^{q_{2}}} u(x,t) \bigg| \leq 
\sqrt{3} \sum_{j=1}^{\infty} \bigg|\frac{\partial^{1+q_{1}}}{(\partial x)^{1+q_{1}}}\frac{\partial^{q_{2}}}{(\partial t)^{q_{2}}} m_{j}^{(1)}(x,t)\bigg| \leq C \frac{1+|x|}{|x-x_{0}|},
\end{align*}
which proves \eqref{bound for der der u in tlm gene}.  Moreover, using \eqref{lol24}, \eqref{lol20 bis bis} and \eqref{lol34 der}, we infer that for any $q_{1},q_{2}\in \N_{>0}$ satisfying $q_{1}+2q_{2}=q+1$,
\begin{align*}
\frac{\partial^{q_{1}}}{(\partial x)^{q_{1}}}\frac{\partial^{q_{2}}}{(\partial t)^{q_{2}}} u(x_{0},t) = -(1+o(1)) \frac{\sqrt{3}\re[(-i)^{q_{1}}\omega]}{2^{q+2}\pi}  \ell^{\eta+q+3} \mathrm{LOG}(\ell) \int_{0}^{+\infty} e^{-y^{2}} y^{\eta+q+2}dy + \mathcal{R}, 
\end{align*}
as $t\to T$, where $\mathcal{R}$ satisfies
\begin{align*}
|\mathcal{R}| \leq \frac{(1+o(1))\sqrt{3}}{M-1}  \ell^{\eta+q+3} \mathrm{LOG}(\ell) \int_{0}^{+\infty} e^{-y^{2}} y^{\eta+q+2}dy, \qquad \mbox{as } t \to T.
\end{align*}
Choosing $M$ such that $\frac{\sqrt{3}}{M-1} < \frac{\sqrt{3}}{2^{q+3}\pi}$ ensures that \eqref{asymp for u in thm gene} holds (recall that $\mathrm{LOG}(\ell)\asymp \mathrm{LOG}(\ell^{2}/4)$ as $t\to T$). This finishes the proof of Theorem \ref{thm:wave breaking generalized}.
\end{proof}

%

\subsection*{Data availability statement.} There is no data associated to this work.

\subsection*{Conflict of interest statement.} There is no conflict of interest.

\subsection*{Acknowledgements}
Support is acknowledged from the Swedish Research Council, Grant No. 2021-04626.

\bibliographystyle{plain}

\begin{thebibliography}{99}
\scriptsize

\bibitem{B1976}
J. G. Berryman, Stability of solitary waves in shallow water, {\it Phys. Fluids} {\bf 19} (1976), 771--777.

\bibitem{BZ2002}
L. V. Bogdanov and V. E. Zakharov, The Boussinesq equation revisited, {\it Phys. D} {\bf 165} (2002), 137--162.

\bibitem{BJM2018} 
M. Borghese, R. Jenkins and K.D.T.-R. McLaughlin, Long time asymptotic behavior of the focusing nonlinear Schr\"{o}dinger equation, {\it Ann. Inst. H. Poincar\'{e} C Anal. Non Lin\'{e}aire} \textbf{35} (2018), 887--920.

\bibitem{B1872}
J. Boussinesq, Th\'eorie des ondes et des remous qui se propagent le long d'un canal rectangulaire horizontal, en communiquant au liquide contenu dans ce canal des vitesses sensiblement pareilles de la surface au fond, {\it J. Math. Pures Appl.} {\bf 17} (1872), 55--108. 

\bibitem{BKST2009}
A. Boutet de Monvel, A. Kostenko, D. Shepelsky, and G. Teschl, Long-time asymptotics for the Camassa-Holm equation, \textit{SIAM J. Math. Anal.} \textbf{41} (2009), no. 4, 1559--1588

\bibitem{CLmain} 
C. Charlier and J. Lenells, On Boussinesq's equation for water waves, arXiv:2204.02365.

\bibitem{CL V} C. Charlier and J. Lenells, Boussinesq’s equation for water waves: asymptotics in sector V, arXiv:2301.10669 (to appear in SIAM Journal on Mathematical Analysis).

\bibitem{CL IV} C. Charlier and J. Lenells, Boussinesq's equation for water waves: the soliton resolution conjecture for Sector IV, arXiv:2303.00434.

\bibitem{CL I} C. Charlier and J. Lenells, Boussinesq's equation for water waves: asymptotics in Sector I, arXiv:2303.01232.



\bibitem{CLsolitonResolution} C. Charlier and J. Lenells, The soliton resolution conjecture for the Boussinesq equation, arXiv:2303.10485.

\bibitem{CLWasymptotics}
C. Charlier, J. Lenells and D. Wang, The ``good'' Boussinesq equation: long-time asymptotics, \textit{Anal. PDE} \textbf{16} (2023), no.6, 1351--1388. 
 
\bibitem{CF2022}
Q. Cheng and E. Fan, Long-time asymptotics for the focusing Fokas-Lenells equation in the solitonic region of space-time, \textit{J. Differential Equations} \textbf{309} (2022), 883--948.

\bibitem{CE1998A} A. Constantin and J. Escher, Global existence and blow-up for a shallow water equation, \textit{Ann. Scuola Norm. Sup. Pisa Cl. Sci.} (4) \textbf{26} (1998), no.2, 303--328.

\bibitem{CE1998} A. Constantin and J. Escher, Wave breaking for nonlinear nonlocal shallow water equations, \textit{Acta Math.} \textbf{181} (1998), no.2, 229--243.

\bibitem{CE2000} A. Constantin and J. Escher, On the blow-up rate and the blow-up set of breaking waves for a shallow water equation, \textit{Math. Z.} \textbf{233} (2000), no.1, 75--91.

\bibitem{DZ1993}
P. Deift and X. Zhou, A steepest descent method for oscillatory Riemann-Hilbert problems. Asymptotics for the MKdV equation, 
{\it Ann. of Math.} {\bf 137} (1993), 295--368.

\bibitem{GLL2024} X. Geng, W. Liu and R. Li, Long-time asymptotics for the coupled complex short-pulse equation with decaying initial data, \textit{J. Differential Equations} \textbf{386} (2024), 113--163.

\bibitem{GM2020}
T. Grava and A. Minakov, On the long-time asymptotic behavior of the modified Korteweg--de Vries equation with step-like initial data, \textit{SIAM J. Math. Anal.} \textbf{52} (2020), 5892--5993.

\bibitem{GT2009}
K. Grunert and G. Teschl, Long-time asymptotics for the Korteweg--de Vries equation via nonlinear steepest descent, {\it Math. Phys. Anal. Geom.} {\bf 12} (2009), 287--324.

\bibitem{GY2010} C. Guanand and Z. Yin, Global existence and blow-up phenomena for an integrable two-component Camassa-Holm shallow water system, \textit{J. Differential Equations} \textbf{248} (2010), no.8, 2003--2014.

\bibitem{H1973}
R. Hirota, Exact $N$-soliton solutions of the wave equation of long waves in shallow-water and in nonlinear lattices, {\it J. Math. Phys.} {\bf 14} (1973), 810--814. 

\bibitem{HWZ2024} L. Huang, D.-S. Wang and X. Zhu, Long-time asymptotics of the Tzitz\'{e}ica equation on the line, arXiv:2404.04999.

\bibitem{HXF2015}
L. Huang, J. Xu, and E. Fan, Long-time asymptotic for the Hirota equation via nonlinear steepest descent method, {\it Nonlinear Anal. Real World Appl.} {\bf 26} (2015), 229--262.

\bibitem{JN2012} Z. Jiang and L. Ni, Blow-up phenomenon for the integrable Novikov equation, \textit{J. Math. Anal. Appl.} \textbf{385} (2012), no.1, 551--558.

\bibitem{J1997}
R.S. Johnson, A modern introduction to the mathematical theory of water waves. Cambridge Texts in Applied Mathematics. Cambridge University Press, Cambridge, 1997.

\bibitem{KL1977}
V.K. Kalantarov and O.A. Ladyzenskaja, Formation of collapses in quasilinear equations of parabolic and hyperbolic types. (Russian) Boundary value problems of mathematical physics and related questions in the theory of functions, 10. {\it Zap. Nau\v{c}n. Sem. Leningrad. Otdel. Mat. Inst. Steklov (LOMI)} {\bf 69} (1977), 77--102.

\bibitem{LS1985}
H.A. Levine and B.D. Sleeman, A note on the nonexistence of global solutions of initial-boundary value problems for the Boussinesq equation $u_{tt} =3u_{xxxx}+u_{xx}-12(u^2)_{xx}$, {\it J. Math. Anal. Appl.} {\bf 107} (1985), 206--210. 

\bibitem{LGWW2019}
N. Liu, B. Guo, D. Wang and Y. Wang, Long-time asymptotic behavior for an extended modified Kortweg--De Vries equation, {\it Commun. Math. Sci.} {\bf 17} (2019), 1877--1913.

\bibitem{LY2006} Y. Liu and Z. Yin, Global existence and blow-up phenomena for the Degasperis-Procesi equation, \textit{Comm. Math. Phys.} \textbf{267} (2006), no.3, 801--820.

\bibitem{M1978}
V.G. Makhankov, Dynamics of classical solitons (in nonintegrable systems), {\it Phys. Rep.} {\bf 35} (1978), 1--128.

\bibitem{MY2023} Z. Meng and Z. Yin, Blow-up phenomena and the local well-posedness and ill-posedness of the generalized Camassa-Holm equation in critical Besov spaces, \textit{Monatsh. Math.} \textbf{200} (2023), no.4, 933--954.

\bibitem{RS2019}
Y. Rybalko and D. Shepelsky, Long-time asymptotics for the integrable nonlocal nonlinear Schr\"{o}dinger equation, {\it J. Math. Phys.} {\bf 60} (2019), 031504.

\bibitem{T1975}
M. Toda, Studies of a non-linear lattice, {\it Phys. Rep.} {\bf 18C} (1975), 1--123. 

\bibitem{WY2023} Z. Wang and K. Yan, Blow-up data for a two-component Camassa-Holm system with high order nonlinearity, \textit{J. Differential Equations} \textbf{358} (2023), 256--294.

\bibitem{WZ2006} S. Wu and Z. Yin, Blow-up, blow-up rate and decay of the solution of the weakly dissipative Camassa-Holm equation, \textit{J. Math. Phys.} \textbf{47} (2006), no.1, 013504, 12 pp.

\bibitem{XF2020}
J. Xu and E. Fan, Long-time asymptotic behavior for the complex short pulse equation, \textit{J. Differential Equations} \textbf{269} (2020), 10322--10349.

\bibitem{YLZ2012} W. Yan, Y. Li and Y. Zhang, Global existence and blow-up phenomena for the weakly dissipative Novikov equation, \textit{Nonlinear Anal.} \textbf{75} (2012), no.4, 2464--2473.

\bibitem{Y2002}
Z. Yang, On local existence of solutions of initial boundary value problems for the ``bad'' Boussinesq-type equation, {\it Nonlinear Anal.} {\bf 51} (2002), 1259--1271.

\bibitem{YW2003}
Z. Yang and X. Wang, Blowup of solutions for the ``bad" Boussinesq-type equation, \textit{J. Math. Anal. Appl.} \textbf{285} (2003), no.1, 282--298.

\bibitem{YF2023}
Y. Yang and E. Fan, Soliton resolution and large time behavior of solutions to the Cauchy problem for the Novikov equation with a nonzero background, \textit{Adv. Math.} \textbf{426} (2023), Paper No. 109088.

\bibitem{YC2024} S. Yang and J. Chen, On the finite time blow-up for the high-order Camassa-Holm-Fokas-Olver-Rosenau-Qiao equations, \textit{J. Differential Equations} \textbf{379} (2024), 829--861.

\bibitem{Z1974}
V.E. Zakharov, On stochastization of one-dimensional chains of nonlinear oscillations, {\it Soviet Phys. JETP} {\bf 38} (1974), 108--110.

\bibitem{ZW2024} X. Zhao and L. Wang, A two-component Sasa-Satsuma equation: large-time asymptotics on the line, \textit{J. Nonlinear Sci.} \textbf{34} (2024), no.2, Paper No. 38, 45 pp.

\bibitem{Z2004} Y. Zhou, Blow-up phenomenon for the integrable Degasperis-Procesi equation, \textit{Phys. Lett. A} \textbf{328} (2004), no.2-3, 157--162.
\end{thebibliography}

\end{document}